\documentclass{amsart}
\usepackage[english]{babel}
\usepackage{amsfonts,amsmath,amsxtra,amsthm,tikz,bbm}
\usepackage{tikz-cd}
\usetikzlibrary{arrows} 
\usepackage{mathrsfs}
\usepackage{times}
\usepackage{anysize}
\usepackage{graphicx}
\usepackage{youngtab}
\usepackage[utf8]{inputenc}
\usepackage{color}
\usepackage{hyperref}
\usepackage{subfigure}
\usepackage{stmaryrd, MnSymbol}
\usepackage{enumerate}

\marginsize{1in}{1in}{1in}{1in}
\newtheorem{lemma}{Lemma}[section]
\newtheorem{theorem}[lemma]{Theorem}%[section]
\newtheorem{conjecture}[lemma]{Conjecture}%[section]
\newtheorem{corollary}[lemma]{Corollary}%[section]
\newtheorem{proposition}[lemma]{Proposition}%[section]
\theoremstyle{definition}
\newtheorem{example}[lemma]{Example}%[section]
\newtheorem{remark}[lemma]{Remark}%[section]
\newtheorem{definition}[lemma]{Definition}%[section]

\def\bar{\overline}
\def\tilde{\widetilde}
\def\hat{\widehat}
\def\Sp{{\mathrm{Sp}}}
\def\SL{{\mathrm{SL}}}
\def\GL{{\mathrm{GL}}}
\def\PGL{{\mathrm{PGL}}}

\def\tSp{{\widetilde{\mathrm{Sp}}}}
\def\tG{{\underline{G}}}

\def\Sym{{\mathrm{Sym}}}
\DeclareMathOperator{\Hilb}{Hilb}
\DeclareMathOperator{\Spec}{Spec}
\DeclareMathOperator{\Ad}{Ad}
\DeclareMathOperator{\pt}{\mathrm{pt}}
\DeclareMathOperator{\Proj}{\mathrm{Proj}}

\DeclareMathOperator{\id}{\mathrm{id}}

\DeclareMathOperator{\Stab}{\mathrm{Stab}}

\DeclareMathOperator{\Lie}{Lie}
\DeclareMathOperator{\Hom}{Hom}
\DeclareMathOperator{\Fl}{Fl}
\DeclareMathOperator{\Gr}{Gr}
\DeclareMathOperator{\Coh}{Coh}

\DeclareMathOperator{\Diff}{Diff}
\DeclareMathOperator{\FT}{FT}

\DeclareMathOperator{\diag}{\mathrm{diag}}
\DeclareMathOperator{\gr}{gr}
\DeclareMathOperator{\HHH}{\mathrm{HHH}}

\DeclareMathOperator{\End}{End}
\newcommand{\kk}{\mathbf{k}}
\newcommand{\eee}{\mathbf{e}}

\newcommand{\Comm}{\mathfrak{C}}
\newcommand{\triv}{\mathrm{triv}}
\newcommand{\calO}{\mathcal{O}}
\DeclareMathOperator{\red}{\mathrm{red}}
\newcommand{\tW}{\widetilde{W}}
\newcommand{\tcA}{\widetilde{\mathcal{A}}}

\newcommand\bfe{\mathbf{e}}
\newcommand{\cO}{\mathcal{O}}
\newcommand{\cB}{\mathcal{B}}

\newcommand{\cP}{\mathcal{P}}

\newcommand{\cF}{\mathcal{F}}
\newcommand{\cG}{\mathcal{G}}

\newcommand{\cN}{\mathcal{N}}
\newcommand{\cL}{\mathcal{L}}
\newcommand{\cR}{\mathcal{R}}
\newcommand{\tcR}{\widetilde{\mathcal{R}}}
\newcommand{\cT}{\mathcal{T}}
\newcommand{\cK}{\mathcal{K}}

\newcommand{\cA}{\mathcal{A}}
\newcommand{\fh}{\mathfrak{h}}

\newcommand{\fb}{\mathfrak{b}}

\newcommand{\fm}{{\mathfrak m}}
\newcommand{\fn}{{\mathfrak n}}
\newcommand{\fg}{\mathfrak{g}}
\newcommand{\fri}{{\mathfrak i}}

\newcommand{\fP}{{\mathfrak P}}
\newcommand{\ft}{\mathfrak{t}}

\newcommand{\bI}{\mathbf{I}}
\newcommand{\bP}{\mathbf{P}}
\newcommand{\bF}{\mathbf{F}}

\newcommand{\tbP}{{\underline{\mathbf{P}}}}

\newcommand{\tbI}{{\underline{\mathbf{I}}}}

\newcommand{\Br}{\mathfrak{B}r}
\newcommand{\C}{{\mathbb C}}
\newcommand{\G}{{\mathbb G}}

\newcommand{\A}{\mathbb{A}}
\newcommand{\Z}{{\mathbb Z}}
\newcommand{\Q}{{\mathbb Q}}
\newcommand{\F}{{\mathbb F}}
\newcommand{\HH}{\mathbb{H}}

\renewcommand{\O}{\mathbb{O}}
\renewcommand{\P}{\mathbb{P}}
\newcommand{\tX}{\widetilde{X}}
\newcommand{\tfC}{\widetilde{\mathfrak{C}}}
\DeclareMathOperator{\codim}{\mathrm{codim}}

\newcommand{\Alt}{\mathrm{Alt}}
\newcommand{\sort}{\mathrm{sort}}

\newcommand{\yy}{\mathbf{y}}

\newcommand{\rat}{\mathrm{rat}}
\newcommand{\trig}{\mathrm{trig}}
\newcommand{\QCoh}{\mathrm{QCoh}}
\newcommand{\ad}{\mathrm{ad}}
\newcommand{\reg}{\mathrm{reg}}
\newcommand{\rot}{\mathrm{rot}}
\newcommand{\dil}{\mathrm{dil}}
\newcommand{\ev}{\mathrm{ev}}
\newcommand{\Ker}{\mathrm{Ker}}

\newcommand{\h}{\mathrm{cot}}

\newcommand{\Sch}{\mathrm{Sch}}

\begin{document}
\title{The affine Springer fiber -- sheaf correspondence}
\author{Eugene Gorsky}
\address{Department of Mathematics, UC Davis}
\email{egorskiy@math.ucdavis.edu}
\author{Oscar Kivinen}
\address{Department of Mathematics and Systems Analysis, Aalto University}
\email{oscar.kivinen@aalto.fi}
\author{Alexei Oblomkov}
\address{Department of Mathematics, UMass Amherst}
\email{oblomkov@math.umass.edu}
\date{\today}
\begin{abstract}
Given a semisimple element in the loop Lie algebra of a reductive group, we construct a quasi-coherent sheaf on a partial resolution of the trigonometric commuting variety of the Langlands dual group. The construction uses affine Springer theory and can be thought of as an incarnation of 3d mirror symmetry. For the group $GL_n$, the corresponding partial resolution is $\Hilb^n(\C^\times\times \C)$. We also consider a quantization of this construction for homogeneous elements.
\end{abstract}
\maketitle
\tableofcontents
\section{Introduction}
\label{sec:introduction}

Let $\kk=\bar{\kk}$ be an algebraically closed field of characteristic zero or large positive characteristic. We fix a connected 
reductive group $G/\kk$ and a maximal torus $T\subset G$. We let \(\Lie(G)=:\fg\supset
\ft:=\Lie(T)\), and denote by $G^{\vee}$ the Langlands dual group over $\C$ (or $\overline{\Q}_\ell$).

In this paper we explain how one can naturally associate to an affine Springer fiber for $G$, or rather to a conjugacy class of the loop Lie algebra $\fg(\!(t)\!)$, a quasi-coherent sheaf on a certain partial resolution of the commuting variety associated to $G^{\vee}\times \fg^{\vee}$. We prove that the sheaves constructed this way are coherent in a number of cases and conjecture they are coherent in general. The sheaf on the partial resolution remembers homological invariants of the affine Springer fiber, and  we expect that our construction provides a unified perspective on the affine Springer fibers and their various functorial properties.

\subsection{Main results}

Our first main character is a partial resolution of the trigonometric version of the commuting variety for $G^{\vee}$, which we denote by 
\(\widetilde{\Comm}_{G^{\vee}}=
\widetilde{T^*T^\vee/W}\). It is in general a singular Poisson variety which we conjecture to locally agree with those constructed in \cite{BALK,LosevDeformations,Namikawa}. The variety $\widetilde{\Comm}_{G^{\vee}}$ is defined  via its homogeneous coordinate ring defined below in equation \eqref{eq: def resolution intro}.

For $G=\GL_n$, 
we recover the Hilbert scheme of $n$ points on $\C^{\times}\times \C$ \cite{Nak}. The variety $\widetilde{\Comm}_{G^{\vee}}$ has a natural action of the torus $\C^{\times}_{\h}$ which lifts the torus action on $T^*T^\vee/W$ dilating the cotangent fibers. A more detailed construction is given in Section \ref{sec:commuting}. 

Let $\Gr_G$  be the affine Grassmannian of $G$. On the level of $\kk$-points this is $G(\cK)/G(\cO)$, where \(\calO=\kk\llbracket t\rrbracket\) and \(\cK=\kk(\!( t)\!)=\kk\llbracket t\rrbracket[t^{-1}]\). Our second main character is the {\em affine Springer 
fiber} \(\Sp_\gamma\subset \Gr_G\), \(\gamma\in \fg(\cK)\), defined as the fixed locus of the vector field 
\(\gamma\). More precisely, let $\gamma\in 
\fg(\cK)$ be a compact semisimple element. On the level of $\kk$-points $$\Sp_\gamma(\kk):=\left\{gG({\kk\llbracket t\rrbracket})|\Ad(g^{-1})\gamma\in \Lie(G)({\kk\llbracket t\rrbracket})\right\}\subset \Gr_G(\kk).$$ Under these assumptions, $\Sp_\gamma$ is a nonempty ind-scheme over $\kk$. If $\gamma$ is also regular, $\Sp_\gamma$ is finite-dimensional and locally of finite type over $\kk$ \cite[2.5.2]{YunLectures}. We will only be interested in the \'etale or singular cohomologies of the $\Sp_\gamma$, so will be writing $\Sp_\gamma$ for the $\kk$-points $\Sp_\gamma(\kk)$. If $\kk=\C$ we will use the analytic topology and if $\kk=\overline{\F}_q$ we will use the \'etale topology. Our main result is the following.

\begin{theorem}
\label{thm:mainthm}
Let $\gamma\in \fg(\cK)$ be a semisimple element and $G_\gamma=C_{G(\cK)}(\gamma)$ its centralizer in $G(\cK)$. The group $G_\gamma$ admits a N\'eron model $J_\gamma/\cO$. Then for every subgroup $K_\gamma\subseteq J_\gamma(\cO)$, there exists a quasi-coherent sheaf \(\cF_{\gamma}^{K_{\gamma}}\in \QCoh_{\G_m}
(\widetilde{\Comm}_{G^{\vee}})\)
such that:
\begin{itemize}
\item[(1)]   $\cF_{t\gamma}^{K_{\gamma}}=\cL\otimes 
\cF_{\gamma}^{K_{\gamma}}$ where $\cL=\cO(1)$ is the Serre twisting sheaf coming from the $\Proj$-construction of $\widetilde{\Comm}_{G^{\vee}}$ \eqref{eq: def resolution intro}.
\item[(2)] There exists an integer $M$ such that for $m>M$ we have
\[H^0\left(\widetilde{\Comm}_{G^{\vee}},\cF_{t^m\gamma}^{K_{\gamma}}\right)=H_*^{K_\gamma}(\Sp_{t^m\gamma})\]
\end{itemize}

 Moreover, the sheaf $\cF_{\gamma}^{K_{\gamma}}$ is equivariant with respect to the action of $\C^{\times}_{\h}$ and 
 the homological grading on the affine Springer fiber side can be recovered from this  action.

\end{theorem}
We conjecture below that $\cF_{\gamma}^{K_{\gamma}}$ is actually coherent (see Conjecture \ref{conj: coherent}) and prove it in some cases.  For trivial subgroup $K_{\gamma}$ we denote $\cF_{\gamma}^{K_{\gamma}}$ by $\cF_{\gamma}$.
\begin{remark}
Note that the quasicoherent sheaf $\cF_\gamma^{K_{\gamma}}$ only depends on the underlying reduced structure of the ind-scheme 
$\Sp_\gamma$.\end{remark}
\begin{remark}
Here 
and in the rest of the paper, $H_*(-)$ denotes Borel-Moore homology, defined as $H_*(X):=H^{-*}(X,\omega_X)$. The Borel-Moore homology of an ind-variety $X=\varinjlim X_i$ will be defined as $H_*(X)=\varinjlim H_*(X_i)$.
\end{remark}
It is natural to wonder what kind of sheaves $\cF_\gamma$ Theorem \ref{thm:mainthm} yields. For $G=\GL_n$ and $\gamma$ {\em homogeneous}, we have the following. 
\begin{theorem}[Proposition \ref{prop:punctualexample}]
\label{thm: intro homogeneous 1}
If $G=\GL_n$, $K_{\gamma}$ is trivial and $\gamma$ elliptic homogeneous of slope  $\frac{kn+1}{n}$ then $\cF_{\gamma}$ agrees with the restriction of the line bundle $\cO(k)$ to the punctual Hilbert scheme at $(1,0)$.

\end{theorem}
\begin{theorem}[Proposition \ref{prop:procesi}]
\label{thm: intro homogeneous 0}
For $G=\GL_n, K_{\gamma}=T$ and $\gamma$ regular semisimple of integral slope $k$, the sheaf $\cF_\gamma^{T}$ agrees with $\cP\otimes \cO(k)$, where $\cP$ is the Procesi bundle restricted to $\Hilb^n(\C^\times\times \C)$.

For the same $\gamma$ and $K_{\gamma}$  trivial we get a sheaf $\cF_\gamma=\cP\otimes \cO(k)/(y_1,\ldots,y_n)\cP\otimes \cO(k)$ where $y_i$ are interpreted as endomorphisms of $\cP$.
\end{theorem}

\begin{remark}
For arbitrary $G$ and homogeneous $\gamma$ of integral slope, we also get an explicit description of the sheaf $\cF_{\gamma}$, see Theorem \ref{thm: integral slope any type}.
\end{remark}

We also prove a noncommutative version of the above results. The ring of functions on $\widetilde{\Comm}_{G^{\vee}}$ admits a deformation, or quantization known as the spherical {\em trigonometric} or {\em graded} Cherednik algebra (or graded DAHA). By the work of Yun \cite{Yun1,YunSph}, this algebra acts in homology of $\Sp_{\gamma}$. 

The sheaf $\cL=\cO(1)$ and its powers are quantized to bimodules between two trigonometric Cherednik algebras with different values of quantization parameters. These algebras and modules are assembled to a large $\Z$-algebra \cite{BP}, and one of the main results of the paper (Theorem \ref{thm:action}) associates a graded module over this $\Z$-algebra to a collection of affine Springer fibers $\{\Sp_{\gamma},\Sp_{t\gamma},\Sp_{t^2\gamma},\ldots\}$. Roughly speaking, considering the bimodules between different algebras allows one to move between different affine Springer fibers.
For $G=\GL_n$, the $\Z$-algebra is similar to the one considered by Gordon and Stafford for rational Cherednik algebras of $\GL_n$ in \cite{GS1, GS2}. 

The main tool we use is a novel construction of $\Z$-algebras related to the Coulomb branches of $3d\ \cN=4$ theories. The study of the latter was mathematically initiated by Braverman, Finkelberg, Nakajima \cite{BFN,BFN2} and further explored by for example Webster \cite{Webster}. The original Coulomb branch algebra from \cite{BFN} associated to $G$ and an algebraic representation $N\in \text{Rep}(G)$ is defined as the convolution algebra in the equivariant Borel-Moore homology of a certain "space of triples" $\cR_{G,N}$, which is modeled after the affine Grassmannian of $G$. This construction admits a natural quantization by considering additional equivariant parameters, and one can study both commutative and quantized versions. 

In addition to the quantizations and the $\Z$-algebras, there are several other generalizations of the original construction, such as the line operators in \cite{DGGH} which for example allow for different partial affine flag varieties. We use the machinery of Coulomb branches to achieve the following goals:
\begin{itemize}
\item We realize the full graded DAHA as the Iwahori version of the Coulomb branch algebra associated to the adjoint representation and construct its Dunkl-Cherednik embedding to $\hbar$-difference operators on the Lie algebra of the torus $T^\vee$ (Theorem \ref{thm:iwahoricoulombisdaha}).
\item We give explicit formulas for the Coulomb branch $\Z$-algebra in difference presentation (Theorem \ref{thm:gr of bruhat}, Proposition \ref{prop:noncomm localization for any group}).
\item We prove the shift isomorphisms for the spherical/anti-spherical subalgebras of the graded DAHA, in the difference-operator representation (Theorem \ref{thm:shiftiso}). This allows us to define the shift bimodules  and $\Z$-algebras associated to graded DAHA for arbitrary $G$.
\item In the commutative version, the Coulomb branch $\Z$-algebra is equivalent to a graded algebra which we identify explicitly (Theorem \ref{thm: coulomb is symbolic}). This allows us to define $\widetilde{\Comm}_{G^{\vee}}$ using the $\Proj$ construction.
\item Finally, we prove that a collection of affine Springer fibers $\{\Sp_{\gamma},\Sp_{t\gamma},\Sp_{t^2\gamma},\ldots\}$ yields a module over the Coulomb branch $\Z$-algebra (Corollary \ref{cor:springer modules}). This is done using a variant of the Springer theory developed by Hilburn, Kamnitzer and Weekes \cite{HKW}, and Garner and the second author \cite{GK}. 
\end{itemize}

We give a more detailed outline of the results and arguments in Section \ref{sec:outline}. We also comment on various conjectures and connections to physics of ``3d Mirror Symmetry" and link homology (for $G=\GL_n$). 

Of course, the technology of Coulomb branches as introduced in \cite{BFN} works in far greater generality than the "adjoint matter" case studied in this paper. In Sections \ref{sec:coulomb} and \ref{sec:bfnspringer} we give definitions for the general case but focus our study on the $(G,Ad)$-variant. Following these definitions, the associated $\Z$-algebras and their geometrically defined modules could be studied in much larger generality but as far as the authors are aware, this remains a fairly unexplored direction.

 \subsection{3d Mirror Symmetry}
Our main construction can be thought of as a part of the $3d$ mirror symmetry for topological twists of $3d \; \cN=4$ gauge theories \cite{IS}. The 3d mirror symmetry exchanges the algebras of local operators on a resolved Higgs branch and on a resolved Coulomb branch. 
In particular, it is known that the Coulomb branch of the B-twist of the $(G,\Ad)$-theory is a partial resolution of $T^*T^\vee/W$. 
In physics terms, our construction starts with a ''boundary condition" for the category of line operators in the A-twist of the $(G,\Ad)$-theory, thought of as a degenerate line operator (skyscraper sheaf on the stack of conjugacy classes in the loop Lie algebra) and produces another degenerate line operator in the $B$-twist of the dual theory (a (quasi-)coherent sheaf on a resolution of the Higgs branch).

More precisely, according to \cite[Eq. (1.4)]{DGGH},
the categories of line operators in the two twists are given by $$\mathcal{C}_A\cong \text{D-mod}_{G(\cK)}(\Ad(\cK)), \; \mathcal{C}_B\cong \QCoh(\text{Maps}(S^1_{dR},\Ad/G))$$ and by the subsequent discussion in {\em loc. cit.} $\mathcal{C}_B$ can be replaced by $\QCoh(\mathcal{M}_H)$, where $\mathcal{M}_H$ is a resolution of the Higgs branch.

Our construction is far from giving any sort of categorical equivalence between $\mathcal{C}_A, \mathcal{C}_B$, as even defining the categories involved is a delicate matter. However, supposing this done and writing $\delta_\gamma\in \mathcal{C}_A$ for $\gamma\in\fg(\cK)$ for the skyscraper sheaf on the conjugacy class of $\gamma$ as before, our construction gives an explicit ''mirror map" sending $\delta_\gamma\mapsto \cF_\gamma$, where $\cF_\gamma$ is as in Theorem \ref{thm:mainthm}. Doing this for some other line operators, such as the bimodules in Section \ref{sec: adjoint coulomb}, is also possible but we don't know how to make the construction functorial.

We hope this construction gives a starting point for rigorous constructions of 3d mirror symmetry for line operators. The fact that these categories have putative definitions in terms of vertex operator algebras \cite{CoCrGa} is an interesting topic for further investigations.
\subsection{Conjectures}
\label{sec:conjectures}

The main construction 
in Theorem \ref{thm:mainthm} produces a \(\C^{\times}\)-equivariant quasi-coherent sheaf  
$$\cF_\gamma\in \text{QCoh}_{\C^{\times}}\left(\widetilde{\Comm}_{G^{\vee}}\right)$$ for \(\gamma\in \fg(\cK)\)  (and trivial subgroup $K_{\gamma}$).
Quasi-coherence of the sheaf follows directly from our  construction, but we suspect that a stronger statement is in fact true:

\begin{conjecture}
\label{conj: coherent}
  For any regular semisimple \(\gamma\in \fg(\cK)\) the sheaf $\cF_\gamma$ is coherent:
  \[\cF_\gamma\in \Coh_{\C^{\times}}(\widetilde{\Comm}_{G^{\vee}}).\]
\end{conjecture}

\begin{remark}
Under the natural projection $\widetilde{\Comm}_{G^{\vee}}\to T^*T^{\vee}/W$, the sheaf $\cF_\gamma$  pushes forward to a certain sheaf on $T^*T^{\vee}/W$ which we describe in detail in Section \ref{sec: lattice action}. In Lemma \ref{lem:support} we show that its reduced support is contained in the Lagrangian subvariety $\{0\}\times T^{\vee}/W\subset T^*T^{\vee}/W$.

Similarly, we expect $\cF_{\gamma}$ to be supported on a certain Lagrangian subvariety of $\widetilde{\Comm}_{G^{\vee}}$ of dimension $r=\mathrm{rank}(\fg).$
\end{remark}

The coherence conjecture already has interesting numerical corollaries. It is known that \(H_*(\Sp_\gamma)\) is finite-dimensional if \(\gamma\) is elliptic
and $G$ is simply connected. Thus the conjecture above implies an estimate on the growth of the dimensions of these cohomology spaces, as we multiply $\gamma$ by $t^m$.

\begin{conjecture}
\label{conj:polygrowth}
  For any elliptic regular semisimple \(\gamma\in \fg(\cK)\) with \(G\) being simple and simply connected there exist \(c_i\in \Q\) and \(M\in \Z\) such that
  \[\dim(H_*(\Sp_{t^m\gamma})\!)=\sum_{i=0}^r c_i m^i, \quad m>M, \quad r=\mathrm{rank}(\fg).\]
\end{conjecture}

In the case of homogeneous elliptic \(\gamma\) it was shown in \cite{VVfd} that the conjecture is true, the left-hand side being given by variants of rational Coxeter-Catalan numbers, in particular the $c_i$ can be explicitly computed.\footnote{For $\tSp_{t^m\gamma}$ i.e. the version in affine flags, the result is easier to state and simply says \(c_i=0\), \(i<r\) for the homogeneous elliptic cases.} Many low-rank examples are treated in a lot of detail in \cite{OY}. For $G=GL_n$ and $\gamma$ of slope $\frac{kn+1}{n}$, the corresponding sheaf $\cF_{\gamma}$ is described in Theorem \ref{thm: intro homogeneous 1} above.
More complicated non-homogeneous elliptic cases for \(G=\SL_n\) were studied in \cite{GMO,KT,Pi}. 

In a different direction, we can consider the case of unramified $\gamma$ introduced in \cite{GKM}. The second author has computed in \cite{Kivinen} the equivariant cohomology of the affine Springer fibers in the unramified case and related them to the Procesi bundle on the Hilbert scheme of points.

More precisely, when $\gamma$ is equivalued of valuation $k\in \Z_{\geq 0}$, we describe explicitly the graded module corresponding to $\cF_\gamma$ for all $G$ in Theorem \ref{thm: integral slope any type}. For such $\gamma$ and $G=GL_n$ the sheaf $\cF_{\gamma}$ is indeed coherent and described by Theorem \ref{thm: intro homogeneous 0} above. 

After the first version of this paper appeared on arXiv, Turner proved in \cite[Theorem 1.9]{Turner} that for $G=GL_3$ and an arbitrary unramified $\gamma$ the sheaf $\cF_{\gamma}$ is coherent (and, in fact, the sheaf $\cF_{\gamma}^T$ is coherent). See \cite{Turner} for more details and an explicit conjectural description of the corresponding graded module for $G=GL_n$ and unramified $\gamma$.

\subsection{Relation to conjectures stemming from knot theory}

The case of \(G=\GL_n, \SL_n\) is of special interest because of the applications to knot theory \cite{GORS,GNR,GH,OS,ORS,OR}. In particular, the characteristic polynomial of a compact regular element  \(\gamma\in \mathfrak{gl}_n(\cK)\)  defines a  germ of a planar curve singularity and the link of this
singularity is the closure of the braid (conjugacy class) \(\beta(\gamma)\in \Br_n\). When $G=\SL_n$ and $\gamma$ is elliptic, $\beta(\gamma)$ closes to a knot and the conjecture \cite[Conjecture 2 and Proposition 4]{ORS} predicts an isomorphism between the (reduced) triply graded Khovanov-Rozansky homology of
\(\beta(\gamma)\)  and the cohomology of the affine Springer fiber \(H^*(\Sp_\gamma)\) enhanced with the perverse filtration \cite{MS,MY,MSV}.  Notice that these papers use {\em cohomology} whereas the present work uses {\em BM homology}, but this distinction is immaterial for numerical comparisons of multiply graded finite dimensional vector spaces with characteristic zero coefficients.

In this paper, we enrich the algebro-geometric side of the above  conjectures by considering an infinite family of affine Springer fibers $\{\Sp_{\gamma},\Sp_{t\gamma},\Sp_{t^2\gamma},\ldots\}$. It is easy to see that multiplication of $\gamma$ by $t^m$ corresponds to the multiplication of the braid $\beta(\gamma)$ by $\FT^m$, the $m$-th power of the full twist braid $\FT$. Since $\FT$ is central in the braid group, the conjugacy class of $\beta(\gamma)\cdot \FT^m$ is determined by the conjugacy class of $\beta(\gamma)$.

Khovanov and Rozansky defined in \cite{KR, KR2,KhSoergel} the triply graded link homology $\HHH(\beta)$ for an arbitrary braid $\beta$ and proved that it is a topological invariant of the link obtained by the closure of $\beta$. The three gradings on $\HHH(\beta)$ are  usually referred to as $a,q,t$. 
By construction of  triply graded homology $\HHH(-)$, there are natural grading-preserving multiplication maps
\begin{multline}
\label{eq: HHH algebra}
\HHH(\beta(\gamma))\otimes \HHH(\FT^m)\to \HHH(\beta(\gamma)\cdot \FT^m)\simeq  \HHH(\beta(t^m\gamma)),\\ \HHH(\FT^m)\otimes \HHH(\FT^{m'})\to \HHH(\FT^{m+m'}),
\end{multline}
and hence $\bigoplus_m \HHH(\beta(t^m\gamma))$ has a structure of a graded module over the graded algebra 
$\bigoplus_m \HHH(\FT^m)$. The latter graded algebra, as conjectured in \cite{GNR} and proved in \cite{GH}, is closely related   to the homogeneous coordinate ring of $\Hilb^n(\C^2)$, and to the $\Z$-algebras appearing in this paper. In other words, in this paper we establish a precise analogue of multiplication maps \eqref{eq: HHH algebra} on the affine Springer side by means of geometric representation theory.

In a series of papers \cite{OR,OR1} the third named author and Rozansky took a different approach to knot invariants and defined  a \(\C^{\times}\times \C^{\times}\)-equivariant complex of coherent sheaves \[\cG_\beta \in D^b_{\C^{\times}\times\C^{\times}}(\text{Coh}(\Hilb_n(\C^2)))\] such that
the hypercohomology \(H^*(\Hilb_n(\C^2),\cG_\beta)\) is isomorphic as a bigraded vector space to the ''lowest row"  $\HHH^{a=0}(\beta)$ of the triply-graded homology. The action of $\C^{\times}\times \C^{\times}$ corresponds to (after an appropriate normalization) the $(q,t)$-grading on $\HHH^{a=0}(\beta)$.

To connect these constructions with the present one, note that the natural inclusion map \(i_{\C^{\times}}: \C^{\times}\times \C\to \C^2\), \(i_{\C^{\times}}(x,y)=(x-1,y)\) induces an inclusion \(i_{\C^{\times}}: \Hilb^n(\C^{\times}\times \C)\to \Hilb^n(\C^2)\) which identifies the punctual Hilbert schemes at $(1,0)$ and $(0,0)$.

\begin{conjecture}
\label{conj:sheafiso}
  For any regular semisimple \(\gamma\in \mathfrak{gl}_n\llbracket t\rrbracket\) there is an isomorphism of \(\C^{\times}\)-equivariant sheaves
  \[\cF_\gamma\simeq i^*_{\C^{\times}}(\cG_\beta),\quad \beta=\beta(\gamma).\]
\end{conjecture}

Let us point out that the results of Maulik \cite{Maulik}, the aforementioned \cite{MY,MS,MSV}, and the results of the third author with Rozansky in \cite{OR1} can be combined to show that for elliptic $\gamma,$ the conjecture is true on the level of Euler characteristics:
\[\chi(\Sp_\gamma)=\chi_{\C^{\times}}(i^*_{\C^{\times}}(\cF_\beta)),\quad \beta=\beta(\gamma),\]
for \(\gamma\in \mathfrak{sl}_n\llbracket t\rrbracket\) elliptic regular semisimple. 

In particular, we derive an Euler characteristics version of the weak coherence conjecture \ref{conj:polygrowth}:
\begin{proposition}
  For any elliptic regular semisimple \(\gamma\in \mathfrak{sl}_n\llbracket t\rrbracket\)  there are \(c_i\in \Q\) and \(M\in \Z\) such that
  \[\chi(\Sp_{t^m\gamma})=\sum_{i=0}^{n-1} c_i m^i, \quad m>M.\]
\end{proposition}
\begin{remark}
In fact, the above Proposition also follows  from Definition 5.13. and Eq. (6.7) in \cite{KT} for $G=\SL_n$. Under the purity hypothesis, the stronger conjecture \ref{conj:polygrowth} would also follow in this case from the same results.
\end{remark}

Conjecture \ref{conj:sheafiso} implies for example that the coefficients of the HOMFLYPT polynomial of the closure of \(\beta(\gamma)\cdot \FT^m\) are polynomials of \(m\). This property of the HOMFLYPT polynomial could be derived, for example, from the results of
\cite{OR,OR1} and equivariant Riemann-Roch formula.

Finally, note that Maulik's result in \cite{Maulik} actually keeps track of the Euler characteristics of the Hilbert schemes of points on the germ defined by $\gamma$ and hence the perverse filtration \cite{MY,MS}. For elliptic $\gamma$ we may also conjecture that there exists a Springer-theoretic construction of a sheaf $\gr^P\cF_\gamma\in \text{Coh}_{\C^\times\times \C^\times}(\Hilb^n(\C^2))$  which $\C^\times\times\C^\times$-equivariantly agrees with $\cF_\beta$. For partial results in this direction, see \cite[Section 9]{KT}.

\subsection{Outline}
\label{sec:outline}
\subsubsection{Outline of the argument}
\label{sec:outline-argument}

The key ingredients of the construction are 1) the technology developed by Braverman, Finkelberg and Nakajima \cite{BFN,BFN2,BFN3} on the affine Springer side (Topology) and 2) noncommutative geometry methods akin to the work by \cite{GS1} on the Hilbert scheme side (Algebraic Geometry). The theory of the 3) Double affine Hecke algebras (Algebra) links these two theories together. Our work provides a dictionary between objects in the three theories.  A part of this dictionary is as follows\footnote{For simplicity of introduction we discuss the type \(A\) case, for other types see section~\ref{sec:finitegeneration}}:

$$\begin{array}{|c|c|c|}\hline
    \text{Topology} & \text{Algebra} & \text{Algebraic Geometry}\\ \hline
    {}_i \cA_i:=H_*^{\tG(\cO)\rtimes\C^\times}({}_i\cR_i) & \eee \HH_{c+i\hbar,\hbar} \eee & \C[T^*T^\vee/W] \\
    {}_{i-1} \cA_i:=H_*^{\tG(\cO)\rtimes\C^\times}({}_{i-1}\cR_i) & {}_{i-1}\cB_i= \eee_- \HH_{c+i\hbar,\hbar} \eee & \cO(1)\\
    {}_i\tcA_i:=H_*^{\underline{\bI}\rtimes\C^\times}({}_i\tcR_i) & \HH_{c+i\hbar,\hbar} & H^0\left(\widetilde{T^*T^\vee/W},\cP\right)\\
    {}_{i-1}\tcA_i:=H_*^{\underline{\bI}\rtimes\C^\times}({}_{i-1}\tcR_i) & {}_{i-1} \widetilde{\cB}_i & H^0\left(\widetilde{T^*T^\vee/W},\cP\otimes \cO(1)\right)\\
    H_*^{\C^\times}(\tSp_{\frac{nk+1}{n}}) & \HH_{c+k\hbar,\hbar} \curvearrowright L_{\frac{nk+1}{n}}(\C) & \cO_{\pi^{-1}(1,0)}(k)\\\hline
\end{array}$$

\noindent
In the Algebra column of the table we have the algebraic objects related to
the representation theory of the graded double affine Hecke algebra \(\HH_{c,\hbar}\)
defined in Definition \ref{def:dDAHA}. This algebra is also known under the names {\it trigonometric} or {\em degenerate} DAHA. The algebra  \(\HH_{c,\hbar}\) contains the finite Weyl group \(W\) and \(\eee,\eee_-\in W\)
are the projectors for the trivial and sign representations. We define an explicit representation of these algebras using difference operators and use it to prove the following:

\begin{theorem}(Theorem \ref{thm:shiftiso})
The {\it spherical subalgebra } \(\eee \HH_{c,\hbar}\eee\) is isomorphic to the {\it anti-spherical subalgebra } \(\eee_-\HH_{c-\hbar,\hbar}\eee_-\) with shifted parameter.
\end{theorem}

Similar shift isomorphisms are well known in the theory of rational Cherednik algebras and for the Dunkl differential-difference representation \cite{BEG, Heckman, Opdam}, but it appears to be new for the difference representation of trigonometric DAHA. See also \cite{WLiu} for similar results.

Thus \(\eee_-\HH_{c+i\hbar,\hbar}\eee\) naturally has left \(\eee\HH_{c+(i-1)\hbar,\hbar}\eee\) and right \(\eee\HH_{c+i\hbar,\hbar}\eee\) actions and
we set:
\[
{}_i\cB_i=\eee\HH_{c+i\hbar,\hbar}\eee,\quad {}_i\cB_{i+1}=\eee_-\HH_{c+i\hbar,\hbar}\eee,\quad {}_i\cB_{j+1}={}_i\cB_{j}\otimes_{{}_j\cB_j}{}_j\cB_{j+1}.
\]
Here ${}_i\cB_{i}$ are algebras, ${}_i\cB_{j}$ are bimodules over ${}_i\cB_{i}$ and ${}_j\cB_{j}$ for all $i\le j$, and we get well-defined multiplication maps
$$
{}_i\cB_{j}\bigotimes_{{}_j\cB_{j}} {}_j\cB_{k}\to {}_i\cB_{k},\ i\le j\le k.
$$

The direct sum \({}_\bullet\cB_\bullet=\bigoplus_{i\leq j}{}_i\cB_{j}\) is an example of a \(\Z\)-algebra, introduced by Polishchuk and Bondal \cite{BP} and studied in a setting relevant to us  by Gordon and Stafford \cite{GS1,GS2,Rains}.

We now explain how the above mentioned structures exist in the affine Springer theory, which corresponds to the Topology column. The key geometric object is an ind-scheme ${}_i\tilde{\cR}_j$, a variant of ''the space of triples $\cR$" central to the work of Braverman, Finkelberg and Nakajima \cite{BFN,BFN2,BFN3} on Coulomb branches:
\[{}_j\tilde{\cR}_i=\left\{(g,v)\in G(\mathcal{K})\times t^i\Lie(\bI)| g\cdot v\in t^j\Lie(\bI)\right\}/\bI,\]
where \(\bI\) is the Iwahori subgroup (see Section \ref{sec: asf definitions}).
On the level of sets, the quotient space \(\bI\backslash {}_0\tilde{\cR}_0\) is in bijection with the quotient \(\widetilde{St}/G(\mathcal{K})\) of the affine Steinberg space \(\widetilde{St}=\{(b_1,g,b_2)\in \Fl_G\times \fg(\cK)\times \Fl_G| g\in b_1\cap b_2\}\), as also explained in the introduction to \cite{BFN2}. There is an action of $\C^{\times}_{\dil}\times \C^{\times}_{\rot}$ on  $\widetilde{St}$ where $\C^{\times}_{\dil}$ acts by dilating the $\fg(\cK)$ factor and 
$\C^{\times}_{\rot}$ acts by loop rotation. We denote
\begin{equation}
\label{eq: extended groups}
\underline{\bI}=\bI\times \C^{\times}_{\dil},\ \underline{G(\cO)}=G(\cO)\times  \C^{\times}_{\dil},\ \underline{G(\cK)}=G(\cK)\times  \C^{\times}_{\dil}.
\end{equation}

It was explained by Varagnolo and Vasserot \cite{VV2} that under a certain specialization of parameters, equivariant
homology group of the affine Steinberg variety can be defined as
\[
H_*^{\underline{G(\mathcal{K})}\times \C^{\times}_{\rot}}(\widetilde{St}):= 
H_*^{\underline{\bI} \rtimes \C^{\times}_{\rot}}({}_0\tilde{\cR}_0)
\]
and the latter is isomorphic to \(\HH_{c,\hbar}\) under a specialization of parameters. Here the parameter \(c\) depends on our choice of the  equivariant structure with respect to the loop rotation group \(\C^{\times}_{\rot}\). Their work however uses localization techniques which we are able to avoid, thereby providing an isomorphism over the full parameter space, see Theorem \ref{thm:iwahoricoulombisdaha}.

Similarly, one can define the affine Grassmannian version \({}_i\cR_j\) of the above spaces. Since the fibers of the projection \({}_i\tilde{\cR}_j\to {}_i\cR_j\) are classical Springer fibers, we have a geometric model for the spherical algebra (see Corollary \ref{cor: sperical coulomb is daha})  
\[
\eee\HH_{c,\hbar}\eee\cong H_*^{\underline{G(\mathcal{O}})\ltimes \C^{\times}_{\rot}}({}_0\cR_0).
\]

 Thus, it is natural to define
\[{}_i\cA_j=H_*^{\underline{G(\mathcal{O}})\ltimes \C^{\times}_{\rot}}({}_i\cR_j),\quad {}_\bullet\cA_\bullet=\bigoplus_{i\le j}{}_i\cA_j.\]
As explained in \cite{BFN3, DGGH, Webster}, there is a natural associative convolution product  
\[H_*^{\underline{G(\mathcal{O}})\ltimes \C^{\times}_{\rot}}({}_i\cR_j)\otimes  H_*^{\underline{G(\mathcal{O}})\ltimes \C^{\times}_{\rot}}({}_j\cR_k)\to H_*^{\underline{G(\mathcal{O}})\ltimes \C^{\times}_{\rot}}({}_i\cR_k).\]
By associativity the convolution descends to give bilinear product maps
\[{}_i\cA_j\bigotimes_{{}_j\cA_j}{}_j\cA_k\to {}_i\cA_k.\]
One of our main results partially identifies the Coulomb branch $\Z$-algebra ${}_\bullet \cA_\bullet$  in terms of the algebraic $\Z$-algebra ${}_\bullet \cB_\bullet$.

\begin{theorem}
The Coulomb branch $\Z$-algebra ${}_\bullet \cA_\bullet$ satisfies the following properties:
\begin{itemize}
\item [(a)] For all $i$ the algebras ${}_i \cA_i$ and ${}_i \cB_i$ are isomorphic.
\item [(b)] For all $i$ the bimodules ${}_i \cA_{i+1}$ and ${}_i \cB_{i+1}$ are isomorphic.
\item [(c)] For $G=\GL_n$, the bimodules ${}_i \cA_{j}$ and ${}_i \cB_{j}$ are isomorphic for all $i\le j$. Moreover, the $\Z$-algebras ${}_\bullet \cA_\bullet$ and  ${}_\bullet \cB_\bullet$ are isomorphic.
\item[(d)] The homological grading on ${}_i\cA_j=H_*^{\underline{G(\mathcal{O}})\ltimes \C^{\times}_{\rot}}({}_i\cR_j)$ corresponds to the grading on ${}_i \cB_{j}$ induced by the grading on $\HH_{c,\hbar}$ (see Section \ref{sec:cherednik}).
\end{itemize}
\end{theorem}

We prove part (a) as Theorem \ref{thm:iwahoricoulombisdaha}, part (b) as Theorem \ref{thm: one step} and part (c) as Theorem \ref{thm:Z-algebra-iso}. In Proposition \ref{prop:noncomm localization for any group} we also provide an explicit basis for the associated graded of  ${}_\bullet \cA_\bullet$  with respect to Bruhat filtration in all types,  see in particular Theorem \ref{thm:gr of bruhat} for $G=\GL_n$.

The main difficulty in proving part (c) is that  ${}_\bullet \cB_\bullet$ is generated by the ``degree one bimodules" ${}_i \cB_{i+1}$  by definition, while this is not clear at all for ${}_\bullet \cA_\bullet$. For $G=\GL_n$, we resolve this difficulty by a careful combinatorial analysis of the basis in Theorem \ref{thm:gr of bruhat}, and using deep results of Gordon and Stafford on $\Z$-algebras for rational Cherednik algebras \cite{GS1,GS2}.

Next, we turn to the Algebraic Geometry column of the table. In the commutative limit $c=\hbar=0$ the $\Z$-algebra ${}_\bullet \cA^{\hbar=0}_\bullet$ becomes a graded commutative algebra, as ${}_i \cA^{\hbar=0}_j$ only depends on the difference $d=j-i$. 
For $d=0$ the algebra ${}_0 \cA^{\hbar=0}_0$ can be identified with the algebra of symmetric polynomials on $T^*T^{\vee}$, or, equivalently, the algebra of functions on $T^*T^\vee/W$. 
For $d=1$ the module ${}_0 \cA^{\hbar=0}_1$ can be identified with the space of antisymmetric polynomials on $T^*T^{\vee}$  (see Theorem \ref{thm: properties of coulomb algebra} for both $d=0$ and $d=1$ cases).
Our next main result identifies this graded algebra explicitly.

\begin{theorem}(Theorem \ref{thm: coulomb is symbolic})
\label{thm: coulomb is symbolic intro}
Let $\epsilon$ be the sign representation of $W$ and $\eee_d$ be the idempotent in $\C[W]$ corresponding to $\epsilon^{\otimes d}$.
We have 
$$ 
{}_0 \cA^{\hbar=0}_d\simeq \eee_d \bigcap_{\alpha\in\Phi^+}\langle 1-\alpha^\vee, y_\alpha\rangle^{d}
$$
where in the right hand side we have an intersection of ideals in $\C[T^*T^\vee]$. Here we regard $1-\alpha^{\vee}$ as functions on $T^{\vee}$ and $y_\alpha$ as coordinates in the cotangent fiber corresponding to the  positive roots $\alpha\in\Phi^+$.
The isomorphism agrees with the convolution structure on the left hand side and the multiplication on the right hand side.

The homological grading on ${}_0 \cA^{\hbar=0}_d$ corresponds to the natural grading on $\C[T^*T^\vee]$ given by dilation of the cotangent fiber (so that  $1-\alpha^{\vee}$ have degree 0 and $y_\alpha$ have degree $-1$).
\end{theorem}

We can then define an algebraic variety
\begin{equation}
\label{eq: def resolution intro}
\widetilde{\Comm}_{G^{\vee}}:=\Proj \bigoplus_{d=0}^{\infty} {}_0 \cA^{\hbar=0}_d
\end{equation}
which is a partial resolution of $\Spec {}_0 \cA^{\hbar=0}_0=T^*T^\vee/W$. In other words, we identify the graded algebra $\bigoplus_{d=0}^{\infty} {}_0 \cA^{\hbar=0}_d$ with the homogeneous graded ring of $\widetilde{\Comm}_{G^{\vee}}$.
The homological grading on ${}_0 \cA^{\hbar=0}_d$ described in Theorem \ref{thm: coulomb is symbolic intro}       
  corresponds to the $\C^{\times}_{\h}$ action on 
$\widetilde{\Comm}_{G^{\vee}}$.

\begin{remark}
\label{rem: finite generation}
The algebra described in Theorem \ref{thm: coulomb is symbolic intro} is (up to a projection by $\eee_d$) an example of the so-called symbolic Rees algebra for the ideal  
$$
\bigcap_{\alpha\in\Phi^+}\langle 1-\alpha^\vee, y_\alpha\rangle\subset \C[T^*T^\vee].
$$
It is in general a complicated question whether a symbolic Rees algebra is finitely generated (e.g. \cite{Grifo}). However, \cite[Theorem 3.18]{BFN2} ensures that $\widetilde{\Comm}_{G^{\vee}}$ can be realized as a Hamiltonian reduction of a larger 
Coulomb branch with flavor symmetry by a certain torus action. This implies that $\bigoplus_{d=0}^{\infty} {}_0 \cA^{\hbar=0}_d$ is a finitely generated algebra, and $\widetilde{\Comm}_{G^{\vee}}$ is a quasi-projective variety.
\end{remark}

By the work of Haiman \cite{Haiman}, for $G=\GL_n$ this implies the following:

\begin{proposition}
\label{thm: intro hilb}
For $G=\GL_n$, we get 
$$
\Proj \bigoplus_{d=0}^{\infty} {}_0 \cA^{\hbar=0}_d=\Hilb^n(\C^{\times}\times \C),
$$ 
in particular, it is a smooth resolution of $T^*T^\vee/W\simeq (\C^{\times}\times \C)^n/S_n$.
\end{proposition}

\begin{remark}
A different proof of Theorem \ref{thm: intro hilb} essentially follows from \cite[Theorem 3.10]{BFN3}, which identifies a resolution of the Coulomb branch for $(\GL_n,\mathfrak{gl}_n\oplus \C^n)$, constructed using flavor symmetry, with $\Hilb^n(\C^2)$. The same proof in {\em loc. cit.} applied to $(\GL_n,\mathfrak{gl}_n)$ yields Theorem  \ref{thm: intro hilb}.

Another proof can be extracted from more general results of Nakajima and Takayama \cite{NT,Tak} which identify the resolved Coulomb branches  of quiver gauge theories of affine type A with certain Cherkis bow varieties \cite{Cherkis}, and certain moduli spaces of parabolic sheaves on $\P^1\times \P^1$. See, in particular, \cite[Theorem 4.9]{BFN3} for more context and details.
\end{remark}

Next, we study graded modules over all of the above $\Z$-algebras and graded algebras. The convolution structure  of spaces
 \({}_i\tilde{\cR}_j\) allows us to define a correspondence between the affine Springer fibers in the affine flag variety \(\tilde{\Sp}_{t^i\gamma}  \),
\(\tilde{\Sp}_{t^j
  \gamma}\subset G(\cK)/\bI \). 
Similarly, we prove in Theorem \ref{thm:action} that \(\bigoplus_{j\ge 0} H_*(\Sp_{t^j\gamma})\) is a module over \(\Z\)-algebra \({}_\bullet\cA_\bullet\). This is very similar to the BFN Springer theory developed in \cite{HKW,GK}. 
In the commutative variant, we obtain a graded module over the graded algebra, and hence a quasi-coherent sheaf over its $\Proj$ construction defined by \eqref{eq: def resolution intro}, which proves
Theorem~\ref{thm:mainthm}.

Finally, we outline a simple construction of the above action. Recall that the homology of the affine Grassmanian version of the Springer fibers \(\Sp_{t^j\gamma}\subset G(\cK)/G(\cO)\) can be described in terms
of the action of the finite Weyl group:
\begin{equation}\label{eq:sig-triv-Sp}
H_*(\Sp_{t^j\gamma })= H_*(\tSp_{t^j\gamma })^W,\quad H_*(\Sp_{t^{j-1}\gamma })= H_*(\tSp_{t^j\gamma })^{\epsilon}[-2\dim G/B].
\end{equation}
where $\epsilon$ is the sign representation of $W$ and $[-2\dim G/B]$ denotes shift in homological degree.
We explain the details of the isomorphisms in the Lemma~\ref{lem: spherical antispherical}.  Given a class in 
$H_*(\Sp_{t^{j-1}\gamma })$, we can identify it with a class in $H_*(\Sp_{t^j\gamma })^{\epsilon}$, then act by an antisymmetric polynomial (using the commutative version of the double affine action of \cite{Yun1}) and get a class in $H_*(\Sp_{t^j\gamma })^W=H_*(\Sp_{t^j\gamma })$. This construction gives a map
\begin{equation}
\label{eq: intro A acts}
A_G\otimes H_*(\Sp_{t^{j-1}\gamma })\to H_*(\Sp_{t^{j}\gamma })
\end{equation}
where $A_G={}_0 \cA^{\hbar=0}_1$ is the space of antisymmetric polynomials. 

It is unclear if this approach can be used to define the action of the full graded algebra $\bigoplus_{d=0}^{\infty} {}_0 \cA^{\hbar=0}_d$, the main obstacles are:
\begin{itemize}
\item It is unclear how to verify the relations between the products of elements of $ {}_0 \cA^{\hbar=0}_1$ inside ${}_0 \cA^{\hbar=0}_d$
\item For $G\neq \GL_n$, it is unclear if the algebra is generated in degree 1. 
\end{itemize}

To avoid these obstacles, we have abandoned this approach altogether and instead used the machinery of Coulomb branches throughout the paper.  Nevertheless, {\em a posteriori} we conclude that  the action of the degree 1 part of the algebra agrees with \eqref{eq: intro A acts}, and hence all necessary relations are satisfied.

\subsubsection{Outline of the paper}
In Section \ref{sec:asf} we define the affine Springer fibers and some background material.
The sheaves we construct live on a partial resolution of $T^*T^\vee/W$, which is introduced in Section \ref{sec:commuting}. In Section \ref{sec:cherednik}, we study the trigonometric Cherednik algebra and a natural $\Z$-algebra built out of it, which is the algebraic main part of the construction. In Sections \ref{sec:coulomb},\ref{sec: adjoint coulomb} and \ref{sec:bfnspringer} we study the affine Springer fibers using Coulomb branch algebra machinery, in particular constructing a geometric $\Z$-algebra action and comparing it to the one in Section \ref{sec:cherednik}. In Section \ref{sec:finitegeneration} we prove that the sheaves we construct are coherent whenever $\gamma$ is homogeneous. In Section \ref{sec:examples}, we study some homogeneous examples in detail.

\subsection*{Acknowledgments}
The authors  would like to thank Roman Bezrukavnikov, Tudor Dimofte, Pavel Etingof, Pavel Galashin, Niklas Garner, Ivan Losev, Ng\^o Bao Ch\^au, Hiraku Nakajima, Wenjun Niu, Jos\'e Simental Rodriguez, Eric Vasserot, Monica Vazirani, Ben Webster, Alex Weekes and Zhiwei Yun for useful discussions. We thank the organizers of the MRSI program ``Enumerative Geometry Beyond Numbers'' in 2018 for providing an ideal working environment which motivated us to start this project. 

The work of E. G. was partially supported by the NSF grant DMS-1760329 and DMS-2302305. The work of A. O. was partially supported by the NSF grants DMS-1352398 and DMS-1760373. The work of O. K. was partially supported by the Finnish Academy of Science and Letters grant for postdoctoral research abroad.

\section{Affine Springer fibers}
\label{sec:asf}
\subsection{Definitions}
\label{sec: asf definitions}
Let $G/\kk$ be a reductive group over a field $\kk$. We assume $\kk=\bar{\kk}$ and that the 
characteristic is zero or large enough (no attempt will be made to give bounds). Fix some pinning $T\subset B
\subset G$ of the root system of $G$. Define $N=\dim(G/B)$.
Let $\cK=\kk(\!(t)\!)$ and $\cO=\kk\llbracket t\rrbracket$. Write $G(\!(t)\!)=G(\cK)$ and $G\llbracket t\rrbracket=G(\cO)$. Denote also $\fg=\Lie(G)$ 
and $\fg(\!(t)\!)=\fg(\cK), \fg\llbracket t\rrbracket=\fg(\cO)$.

Let $\bP$ be a standard parahoric subgroup of $G(\cK)$, i.e. the pullback of a standard parabolic subgroup $P\subset G(\kk)$ under the evaluation at zero map $\ev_0: G(\cO)\to G(\kk)$. Let $\Fl^\bP_G=G(\cK)/\bP$ 
be the corresponding partial affine flag variety. When $\bP=\bI$ is the 
Iwahori subgroup of $G(\cK)$ corresponding to $P=B$, we simply write $
\Fl_G:=\Fl^\bI$ for the affine flag variety and when $\bP=G(\cO)$ we 
write $\Gr_G:=\Fl^{G(\cO)}$ for the affine Grassmannian. Since it will usually be clear from the context, we will also omit the subscript $G$.
  
For any $\bP$ 
and $\gamma\in \fg(\cK)$ , define the affine Springer fiber
$$\Sp_\gamma^\bP:=\{g\bP|\Ad(g^{-1})\gamma \in \Lie(\bP)\}\subset \Fl^\bP.$$ When $\bP=G(\cO)$ we omit the superscript and when $\bP=\bI$ we 
write $\tSp_\gamma=\Sp_\gamma^\bI$. 

The space $\Sp_\gamma^\bP$ is a sub-ind-scheme of $\Fl^\bP$. It is always nonreduced, but since it makes no difference to us, we will only work with the reduced structure of $\Sp_\gamma^\bP$ (see \cite[Sections 2.2.9 and 2.5.1]{YunLectures} for more details). $\Sp_\gamma^\bP$ is locally
finite-dimensional if and only if $\gamma$ is regular semisimple
\cite{KL88}. We shall mostly focus on the case when $\gamma$ is regular semisimple from now 
on. We also assume $\gamma$ is compact, i.e. contained in some Iwahori subgroup, or equivalently that the affine Springer fiber is nonempty.
The ind-scheme $\Sp_\gamma$ is locally 
of finite type, and by results of \cite{KL88} there exists a free 
abelian group $L_\gamma$ acting on $\Sp_\gamma^\bP$ freely and a projective 
scheme $S\subset \Sp_\gamma^\bP$ such that $L_\gamma\cdot S=\Sp_\gamma^\bP$. The free abelian group $L$ can be identified with the cocharacter lattice of the centralizer $G_\gamma:=C_{G(\cK)}(\gamma)$  of $\gamma$:
$$
L_\gamma=X_*(G_\gamma).
$$
In particular, $L$ is trivial if and only if $\gamma$ is elliptic. In this case, $\Sp_\gamma^\bP$ is a projective variety.
\begin{remark}
By the Jordan decomposition, we can write any $\gamma$ as the sum of commuting semisimple and nilpotent elements: $\gamma=\gamma_s+\gamma_n$. Therefore, we can reduce the study of general $\gamma$ to nilpotent and semisimple $\gamma$, as $\Sp_\gamma=\Sp_{\gamma_s}\cap \Sp_{\gamma_n}$. 
While it is also interesting to study non-regular semisimple elements, much about this case can in principle be extracted from the regular semisimple case by the fact that the centralizer of a semisimple element is a reductive group over $\cK$. In the present work, we are mostly concerned with $H^{K_\gamma}_*(\Sp_\gamma)$ for $K_\gamma\subseteq J_\gamma(\cO)$, i.e. a subgroup of the $\cO$-points of a smooth integral model of the centralizer of $\gamma$\footnote{This is to ensure that the $K_\gamma$-action on $\Sp_\gamma$ factors through a finite-dimensional quotient and the equivariant BM homology can be defined. The assumption on $K_\gamma$ can possibly be weakened.}  defined in Section \ref{sec:endoscopy}.

However, the nilpotent elements seem more mysterious from our point of view. It is clear e.g. by the convergence of the corresponding orbital integrals that the centralizers of nilpotent elements are quite large. In the case where $\gamma$ is nilpotent, it is not even known if there is a Levi factor of the centralizer. It does make sense to ignore the centralizer (or at most use some compact subgroups thereof) and use finite-dimensional approximation to study the Borel-Moore homology of $\Sp_\gamma$ even in the nilpotent cases, as is done e.g. by Sommers in \cite{Sommers}.
\end{remark}
\subsection{The Springer action}
Assume for now that $\kk=\C$ or that we are using \'etale cohomology over $\bar{\Q}_\ell$ for $\ell\neq \text{char}(k)$. One of the remarkable things about $\tSp_\gamma$ is that $H_*(\tSp_\gamma)$ has an action of the extended affine Weyl group $\tW=W\ltimes X_*(T)$ as shown by Lusztig \cite{Lu96} (for $G$ adjoint) and Yun \cite{Yun1} (in general), analogously to the Weyl group action in the cohomology of classical Springer fibers. 

\begin{lemma}
\label{lem: spherical antispherical}
Let $\gamma\in \fg(\cK)$ be an element such that $\gamma=t\gamma_0$ for a regular semisimple compact element $\gamma_0\in \fg(\cK)$.
Then under the Springer action of $W\subset \tW$ on $H_*(\tSp_\gamma)$, we have a natural identification
\begin{equation}\label{eq:triv}
H_*(\tSp_\gamma)^W=H_*(\Sp_\gamma)
\end{equation}
and an isomorphism 
\begin{equation}\label{eq:sign}
H_*(\tSp_\gamma)^{\epsilon}\cong H_*(\Sp_{\gamma_0})[-2N].
\end{equation}
Here $[-2N]$ means a shift in homological degree and $N=\dim G/B$, and $\epsilon$ denotes the sign representation of $W$.
\end{lemma}
\begin{proof}
The first part is due to \cite[Section 2.6]{YunSph}, and the second part is well-known but not found in the literature, so we give a proof here. 

First recall the construction of the Springer action for the subgroup $W\subset \tW$. Since $\bI\subset G\llbracket t\rrbracket$, we have natural projections 
$\Fl\to \Gr$ and $\tSp_\gamma\to \Sp_\gamma$. For any $\bP$, write $P=\bP/t\bP$ and $\mathfrak{p}$ for its Lie algebra. There are natural maps of fpqc sheaves 
$\Sp_\gamma^\bP\to [\mathfrak{p}/P],$ which send the cosets $g\bP$ to the respective images of $g^{-1}\gamma g$ under the projections $\Lie(\bP)\to \mathfrak{p}$. In particular there is a cartesian diagram 
\begin{equation}
\label{eq: Springer diagram}
\begin{tikzcd}
\tSp_\gamma \arrow[r,"\varphi'"]\arrow[d, "\pi"'] & \left[ \tilde{\fg}/G \right]=[\fb/B] \arrow[d, "\pi'"]\\
\Sp_\gamma \arrow[r,"\varphi"] &  \left[ \fg/G \right]
\end{tikzcd}
\end{equation}
where the right-hand side is naturally identified with the Grothendieck-Springer resolution for $G$. Since $\gamma=t\gamma_0$, it is clear that the image of $\varphi$ will be contained in $[\cN/G]\subset [\fg/G]$ and the image of $\varphi'$ will be contained in $[\widetilde{\cN}/G]$. The restriction $S:=\pi'_*\C|_\cN$ is perverse, and is called the Springer sheaf. It is in fact isomorphic to a direct sum of IC complexes on nilpotent orbits, and it is known by classical Springer theory that $$\End_{\text{Perv}^G(\cN)}(S)\cong \C[W].$$ In particular, there is a map $$\C[W]\to \End_{\text{Perv}(\Sp_\gamma)}(\varphi^*S)$$ and hence an action of $W$ on $H_*(\tSp_\gamma)$. Decomposing the regular representation of $W$, we see that there is some IC complex $\cF$ on $[\cN/G]$ corresponding to the sign representation $\epsilon$. From classical Springer theory it follows that this complex is isomorphic as a perverse sheaf to the shifted skyscraper sheaf $\C_{\{0\}}[-2N]$. 

By proper base change, 
$$H_*(\tSp_\gamma)^{\epsilon}=\Hom(\varphi^*\cF, \pi_*\C)\cong \Hom(\varphi^*\cF, \varphi^*S).$$
Now note that $$\varphi^{-1}(0)=\{gG(\cO)\in \Sp_\gamma|g^{-1}\gamma g\equiv 0 \; (\text{mod } t)\}=\Sp_{\gamma_0}.$$ This is a closed subspace, so $\varphi^*\cF$ is isomorphic to the (shifted) extension by zero of the constant sheaf on $\Sp_{\gamma_{0}}$. 
In particular, $R\Gamma(\varphi^*\cF)\cong H_*(\Sp_{\gamma_0})$. On the other hand, by the adjunction 
$$\Hom(\varphi^*\cF,\varphi^*S)=\Hom(\cF,\varphi_*\varphi^*S)$$ it is clear that $\Hom(\varphi^*\cF,\varphi^*S)\cong R\Gamma(\varphi^*\mathbb{C}_{\{0\}})[-2N]$.
Thus $H_*(\tSp_\gamma)^{\epsilon}\cong H_*(\Sp_{\gamma_0})[-2N]$.

\end{proof}

We can rephrase Lemma \ref{lem: spherical antispherical} as follows.
We have the Leray filtration on the Borel-Moore homology of $\tSp_{\gamma}$ such that
\begin{equation}
\label{eq:leray}
\gr^kH_*(\tSp_{\gamma})=\bigoplus_{i+j=k}H_i(\Sp_{\gamma},R^j\pi_*\C),\ 0\le j\le 2N.
\end{equation}

\begin{corollary}
a) The $W$-invariant part of the homology of $\tSp_{\gamma}$ is canonically isomorphic to the $j=0$ part of \eqref{eq:leray}:
$$
H_k(\tSp_{\gamma})^{W}=H_k(\Sp_{\gamma},R^0\pi_*\C)=H_k(\Sp_{\gamma},\C).
$$

b) The associated graded of the $W$-antiinvariant part is isomorphic to the $j=2N$ part of \eqref{eq:leray}.
In other words, the restriction of the obvious map
$$
H_k(\tSp_{\gamma})^{\epsilon}\to H_{k-2N}(\Sp_{\gamma},R^{2N}\pi_*\C)\cong H_{k-2N}(\Sp_{\gamma_0})
$$
is an isomorphism.
\end{corollary}

\begin{remark}
The result is also true in cohomology, cohomology with compact supports and BM homology (where we replace the constant sheaf by $\omega_{\tSp_\gamma}$) by identical reasoning.
\end{remark}
\begin{remark}
The centralizer $G_\gamma=C_{G(\cK)}(\gamma)$ acts naturally on $\tSp_\gamma$, inducing an action of the component group in cohomology. In particular, we have the Springer action for equivariant versions of any of the above theories, as well as a commuting action of the component group of the centralizer.
\end{remark}

\subsection{Extended symmetry}
\label{sec:actions}
In addition to the action of $\tW$ on $H_*(\tSp_\gamma)$ there is a degenerate action of the character lattice $X^*(T)$ on $H_*(\tSp_\gamma)$ defined as follows. There is a natural map $\bI\to T$ realizing $T$ as the reductive quotient of $\bI$. In particular, each character $\chi: T\to \mathbb{G}_m$ gives a map $\bI\to \mathbb{G}_m$ and in particular a $\mathbb{G}_m$-torsor $\cL(\chi)$ on $\Fl$. As a line bundle, we can write this as $$G\times^{\bI,\chi} \A^1\to \Fl.$$ Cap product with $c_1(\cL(\chi))$ defines an action of $X^*(T)\otimes_\Z \C$ on $H_*(\tSp_\gamma)$.

\begin{theorem}
Let $X_*(T)\subset \tW$ be the translation part of the extended affine Weyl group. Then the Springer action of $X_*(T)$ and the action of $X^*(T)$ defined above commute.
\end{theorem}
\begin{proof}
This is proved in \cite[Corollary 6.1.7]{Yun1}.
\end{proof}
\begin{remark}
Note that this is not true in {\em equivariant} Borel-Moore homology when we take into account the loop rotation action of $\G_m^{\rot}$, as one gets essentially the relations in the degenerate DAHA from Definition \ref{def:dDAHA}..
\end{remark}
Note that the action of $X_{*}(T)$ gives rise to an action of $\C[T^\vee]$ and the action of $X^*(T)\otimes_\Z \C$ gives rise to an action of $\C[\ft]$ on $H_*(\tSp_\gamma)$. We can thus summarize the above results in the following

\begin{proposition}
\label{prop:J0 action}
The non-equivariant Borel-Moore homology $H_*(\tSp_\gamma)$ is a (right) module over $\C[T^*T^\vee]\rtimes W$.
\end{proposition}

\begin{proposition}
\label{prop:A0 action}
The non-equivariant Borel-Moore homology  $H_*(\Sp_\gamma)$ is a module over $\C[T^*T^\vee]^W$.
\end{proposition}

\begin{proof}
By Lemma \ref{lem: spherical antispherical} we have $H_{*}(\Sp_{\gamma})=H_*(\tSp_{\gamma})^{W}$, and
by Proposition \ref{prop:J0 action} $H_*(\tSp_{\gamma})$ has an action of $\C[T^*T^\vee]$. By symmetrizing, we get
the action of $\C[T^*T^\vee]^W$ on $H_{*}(\Sp_{\gamma})$.
\end{proof}

We record the following lemma here. 
\begin{lemma}
\label{lem: y vs Leray}
The action of $X^*(T)\otimes_\Z \C$  on $H_*(\tSp_{\gamma})$ from Section \ref{sec:actions} decreases the Leray filtration \eqref{eq:leray} by two.
\end{lemma}
 
\begin{proof}
This follows from the fact that the action comes from cap product with Chern classes of the line bundles $L(\chi)$ constructed in the beginning of this section.
\end{proof}

Let $\Delta$ be the top-dimensional class in $H^{2N}(G/B)$. By the isomorphism $H^*(G/B)\cong \C[\ft]_W$, we can identify $\Delta$ with an antisymmetric polynomial in $\C[\ft]$, namely $\Delta=\prod_{\alpha\in\Phi^+}y_\alpha$. Moreover, note that we may write antisymmetric polynomials as 
$\C[\ft]^{\epsilon}=\Delta\cdot \C[\ft]^{W}$.
\begin{lemma}
\label{lem: Vandermonde}
We have that 
$$
i_*(H_{k-2N}(\Sp_{t^{-1}\gamma}))=\Delta \cdot H_k(\tSp_{\gamma})^{\epsilon}.
$$
where $i:\Sp_{t^{-1}\gamma}\to \Sp_{\gamma}$ is the natural inclusion.
\end{lemma}

\begin{proof}
By Lemma \ref{lem: y vs Leray} the operator $\Delta$ preserves the decomposition \eqref{eq:leray} and decreases the $j$-grading by $2N$. Since $j\le 2N$, the action of $\Delta$ kills all the summands in \eqref{eq:leray} with $j<2N$, so that 
$$
\Delta\cdot H_k(\tSp_{\gamma})^{\epsilon}=\Delta\cdot H_{k-2N}(\Sp_{\gamma},R^{2N}\pi_*\C).
$$
At $p\in \Sp_{\gamma}$ where $\pi^{-1}(p)$ is a proper subset in the flag variety, $R^{2N}\pi_*\C|_{p}=0$.
On the other hand,  $\pi^{-1}(p)$ is the full flag variety if and only if $p\in i(\Sp_{t^{-1}\gamma})$ and in this case
$\Delta:R^{2N}\pi_*\C\to R^{0}\pi_*\C$ is the isomorphism. Therefore 
$$
\Delta\cdot H_{k-2N}(\Sp_{\gamma},R^{2N}\pi_*\C)=i_*(H_{k-2N}(\Sp_{t^{-1}\gamma})).
$$ 
\end{proof}

\subsection{The lattice action}
\label{sec: lattice action}

Let $\gamma$ be compact and regular semisimple, and let $G_\gamma$ be the stabilizer of $\gamma$ in $G(\cK)$ as before. Obviously, $G_\gamma$ acts on $\Sp_\gamma$, giving an action of the component group $\pi_0(G_\gamma)$
in the homology of $\Sp_\gamma$. Since the action in this case is proper \cite{KL88}, we also get an action in the Borel-Moore homology $H_*(\Sp_\gamma)$.

The $W$-invariant translation part of the stabilizer action restricts to an action of the ''spherical part" 
$\C[X_*(T)]^W$ on $H_*(\tSp_\gamma)$.  This action commutes with the Springer action, and the local main theorem of \cite{YunSph} identifies the spherical part of the Springer action with the spherical part of the lattice action. 
More precisely,
we recall the construction \cite[Section 2.7-2.8]{YunSph} of a canonical homomorphism $\C[X_*(T)]^W\twoheadrightarrow \C[\pi_0(G_\gamma)]$.

 Define $\mathfrak{c}=\fg/\!\!/G=\ft/\!\!/W$, and let $a(\gamma)$ be the image of $\gamma$ under the projection  $
 \fg(\cK)\to \mathfrak{c}(\cK)$. Then we can consider the commutative diagram
 $$
 \begin{tikzcd}
 \Spec \widetilde{\cO_{\cK}} \arrow{r} \arrow{dr}& \Spec \cO_{\cK,a} \arrow{r}\arrow{d}& \ft\arrow{d}\\
  & \Spec \cO_{\cK} \arrow{r}{a(\gamma)}& \mathfrak{c}
 \end{tikzcd}
 $$
 where the right square is Cartesian by definition. The ring $\widetilde{\cO_{\cK}}$ is a normalization of $\cO_{\cK,a}$. 
 Choose a component $Y\subset  \Spec \widetilde{\cO_{\cK}}$, and let $W_Y\subset W$ be the stabilizer of $Y$. 
By \cite[Prop. 3.9.2]{Ngo}, the choice of  $Y$
allows one to define a surjection
$$
X_{*}(T)\twoheadrightarrow X_{*}(T)_{W_Y}\twoheadrightarrow \pi_0(G_{\gamma}).
$$
Furthermore, the corresponding group algebra homomorphism (restricted to $\C[X_*(T)]^W$)
\begin{equation}
\label{eq: Ngo surjection}
 \C[X_*(T)]^W\twoheadrightarrow \C[\pi_0(G_\gamma)].
\end{equation}
does not depend on the choice of of $Y$.

\begin{theorem}\cite{YunSph}
\label{thm: lattice action factors}

The spherical part of the Springer action of $\C[X_*(T)]^W$ on $H_*(\tSp_\gamma)$ factors through the canonical homomorphism \eqref{eq: Ngo surjection}.
\end{theorem}

By Propositions \ref{prop:J0 action} and \ref{prop:A0 action}, the (BM) homology of $\Sp_{\gamma}$ and the homology of $\tSp_\gamma$ are modules over $\C[T^*T^{\vee}]^{W}$ and hence define quasicoherent sheaves $\cF_\gamma'$ and $\widetilde{\cF}_\gamma'$ on $(T^*T^\vee)/W$.

\begin{lemma}
\label{lem:support}
These quasicoherent sheaves are actually coherent, and their (reduced) support is contained in the Lagrangian subvariety $\{0\}\times T^\vee /W\subset T^*T^\vee/W$.
Moreover, the dimension of their support equals the rank of the centralizer $G_{\gamma}$.
\end{lemma}

\begin{proof}
By Kazhdan-Lusztig \cite{KL88} the homology of $\Sp_{\gamma}$  is finitely generated over $\pi_0(G_{\gamma})$ and
 $\pi_0(G_{\gamma})$ acts freely on the components of $\Sp_{\gamma}$ which all have the same dimension. 
 By Theorem \ref{thm: lattice action factors} we conclude that the homology of $\Sp_{\gamma}$ is finitely generated over $\C[T^\vee]^W$ and hence the corresponding sheaf $\cF'_\gamma$ on $T^*T^{\vee}/W$ is coherent. Similarly 
$\widetilde{\cF}_\gamma'$  is coherent as well.
 
By Lemma \ref{lem: spherical antispherical} we have $H_*(\Sp_{\gamma})\simeq H_*(\tSp_{\gamma})^W$.
It is clear from Lemma \ref{lem: y vs Leray} that the action of $\C[\ft]$ on  $H_*(\tSp_{\gamma})$ is nilpotent, and symmetric functions in $\C[\ft]^{W}$ act  on  $H_*(\tSp_{\gamma})$  by 0, as they do in $H_*(G/B)$. Therefore   $\cF_\gamma', \widetilde{\cF}_\gamma'$ are supported on  
$\{0\}\times T^\vee/W$.

Furthermore, the surjection \eqref{eq: Ngo surjection} defines a subvariety 
$$
\Spec \C[\pi_0(G_\gamma)]\subset \Spec \C[X_*(T)]^W=T^{\vee}/W
$$ 

and Theorem \ref{thm: lattice action factors} implies that $\cF_\gamma', \widetilde{\cF}_\gamma'$ are supported on  $\{0\}\times \Spec \C[\pi_0(G_\gamma)]$. This proves the last statement.

\end{proof}

\subsection{Equivariant versions, endoscopy}
\label{sec:endoscopy}
We can in fact upgrade this construction with the addition of equivariance to the picture. The centralizer $G_\gamma/\cK$ has a smooth integral model $J_\gamma$ over $\cO$, see e.g. \cite{Ngo}. As shown in {\em loc. cit.}, the stabilizer action on $\tSp_\gamma$ factors through the local Picard group $P_\gamma=G_\gamma(\cK)/J_\gamma(\cO)$ whose underlying reduced scheme is finite-dimensional and locally of finite type. Consider the connected component of the identity $P_\gamma^0$ of this group scheme. This is a linear algebraic group over $\C$ whose maximal reductive quotient contains a split maximal torus of rank equal to $\text{rank}(X^*(G_\gamma))$. Call this torus $T_\gamma$. It also acts on $\Sp_\gamma$, and we may take the equivariant BM homology as in \cite[Section 3]{Kivinen}. 
The construction of the Springer action etc. from the previous section go through $T_\gamma$-equivariantly. For simplicity, assume $T_\gamma \hookrightarrow T$. Then we have:

\begin{proposition}
The equivariant BM homology $H_*^{T_\gamma}(\Sp_\gamma)$ is a quasi-coherent sheaf on $T^*T^\vee/W$. Its reduced support has dimension $\leq 2\text{rk}(T_\gamma)$. In particular, if $\gamma$ is elliptic, the support is of dimension $2\text{rk}(Z(G))$, twice the split rank of the center of the group $G$.
\end{proposition}
\begin{proof}
The proof is similar to the proof of Lemma \ref{lem:support}.
The equivariant cohomology of a point $H^*_{T_\gamma}(\pt)$ acts on $H_*^{T_\gamma}(\Sp_\gamma)$ via the equivariant Chern classes $c_1(\cL(\chi))$ from before, giving that the support is contained in $\ft_\gamma\times T^\vee/W$. Results of  Yun \cite[Lemma 5.14]{YunSph} 

then imply that the spherical part of the Springer action in the equivariant homology $H_*^{T_\gamma}(\Sp_\gamma)$ factors through $\C[\pi_0(P_\gamma)]$ just as in Theorem \ref{thm: lattice action factors}.
 By \cite[3.9.]{Ngo}, the group $\pi_0(P_\gamma)$ is free abelian of rank $\text{rk}(G_\gamma)$, giving the desired dimension bound. When $\gamma$ is elliptic, $G_\gamma$ has the same rank as the center and the lattice action is given by moving between different connected components, in particular it is free. This gives equality in the elliptic case.
\end{proof}
\begin{remark}
For elements in a split torus $T=T_\gamma$ for which we have equivariant formality of the $H^*_{T}(\pt)$-action (studied e.g. in Section \ref{sec:finitegeneration}), the proof shows that the support is all of $T^*T^\vee/W$.
\end{remark}

If $G\neq \GL_n$, we note that in the determination of the above support we run into issues of endoscopy. For example for $G=SL_n$, the center $\mu_n$ will be contained in $T_\gamma$ for any $\gamma$. For example for elliptic $\gamma$, we might get support at $n$ points in $T^*T^\vee/W$. Let us now illustrate how this plays out in the case of general $G$ and $\gamma$ elliptic, following \cite{Ngo}.

Recall from \cite[Sections 4 and 6]{Ngo} that we may decompose $$H_*(\Sp_\gamma)=\bigoplus_{\kappa\in \pi_0(P_\gamma)^\vee}H_*(\Sp_\gamma)_\kappa$$ 
for the local Picard group. As shown in \cite[Section 2.7.]{YunSph}, the $\pi_0(G_\gamma)$-action further factors through the $\pi_0(P_\gamma)$-action, so through the composition
$$\C[X_*(T)]^W\twoheadrightarrow \C[\pi_0(P_\gamma)]$$ the above decomposition may be viewed as a decomposition over homomorphisms $\kappa: \C[X_*(T)]^W\to \C$.
One of the corollaries of the homological version of the Fundamental Lemma (see \cite[6.4.1.]{Ngo}) is that:
\begin{equation}
\label{eq:fundlemma}
H_*(\Sp_\gamma)_\kappa\cong H_*(\Sp^H_{\gamma_H})_{st}[\text{val}(\Delta_G(\gamma))-\text{val}(\Delta_H(\gamma_H))]
\end{equation} for some affine Springer fiber $\Sp^H_{\gamma_H}$ of an endoscopic group $H$ of $G$ and a homological shift corresponding to the transfer factor. Here $\Delta_G, \Delta_H$ denote discriminant functions on the Lie algebras $\fg, \fh$ and  "$st$" denotes the stable part, or in other words the part of the BM homology where the lattice acts unipotently (see e.g. \cite{YunSph}). 

Since $H_*(\Sp_\gamma)$ is finite-dimensional, by using Lemma \ref{lem:support}, the above $\kappa$-decomposition and Eq. \eqref{eq:fundlemma}, we can reformulate the fundamental lemma as follows.
\begin{theorem}
\label{thm:sheaffundlemma}
Consider the coherent sheaf $\cF_\gamma'$ on $T^*T^\vee/W$ constructed in the previous section.
It is supported at finitely many points of the form $(0,\kappa)\in T^*T^\vee/W$, where $\kappa \in T^\vee/W$ as above. Each stalk $(\cF_\gamma')_{(0,\kappa)}$ is isomorphic to the stalk at $(0,1)$ of an ''endoscopic sheaf" on $T^*T^\vee_H/W_H$ for some endoscopic group $H$ of $G$. This isomorphism is as modules for $\C[T^*T^\vee_H]^{W_H}$ (or even as graded modules after taking into account the $\C_\h^\times$-action on both sides). 
\end{theorem}

\section{The commuting variety}
\label{sec:commuting}
\subsection{The commuting scheme}

In this section, we introduce the partial resolution of the commuting variety we will be considering. The construction is algebraic in nature. We show the partial resolution coincides with $\Hilb^n(\C^\times\times \C)$ in the case $G=\GL_n$. 
In general, we show it is a normal variety and conjecture that locally its singularities are modeled by the $\Q$-factorial terminalizations constructed by Losev et al. in \cite{BALK,LosevDerived}. 
Finally, we introduce a certain open chart of the partial resolution, which turns out to be isomorphic to the universal centralizer of $G^{\vee}$. 

As above, let $\fg^{\vee}$ be the Lie algebra of $G^{\vee}$, $\ft^{\vee}=\ft^*$ is the Lie algebra of $T^{\vee}$ and $W$ is the Weyl group. We define two versions of the commuting scheme: $\Comm_{\fg^\vee}'$ is a subscheme of $\fg^{\vee}\times \fg^{\vee}$ cut out by the equation $[x,y]=0$, while $\Comm_{G^{\vee}}'=\Comm_{G^{\vee},\fg^{\vee}}'$ is a subscheme of $G^{\vee}\times (\fg^{\vee})^*$ cut out by the equation $\Ad_{g}(x)=x$.

 Define $\Comm_{\fg^{\vee}}:=\Comm_{\fg^{\vee}}'/\!\!/G^{\vee}=\Spec \ \C[\Comm_{\fg^{\vee}}']^{G^{\vee}}$ and 
 $\Comm_{G^{\vee}}=\Comm'_{G^{\vee}}/\!\!/G^{\vee}=\Spec \ \C[\Comm_{G^{\vee}}']^{G^{\vee}}$. It is a long-standing open question if these schemes are reduced.
 We collect some facts about $\Comm_{\fg^{\vee}}$ and $\Comm_{G^{\vee}}$ here.

There are natural restriction maps $\C[\Comm_{\fg^{\vee}}]\to \C[\ft^*\times \ft^*]^{W}$ and $\C[\Comm_{G^{\vee}}]\to \C[T^{\vee}\times \ft]^{W}$,
which induce maps $(\ft^*\times \ft^*)/W\to \Comm_{\fg^{\vee}}$ and $(T^*T^{\vee})/W\to \Comm_{G^{\vee}}$.
The former is surjective by the result of Joseph \cite[Theorem 0.2]{Joseph}   and defines an isomorphism $(\ft^*\times \ft^*)/W\simeq [\Comm_{\fg^{\vee}}]_{\red}$. 

In \cite[Proposition 5.24]{BFNRingobjects}, the following is proved by essentially reducing to Joseph's results in the rational case:
\begin{theorem}
The restriction of functions induces an isomorphism   $(T^*T^{\vee})/W\simeq [\Comm_{G^{\vee}}]_{\red}$.
\end{theorem}

We note that there are alternative proofs of the theorem, such as the one given by Gan-Ginzburg in type A:
\begin{theorem}[\cite{GG}]
For $G=\GL_n$ the scheme $\Comm_{\fg}$ is reduced.
\end{theorem}

From this it is easy to deduce
\begin{corollary}
For $G=\GL_n$ the scheme $\Comm_{G}$ is reduced.
\end{corollary}

\begin{proof}
For $G=\GL_n$ we have a natural embedding $G\subset \fg$ which induced embedding $\Comm'\subset \Comm_{\fg}'$ and
$\Comm_G\subset \Comm_{\fg}$. Since  $\Comm_{\fg}$  is reduced, $\Comm_G$ is reduced too.
\end{proof}

Recently, Chen-Ng\^o \cite{ChenNgo} proved that $\Comm_{\fg}$ is reduced for $\fg=\mathfrak{sp}_{2n}$,
and
for  results on \(\Comm_{\fg}\) for a general reductive Lie algebra \(\fg\) see \cite{LiNadlerYun}.

\subsection{Partial resolutions}

In this subsection we define several graded commutative algebras closely related to the commuting variety. By applying $\Proj$ construction to these graded algebras, we recover partial resolutions of $\Comm_{G^{\vee}}$. We summarize various maps between the algebras  in the following commutative diagram:

\begin{equation}
\label{diagram of algebras}
\begin{tikzcd}
   & \bigoplus J^d \arrow{r}  & \bigoplus I^{(d)}\\
\bigoplus {}_{0}\cA^{\hbar=0}_{d}  \arrow[rr,bend right,dotted] & \bigoplus A^d \arrow{u}\arrow{l}{a_d} \arrow{r}{b_d}& \bigoplus \eee_{d}  I^{(d)} \arrow{u}.
\end{tikzcd}
\end{equation}

All direct sums are over $d\ge 0$.
Next, we define all of the entries in this diagram, starting with the middle column.

 We denote by $\epsilon$ the one-dimensional sign representation of $W$.
We define by $A_{\ft}$, respectively $A$, as the subspace of $W$-antiinvariant functions (that is, $\epsilon$-isotypic component) in $\C[\ft^*\times \ft]$, respectively $\C[T^\vee \times \ft]$. Also, let $J_{\ft}$, respectively $J$ be the ideal in $\C[\ft^*\times \ft]$ (resp. $\C[T^\vee\times \ft]$) generated by $A_{\ft}$ (resp. A). Now $A^d$ and $J^d$ are the powers of $A$ and $J$ inside the respective polynomial rings, with the assumption that $A^0=\C[T^\vee\times \ft]^{W}$ and  $J^0= \C[T^\vee\times \ft]$.

In the right column we have a family of ideals $$I^{(d)}:=\bigcap_{\alpha\in\Phi^+}\langle 1-\alpha^\vee, y_\alpha\rangle^{d}\subset\C[T^\vee \times \ft]$$
Here $\alpha$ runs over the set of positive roots $\Phi^+$, $y_{\alpha}$ is the equation of the  root hyperplane in $\ft$ corresponding to $\alpha$, and $\alpha^{\vee}$ is the corresponding coroot understood as a character $\alpha^{\vee}:T^{\vee}\to \C^{\times}$.
For $d>1$, the ideals $I^{(d)}$ are the symbolic powers of $I^{(1)}$.  
In particular, we can consider a codimension 1 subtorus $\Ker(\alpha^{\vee})\subset T^{\vee}$ and a codimension 1 hyperplane $\{y_{\alpha}=0\}\subset \ft$, then their product $\Ker(\alpha^{\vee})\times \{y_{\alpha}=0\}$ is a codimension 2 subvariety in $T^{\vee}\times \ft$. For $d=1$ the ideal $I^{(1)}$ corresponds to the union of all such subvarieties over positive roots $\alpha$:
$$
I^{(1)}=I\left(\bigcup_{\alpha\in \Phi^+}\Ker(\alpha^{\vee})\times \{y_{\alpha}=0\}\right).
$$
\begin{example}
For $G=\GL_n$ we have $I^{(d)}=\bigcap_{i\neq j}\langle 1-\frac{x_i}{x_j},y_i-y_j\rangle^d\subset \C[x_1^{\pm},\ldots,x_n^{\pm},y_1,\ldots,y_n]$.
\end{example}

It is easy to see that  $I^{(d_1)}\cdot I^{(d_2)}\subset I^{(d_1+d_2)}$, so   we have a graded algebra structure on the direct sum of all $I^{(d)}$.
Furthermore, let 
$$
\eee_{d}=\frac{1}{|W|}\sum_{\sigma\in W}(-1)^{d\cdot \epsilon(\sigma)}\sigma
$$ 
denote the projector to the representation $\epsilon^d$ in $\C[W]$, that is, symmetrizer $\eee$ for $d$ even and antisymmetrizer $\eee_{-}$ for $d$ odd.

\begin{lemma}
\label{lem: symmetrizer power J}
a) We have $\eee_{d}J^d=A^d$ for all $d\ge 0$.

b) There are natural inclusions $b_d:J^d\to I^{(d)}, A^d\to \eee_dI^{(d)}.$
\end{lemma}

\begin{proof}
a) For $d=0$ this is clear from the definition, so we focus on $d>0$. Since $J$ is the ideal generated by $A$, it is spanned by elements of the form $h\cdot f$ for $h\in A$ and $f\in \C[T^\vee \times \ft]$, and $J^d$ is spanned by elements of the form $h_1\cdots h_d\cdot f$ for $h_1,\ldots,h_d\in A$ and $f\in \C[T^\vee \times \ft]$. Since $h_1,\ldots,h_d$ are antisymmetric, we get

$$
\eee_{d}h_1\cdots h_d\cdot f=\frac{1}{|W|}\sum_{\sigma\in W}(-1)^{d\cdot \epsilon(\sigma)}\sigma(h_1\cdots h_d f)=
\frac{1}{|W|}\sum_{\sigma\in W}h_1\cdots h_d\sigma(f)=
h_1\cdots h_d \eee(f) \in A^d.
$$
The last inclusion follows from the fact that $\eee(f)$ is a symmetric polynomial, so $h_d\eee(f)$ is an antisymmetric polynomial. This shows $\eee_d J^d\subset A^d$. On the other hand, by substituting $f=1$ in the above equation we get  $\eee_{d}h_1\cdots h_d=h_1\cdots h_d$, so $h_1\cdots h_d\in \eee_d J^d$ and $A^d\subset \eee_d J^d$.

b) Suppose that $g\in \C[T^{\vee}\times \ft]$ is an antisymmetric polynomial, $\alpha$ is a positive root and $s_{\alpha}\in W$ is the corresponding reflection. Then for all $(x,y)\in \Ker(\alpha^{\vee})\times \{y_{\alpha}=0\}$ we have 
$s_{\alpha}(x,y)=(s_{\alpha}(x),s_{\alpha}(y))=(x,y)$. Now
$$
s_{\alpha}g(x,y)=g(s_{\alpha}(x),s_{\alpha}(y))=g(x,y),
$$
but since $g$ is antisymmetric we get $s_{\alpha}g(x,y)=-g(x,y)$, hence $g(x,y)=0$. We conclude that $g$ vanishes on $\Ker(\alpha^{\vee})\times \{y_{\alpha}=0\}$ and therefore $g\in \langle 1-\alpha^\vee, y_\alpha\rangle.$ Since this holds for all $\alpha$, we get $g\in \bigcap_{\alpha\in \Phi^+}\langle 1-\alpha^\vee, y_\alpha\rangle=I^{(1)}$.

We get $A\subset I^{(1)}$ and hence $J\subset I^{(1)}$. Therefore $J^d\subset (I^{(1)})^d\subset I^{(d)}$ and the result follows. 
\end{proof}

Finally, in the left column we have {\em commutative Coulomb branch $\Z$-algebra} $\bigoplus {}_{0}\cA^{\hbar=0}_{d}$, to be defined later in Section \ref{sec: adjoint coulomb}. It is a generalization of the commutative Coulomb branch appearing in the work of Braverman, Finkelberg and Nakajima \cite{BFN,BFN2}. It is defined as the convolution algebra in the equivariant Borel-Moore homology of a certain space related to the affine Grassmannnian of $G$, and we postpone its definition to Section \ref{sec: adjoint coulomb}. Here we summarize some of its basic properties. 

\begin{theorem}
\label{thm: properties of coulomb algebra}
The algebras  ${}_{0}\cA^{\hbar=0}_{d}$ have the following properties:

a) For $d=0$, we have ${}_{0}\cA^{\hbar=0}_{0}=\C[T^\vee \times \ft]^W$.

b) For $d=1$, we have ${}_{0}\cA^{\hbar=0}_{1}=A$.

c) For all $d$ the module  ${}_{0}\cA^{\hbar=0}_{d}$ is a free $\C[\ft]^W$-submodule of $\eee_d\C[T^\vee \times \ft]$.

d) For $G=\GL_n$, we have ${}_{0}\cA^{\hbar=0}_{d}=({}_{0}\cA^{\hbar=0}_{1})^d=A^d$ for all $d$.
\end{theorem}

We prove Theorem \ref{thm: properties of coulomb algebra} in Section \ref{sec: proof properties coulomb algebra}.
Note that part (b) of the theorem yields the map $A\to {}_{0}\cA^{\hbar=0}_{1}$ and hence a family of maps $A^d\to {}_{0}\cA^{\hbar=0}_{d}$. These are denoted by $a_d$ in the commutative diagram \eqref{diagram of algebras}.

\begin{corollary}
\label{cor: rank 1}
If $G$ has semisimple rank 1 then $A^d\simeq {}_{0}\cA^{\hbar=0}_{d}\simeq \eee_dI^{(d)}$ for all $d$.
\end{corollary}

\begin{proof}
By Theorem \ref{thm: properties of coulomb algebra}(d) we get $A^d\simeq {}_{0}\cA^{\hbar=0}_{d}$. On the other hand,
in semisimple rank 1 we have only one codimension 2 hyperplane and it is easy to see that $I^{(d)}=J^d$, so by Lemma \ref{lem: symmetrizer power J} we get
$$
\eee_dI^{(d)}=\eee_dJ^d=A^d.
$$
\end{proof}

The following is one of the main theorems of this section, identifying the geometrically constructed graded algebra with the symmetrization of the symbolic Rees algebra.
\begin{theorem}
\label{thm: coulomb is symbolic}
We have an isomorphism of graded algebras $\bigoplus_d {}_{0}\cA^{\hbar=0}_{d}\simeq \bigoplus_d \eee_{d}I^{(d)}$ %for all $d$ 
corresponding to the dotted arrow in the diagram \eqref{diagram of algebras}.
\end{theorem}

\begin{proof}

First, let us prove that there is a natural inclusion ${}_{0}\cA^{\hbar=0}_{d}\to \eee_{d}I^{(d)}$ for all $d$. Indeed, by Theorem \ref{thm: properties of coulomb algebra}(c) we know that ${}_{0}\cA^{\hbar=0}_{d}$ is contained in the $\eee_d$-isotypic component of $\C[T^\vee \times \ft]$,  so it is sufficient to check that   ${}_{0}\cA^{\hbar=0}_{d}\subset I^{(d)}$ or, equivalently, ${}_{0}\cA^{\hbar=0}_{d}\subset \langle 1-\alpha^\vee, y_\alpha\rangle^{d}$ for all $\alpha$.

On the other hand, by Lemma \ref{lem: localization factors through centralizer} the inclusion ${}_{0}\cA^{\hbar=0}_{d}(G)\hookrightarrow \C[T^\vee \times \ft]$ factors through ${}_{0}\cA^{\hbar=0}_{d}(Z_G(t))$ where $Z_G(t)$ is the centralizer of some element $t\in T$. 
We can choose $t$ such that $Z_G(t)$ is a rank one subgroup of $G$ corresponding to $\alpha$. By Corollary \ref{cor: rank 1} we get 
$$
{}_{0}\cA^{\hbar=0}_{d}(G)\subset {}_{0}\cA^{\hbar=0}_{d}(Z_G(t))\simeq \langle 1-\alpha^\vee, y_\alpha\rangle^{d}.
$$
Now we prove that the inclusion is an isomorphism.

Since ${}_{0}\cA^{\hbar=0}_{d}$ is free over $\C[\ft]^W$ by Theorem \ref{thm: properties of coulomb algebra}(c) and  $\eee_{d}I^{(d)}$ is torsion free, it is sufficient to prove that the inclusion is an isomorphism outside of codimension 2 subset. 

Both  $\C[\ft]^W$-modules  ${}_{0}\cA^{\hbar=0}_{d}$ and $\eee_{d}I^{(d)}$ are supported on the union of the root hyperplanes in $\ft/W$. If we specialize to a generic point in one of the hyperplanes, we can replace $G$ by its rank 1 subgroup by Lemma \ref{lem: localization factors through centralizer} as above,
and the isomorphism follows from  Corollary \ref{cor: rank 1}. Therefore the two modules are isomorphic outside of the union of pairwise intersections of hyperplanes, which has codimension 2. 

\end{proof}

We can use the above graded algebras to construct quasi-projective varieties
$$
\tfC_{G^{\vee}}:=\Proj \bigoplus_d {}_{0}\cA^{\hbar=0}_{d}\simeq \Proj \bigoplus_d \eee_d I^{(d)},\ Y_{G}=\Proj \bigoplus_d I^{(d)}.
$$
By the work of Haiman \cite{Haiman} for $G=\GL_n$ we have $\tfC_{G^{\vee}}=\Hilb^n(\C^{\times}\times \C)$ and $Y_G$ is isomorphic to the isospectral Hilbert scheme of $\C^{\times}\times \C$:
$$
\begin{tikzcd}
Y_{\GL_n} \arrow[r]\arrow[d] & (\C^{\times}\times \C)^n \arrow[d]\\
\Hilb^n(\C^{\times}\times \C)\arrow[r] & S^n(\C^{\times}\times \C).\\
\end{tikzcd}
$$
We claim that for a general $G$ the variety $\tfC_{G^{\vee}}$

can be considered as the partial resolutions of the commuting variety which we identify with $T^*T^\vee/W$.
By Remark \ref{rem: finite generation} the algebra $\bigoplus_d {}_{0}\cA^{\hbar=0}_{d}$ is finitely generated, hence the natural projection $\tfC_{G^{\vee}}\to \Spec {}_{0}\cA^{\hbar=0}_{0}=T^*T^\vee/W$ is projective.

\begin{remark}
In \cite{Gin}, Ginzburg defines and studies the {\em isospectral commuting variety} for general $G$ as a certain reduced fiber product. On the other hand, the variety $Y_G=\Proj \bigoplus_{d\geq 0} I_G^{(d)}$ is another candidate for the isospectral commuting variety. It is natural to wonder how the two constructions are related.
\end{remark}

\begin{remark}
In \cite{GiKa}, Ginzburg-Kaledin prove that there are no crepant resolutions of $T^*\ft/W$ for $W$ outside types $A,B,C$. Their definition of symplectic resolution includes the crepant condition, so their statement is non-existence of symplectic resolutions. 
This non-existence of a symplectic resolution is thus likely the case for $T^*T/W$ as well. Therefore, we cannot expect $\tfC_{G^{\vee}}$ to be smooth  outside types $A,B,C$.

We note however that from the results of \cite[Proposition 2.8]{BFM}, it follows that $T^*\ft/W$ and $T^*T/W$ admit birational maps from the universal centralizer 
schemes appearing in Theorem \ref{thm:universalcentralizer as coulomb}, which are smooth for simply connected groups in any type.  More geometrically, we can think of universal centralizer schemes as smooth  open subsets in $\tfC_{G^{\vee}}$.

\end{remark}

\begin{proposition}
\label{prop: Y normal}
$Y_G$ is normal.
\end{proposition}
\begin{proof}
We will prove the homogeneous coordinate ring $\bigoplus_{d=0}^{\infty} I^{(d)}$ is integrally closed. This is essentially a restatement of \cite[Theorem 4.27]{Kivinen}. Indeed,  for a given $\alpha$ the sequence $1-\alpha^\vee, y_\alpha$ is regular. 
Therefore the graded algebra $\bigoplus_{d=0}^{\infty}\langle 1-\alpha^\vee, y_\alpha\rangle ^d$ is integrally closed, and intersection preserves integral closedness. We can write
$$
\bigoplus_{d=0}^{\infty} I^{(d)} =\bigoplus_{d=0}^{\infty} \bigcap_{\alpha\in \Phi^+}\langle 1-\alpha^\vee, y_\alpha\rangle ^d=
 \bigcap_{\alpha\in \Phi^+}\bigoplus_{d=0}^{\infty}\langle 1-\alpha^\vee, y_\alpha\rangle ^d,
$$
so the symbolic blow-up $\Proj \bigoplus_{d=0}^{\infty} I^{(d)}$ considered here is integrally closed. 
\end{proof}
\begin{corollary}
\label{cor:normality}
$\tfC_{G^{\vee}}$ is normal.
\end{corollary}
\begin{proof}
Recall that we have a natural action of $W$ on  $\bigoplus_{d=0}^{\infty} I^{(d)}$ by algebra automorphisms, where the action on $I^{(d)}$ is twisted by $\epsilon^d$. This defines an action of $W$ on the variety $Y_G$, and we get
$$
\tfC_{G^{\vee}}\simeq Y_G/W.
$$
If a normal variety $Y$ is acted upon by a finite group $\Gamma$, $Y/\Gamma$ is normal \cite[Chapter II.5, top of page 128]{Shafarevich} (note that Shafarevich assumes that $Y$ is affine, but the argument works for any variety since this is a local property). By Proposition \ref{prop: Y normal} $Y_G$ is normal, so $\tfC_{G^{\vee}}$ is normal as well.

\end{proof}

\begin{remark}
An alternative proof of normality of $\tfC_{G^{\vee}}$ follows from \cite[Theorem 14]{Weekes2}.
\end{remark}
\begin{remark}
So far in this section, we have been discussing the $\hbar=0$ case, i.e. the non-quantized algebra. Fix some symplectic $\Q$-factorial terminalization of the variety $T^*T^\vee/W$, as constructed in \cite{BCHM,Namikawa}. Denote it by $\tX_{G^{\vee}}$.
The formal Poisson deformations of $\tX_{G^{\vee}}$ are parameterized by $H^2(\tX^{\reg}_{G^{\vee}},\C)$. This construction works for the rational (as opposed to trigonometric) version, giving $\tX_{\fg^{\vee}}\to \ft\oplus \ft^*/W$ where $\tX_{\fg^{\vee}}$ can be identified with the $\Q$-factorial terminalization constructed in \cite{LosevDerived}. The variety $\tX_{\fg^{\vee}}$ is a conical symplectic partial resolution with $\Q$-factorial singularities of $\ft\oplus \ft^*/W$ such that $\codim(\tX_{\fg^{\vee}}-\tX_{\fg^{\vee}}^{\reg})\geq 4$. By the results of Namikawa in \cite{Namikawa}, this implies that $\tX_{\fg^{\vee}}$ is terminal. In this case, it has been proved in \cite{Namikawa,LosevDerived} that the filtered quantizations of $\tX_{\fg^{\vee}}$ are also parametrized by $H^2(\tX^{\reg}_{\fg^{\vee}},\C)$ and correspond rational Cherednik algebras with different parameters. Cherednik algebras were already mentioned in the introduction and their trigonometric version will be studied in Section \ref{sec:cherednik} below.

Unlike the formal Poisson deformations, we do not know if the filtered quantizations for the trigonometric variety $\tX_{G^{\vee}}$ are parametrized by $H^2(\tX^{\reg}_{G^{\vee}},\C)$ or whether they correspond to degenerate DAHA, since the results in the Lie algebra case heavily use the fact that $\ft\oplus \ft^*/W$ has {\em conical} symplectic singularities.

Finally, we note that similar to \cite[Proposition 2.1]{BALK} the normal intermediate partial resolutions $$\tX_{G^{\vee}}\to X_{G^{\vee}}\to T^*T^\vee/W$$ are classified by faces of the ample cone of $\tX$ such that for a given face $F$ and a rational point $f\in F$ a positive rational multiple of $f$ is the first Chern class of an ample line bundle on the corresponding partial resolution.

It seems reasonable thus to expect that $\fP=H^2(\tX^{\reg}_{G^{\vee}},\C)$ parametrizes both the filtered quantizations of $\tX_{G^{\vee}}$ as well as the partial resolutions between $\tX_{G^{\vee}}\to T^*T^\vee/W$, just as it does in the Lie algebra case.
Further, it seems by Corollary \ref{cor:normality} 
reasonable to expect that the partial resolution of the trigonometric commuting variety constructed in this section, namely $\tfC_{G^{\vee}}$, equals a partial resolution constructed this way, and that the singularities of $\tX_{G^{\vee}}\to T^*T^\vee/W$ are locally modeled on those of $\tX_{\fg^{\vee}}$.

\end{remark}

To support the above remarks, we note the following about the local structure of our algebras. Recall the {\em Borel-de Siebenthal algorithm}, explained e.g. in \cite{EtingofReducibility}. If $G$ is simply connected and $T$ a maximal torus as before, let $\Sigma\subset T$ be be the set of elements $a$ whose centralizer $C_{\fg}(a)$ is semisimple of
the same rank as $\fg$ (here $T^{\vee}$ is identified with the quotient of $\ft^{\vee}$ by some lattice via the exponential map). It is known \cite[Section 2]{EtingofReducibility} that $\Sigma$ is a finite set which is in bijection with the set of vertices in the extended Dynkin diagram of $\fg$. Furthermore, the Dynkin diagram of $C_{\fg}(a)$ for $a\in \Sigma$ is obtained from the extended Dynkin diagram of $\fg$ by deleting the corresponding vertex. We refer to Section \ref{sec:cherednik} and \cite{EtingofReducibility} for the notation on Cherednik algebras.

\begin{lemma}
\label{lem:localstructure}
Suppose $G$ is simply connected. Upon completion at $a\in \Sigma$, the Coulomb branch algebra $({}_i \cA_d^{\hbar})^{\wedge a}\cong \HH_{G,c+i\hbar,\hbar}^{\wedge a}$ is isomorphic to $\HH^{\rat, \wedge 0}_{C_{\fg}(a),c+i\hbar,\hbar}$ for the Lie algebra $C_{\fg}(a)$ coming from the Borel-de Siebenthal algorithm as above.
\end{lemma}

\begin{proof}
This is contained in \cite[Proof of Theorem 3.2]{EtingofReducibility}.
\end{proof}

\subsection{The universal centralizer}

In the above, we have defined the partial resolution $\tfC_{G^{\vee}}$ using $\Proj$ construction, and have  limited understanding of its geometry outside of type $A$. Nevertheless, in this subsection we define an affine open chart in $\tfC_{G^{\vee}}$ and prove that it  coincides with the trigonometric version of the {\em universal centralizer} of \cite{BFM, Ngo}.
 It also appears as a Coulomb branch for zero matter, and will be used later in Section \ref{sec:finitegeneration}.

We let $G$ be arbitrary for now. Let $\Delta:=\prod_{\alpha\in \Phi^+}y_{\alpha}\in  A_G$ be the Vandermonde determinant.
\begin{definition}
We define $$U_{\Delta}\subset \widetilde{\Comm}_{G^{\vee}}=\Proj \bigoplus_{d=0}^{\infty}{}_0\cA_{d}^{\hbar=0}$$   as the distinguished open subset given by the element $\Delta\in A_G\simeq {}_0\cA_{1}^{\hbar=0}$. 
\end{definition}

By definition and Theorem \ref{thm: coulomb is symbolic}, $U_{\Delta}$ is the affine variety whose coordinate ring is the degree zero part of the localization 
 of $\bigoplus_{d=0}^{\infty}{}_{0}\cA_{d}^{\hbar=0}\simeq \bigoplus_{d=0}^{\infty}\eee_d I^{(d)}$ in $\Delta$:
$$
\C[U_{\Delta}]=\mathrm{Span}\left\{\frac{f}{\Delta^d}\ :\ f\in \eee_d I^{(d)}\right\}/\left(
\frac{f}{\Delta^d}\sim \frac{f\Delta}{\Delta^{d+1}}
\right).
$$ 

\begin{remark}
Note that $U_{\Delta}$ is different from the preimage of $U'=\{\Delta\neq 0\}\subset T^*T^\vee/W\simeq \Spec {}_0\cA_{0}^{\hbar=0}$ under the natural projection $\pi:\widetilde{\Comm}_{G^{\vee}}\to T^*T^\vee/W$. 

Indeed, let $Z=\{\Delta=0\}\subset T^*T^\vee/W$. Then $\pi^{-1}(U')$ is the complement of the {\bf total transform} of $Z$ in $\widetilde{\Comm}_{G^{\vee}}$ while $U_{\Delta}$ is the complement of the {\bf strict transform} of $Z$. So $U_{\Delta}\supsetneq U'$. 
\end{remark}

We now describe this chart. 
In \cite{BFM}, two trigonometric versions of the universal centralizer are studied. The one of interest to us is defined as follows, see {\em loc. cit.} for more details.
\begin{definition}
The universal centralizer of $G^{\vee}$ is the variety 
$$\mathfrak{B}_{\fg^{\vee}}^{G^{\vee}}:=\{(g,s)\in G^{\vee}\times \fg^{\vee}|\ad_g(s)=s,\; s \text{ is regular } \}\sslash G^{\vee}$$
\end{definition}

\begin{remark}
In \cite{BFM}, this variety is denoted $\mathfrak{Z}_{\fg^{\vee}}^{G^{\vee}}$.  
There is also another version of the trigonometric universal centralizer $\mathfrak{Z}_{G^{\vee}}^{\fg^{\vee}}$ with the roles of $\fg, G$ swapped. It has the nicer geometric property of being symplectically isomorphic to 
$T^*(T^\vee/W)$ when $G$ is adjoint (so $G^{\vee}$ is simply connected).
\end{remark}

In \cite{BFM} explicit description of the coordinate ring of $\mathfrak{Z}_{\fg^{\vee}}^{G^{\vee}}$ is given.  We also have the following Coulomb branch description of $\mathfrak{Z}_{\fg^{\vee}}^{G^{\vee}}$.
\begin{theorem}[\cite{BFM}]
\label{thm:universalcentralizer as coulomb}
We have an isomorphism of algebras:
$$
\C[\mathfrak{B}_{\fg^{\vee}}^{G^{\vee}}]\cong H_*^{G(\cO)}(\Gr_G)
$$
The right hand side can be identified with  the Coulomb branch for $(G,0)$ (see Section \ref{sec:coulomb}.)
\end{theorem}
We now state and prove the main theorem of this section. 

\begin{theorem}\label{thm:U_del}
We have $$\C[\mathfrak{B}_{\fg^{\vee}}^{G^{\vee}}]\cong 
\C[U_{\Delta}].$$
In particular, there is a natural isomorphism $\mathfrak{B}_{\fg^{\vee}}^{G^{\vee}}\cong U_\Delta$.
\end{theorem}
\begin{proof}
We recall and slightly rephrase 
\cite[Section 4]{BFM}\footnote{We follow the arxiv version which is more complete than the published one.} .
First, consider the Lie algebra version 
$$\mathfrak{B}_{\fg^{\vee}}^{\fg^{\vee}}:=\{(g,s)\in \fg^{\vee}\times \fg^{\vee}|\ad_g(s)=s,\; s \text{ is regular } \}\sslash G^{\vee}$$
On the other hand, we consider the family of ideals $I^{(d)}_{\ft}:=\bigcap_{\alpha\in\Phi^+}\langle x_{\alpha^{\vee}}, y_\alpha\rangle^{d}\subset \C[\ft^*\times \ft]$.
We claim that the algebra of functions $\C[\mathfrak{B}_{\fg^{\vee}}^{\fg^{\vee}}]$ can be obtained as the degree zero localization of the graded algebra $\bigoplus_{d}\eee_d I^{(d)}_{\ft}$.

Indeed, let $P_1,\ldots,P_r$ denote some generators in $\C[\fg^{\vee}]^{G}\simeq \C[\ft^{\vee}]^W$. The one-forms 
$dP_1,\ldots,dP_r\in  \Omega^1(\fg^{\vee})\simeq \C[\fg^{\vee}]\otimes (\fg^{\vee})^*$ can be identified with some $G$-invariant polynomial maps $m_1,\ldots,m_r:\fg^{\vee}\to \fg^{\vee}$. By \cite[Lemma 4.4]{BFM} for regular $s$ the elements $m_i(s)$ form a basis in the centralizer of $s$, so we can write 
\begin{equation}
\label{eq: def b_i}
g=\sum_{i=1}^{n} \psi_i(g,s) m_i(s).
\end{equation}
The coordinate ring of $\mathfrak{B}_{\fg^{\vee}}^{\fg^{\vee}}$ is a polynomial ring in $\psi_i$ and $P_i(s)$. 

For example, for $\fg=\mathfrak{gl}_n$ we get that $g$ and $s$ are two commuting $n\times n$ matrices and $s$ is regular, hence $g$ is a matrix polynomial in $s$:
$$
g=\psi_1+\psi_2 s+\ldots+\psi_{n} s^{n-1}.
$$
Here we identify the polynomial maps $m_i(s)=s^{i-1}$.

To compare this with the localization in $\Delta$, we can restrict \eqref{eq: def b_i} to $\ft^*\times \ft^*$ and 
abbreviate $g=\diag(\overline{x}),s=\diag(\overline{y})$.
We interpret \eqref{eq: def b_i} as a linear system of equations on $\psi_i$ for given $g$ and $s$ (resp. $\overline{x}$ and $\overline{y}$), for example for $\fg=\mathfrak{gl}_n$  we get
$$
x_i=\psi_1+\psi_2 y_i+\ldots+\psi_{n} y_i^{n-1}.
$$
 Note that the $W$ acts on $g$ and $s$ (resp. on $\overline{x}$ and $\overline{y}$) but fixes $\psi_i$ by construction. By \cite[Lemma 4.4]{BFM} the determinant of this linear system equals $m_1(s)\wedge \cdots \wedge m_r(s)$ which is proportional to $\Delta$ up to a nonzero constant factor.

Therefore by Cramer's Rule we can write $\psi_i=D_i/\Delta$ for some polynomials $D_i\in \C[\ft^{\vee}\times \ft^{\vee}]$. Since $\psi_i$ are $W$-invariant, we get that $D_i$ are $W$-antiinvariant and hence contained in $A_{\ft}\subset \eee_1 I^{(1)}_{\ft}$. We get $\psi_i\in  \Delta^{-1}\eee_1 I^{(1)}_{\ft}$ and $P_i(s)\in \eee I^{(0)}_{\ft}$, so
$$
\C[\mathfrak{B}_{\fg^{\vee}}^{\fg^{\vee}}]\subset \bigoplus_{d=0}^{\infty} \Delta^{-d}\eee_d I^{(d)}_{\ft}/\sim.
$$

To prove the reverse inclusion, we claim that any element in $\Delta^{-d}\eee_d I^{(d)}_{\ft}$ can be written as a polynomial in $\psi_i$ and $P_i(\overline{y})$ or, equivalently, a $W$-invariant polynomial in $\overline{y}$ and $\psi_i$. Indeed, given a polynomial $f(\overline{x},\overline{y})\in \eee_d\bigcap_{\alpha\in  \Phi^+}\langle x_{\alpha^{\vee}},y_{\alpha}\rangle ^d$, we can substitute $\overline{x}$ using \eqref{eq: def b_i} and write $f$ as a polynomial $h(\psi_i,\overline{y})$.  In this presentation, $x_{\alpha^{\vee}}$ (as a polynomial in $\psi_i,\overline{y}$) is actually divisible by $y_{\alpha}$, so any polynomial in  
$(x_{\alpha^{\vee}},y_{\alpha})^d$ is divisible by $y_{\alpha}^d$. Therefore $h$ is divisible by $\prod_{\alpha\in \Phi^+} y_{\alpha}^d=\Delta^d$ and $\Delta^{-d}h$ is a  polynomial in $\overline{y}$ and $\psi_i$. Finally, since both $f$ and $\Delta^d$ transform by the same sign under $\eee_d$, we conclude that $\Delta^{-d}h$ is a $W$-invariant polynomial in $\overline{y}$ and $\psi_i$.  To sum up, $\Delta^{-d}\eee_d I^{(d)}_{\ft}\subset \C[\mathfrak{B}_{\fg^{\vee}}^{\fg^{\vee}}]$.

The group theoretic version follows from this computation as in \cite[Section 4]{BFM}. In particular, for $G=\GL_n$ we simply require that $g$ is invertible and define additional variables $ \psi^*_i$ such that
\begin{equation}
\label{eq: def b_i dual}
g^{-1}=\sum_{i=1}^{n} \psi^*_i(g,s) m_i(s).
\end{equation}
Similarly, after restricting to the torus $x_i$ are invertible and we can solve both \eqref{eq: def b_i} and \eqref{eq: def b_i dual}
by Cramer's Rule. The rest of the proof proceeds verbatim.

\end{proof}

\begin{remark}
We would like to caution the reader that Theorem \ref{thm:U_del} is similar to \cite[Proposition 2.8]{BFM}, but slightly different from it in the following way. 

Recall that the ideal $I^{(1)}$ defines a union of codimension 2 subvarieties in $T^*T^{\vee}$. In \cite{BFM} the authors (up to a quotient by $W$) consider an open subset in the blow-up of $T^*T^{\vee}$ along $I^{(1)}$. By definition, the latter is given as $\Proj$ of a graded algebra built from {\bf powers} of $I^{(1)}$. Here, we instead consider the {\bf symbolic blow-up} using {\bf symbolic powers} $I^{(d)}$ of  $I^{(1)}$. 

In general, we get a homomorphism
$$
\bigoplus_d \eee_d(I^{(1)})^{d}\to \bigoplus_d \eee_d I^{(d)}
$$
which implies a map $\Proj \bigoplus_d \eee_d I^{(d)}\to \mathrm{Bl}_{I^{(1)}}T^*T^{\vee}$, but we do not expect it to be an isomorphism.  Theorem \ref{thm:U_del} shows that it is an isomorphism over $U_{\Delta}$.

\end{remark}

\begin{remark}
We can interpret the proof of Theorem \ref{thm:U_del} as follows. 
There are natural embeddings $${}_j\cA_i^{\hbar=0}\hookrightarrow H_*^{G(\cO)}(\Gr_G)$$ coming from the construction as Coulomb branches, see for example Section \ref{subsec:spherical adjoint localization} or \cite[Lemma 5.11]{BFN}. On the algebra side, these embeddings realize rational functions of the form $f(x,y)/\Delta^{j-i}$, where $f\in {}_j\cA^{\hbar=0}_i$, as functions on the open chart $U_\Delta$.

When $G=\GL_n$, this construction is closely related to the construction of the open chart ''$U_{(1^n)}$" on $\Hilb^n(\C^2)$ given by Haiman in \cite[Corollary 2.7.]{Haimanqt}.

We also note that the comparison between geometry and algebra explains the appearance of the idempotent $\eee_d$. Indeed, from the geometric point of view the algebra ${}_0\cA_d^{\hbar=0}$ is, in the notations of Section \ref{sec:coulomb}, simply $H^{G(\cO)}_*({}_0\cR_d)\cong\eee H_*^{T(\cO)}({}_0 \cR_d)$. However, the localization formula in Corollary \ref{thm:localization for any group any weight} involves a prefactor $\Delta^d$, exactly corresponding to twisting by the $d$-th power of the sign representation of $W$.
\end{remark}

\subsection{Explicit antisymmetric polynomials}
\label{sec: explicit schur}

In Theorem \ref{thm:Z-algebra-iso} we will need an explicit construction of a $\C$-basis in the space $A$ of antisymmetric ($S_n$-antiinvariant) polynomials for $G=\GL_n$, in order to compare our Coulomb branch construction with the one above. The exposition follows ideas of Haiman in \cite{Haimanqt}. The reader is advised to skip this section on a first reading. 

We denote by $\Alt$ the action of the antisymmetric projector $\eee_-$ on polynomials.
Let $S=\{(a_1,b_1),\ldots (a_n,b_n)\}$ be an arbitrary $n$-element subset of $\Z_{\ge 0}\times \Z$. We define 
$$
\Delta_S(y_1,u_1,\ldots,y_n,u_n)=\Alt\left(y_1^{a_1}u_1^{b_1}\cdots y_n^{a_n}u_n^{b_n}\right)=\frac{1}{n!}\det\left(y_i^{a_j}u_i^{b_j}\right).
$$
For a composition $\alpha$ with $\sum \alpha_i=n$, we can consider the set 
$$
S_{\alpha}=\{(0,0),\ldots ,(\alpha_1-1,0),(0,1),\ldots (\alpha_2-1,1),\ldots\}
$$
and denote $\Delta_{S_{\alpha}}=\Delta_{\alpha}$.
In particular,
$$
\Delta=\prod_{i<j}(y_i-y_j)=\Delta_{(n)}.
$$
Given a composition $\alpha$, write $\lambda(\alpha)=(0^{\alpha_1},1^{\alpha_2},\ldots)$. More generally,  for any subset $S=\{(a_i,b_i)\}\subset \Z_{\ge 0}\times \Z, |S|=n$, we define $\lambda(S)=\mathrm{sort}(b_i)$ as the vector in $\Z^{n}$ with coordinates obtained by sorting $b_i$ in the  non-decreasing order. Clearly, 
$\lambda(S_{\alpha})=\lambda(\alpha)$. Furthermore, we define a collection of subsets
$$
S_k=\{a_i\ :\ (a_i,k)\in S\}=\left\{a_{i_1^{(k)}}<\ldots<a_{i_{r_k}^{(k)}}\right\}
$$
and a partition
$$
\mu_k(S)=(a_{i_1^{(k)}},a_{i_2^{(k)}}-1,\ldots,a_{i_{r_{k}}^{(k)}}-r_{k}+1).
$$
Finally, we define $\yy_k=(y_{i}\ :\ \lambda_i=k)$.

\begin{lemma}
\label{lem: Haiman dets}
(a) The determinants $\Delta_S$ span the vector space $A$.

(b) We have the following formula for the determinant $\Delta_{\alpha}$:
$$
\Delta_{\alpha}=c\cdot \Alt\left[u^{\lambda}\prod_{r<s,\lambda_{r}=\lambda_{s}}(y_r-y_s)\right].
$$
where $\lambda=\lambda(\alpha)$.

(c) More generally, we have the following formula for the determinant $\Delta_{S}$:
$$
\Delta_{S}=c\cdot \Alt\left[\prod_{k}s_{\mu_k}(\yy_k)u^{\lambda}\prod_{r<s,\lambda_{r}=\lambda_{s}}(y_r-y_s)\right]
$$
where $s_{\mu_k}$ are Schur polynomials, $\lambda=\lambda(S)$ and $c$ is some  nonzero scalar factor depending on the size of stabilizer of $\lambda$.
\end{lemma}

\begin{proof}
(a) The space $\C[T^*T^{\vee}]$ is spanned by the monomials $y_1^{a_1}u_1^{b_1}\cdots y_n^{a_n}u_n^{b_n}$, so $A=\eee_{-}\C[T^*T^{\vee}]$ is spanned by their antisymmetrizations (recall that $u_i$ are invertible, so $b_i$ are allowed to be negative). If some of pairs $(a_j,b_j)$ coincide, then the antisymmetrization vanishes, so it is sufficient to assume that $(a_j,b_j)$ are pairwise distinct and form an $n$-element subset $S$.

Clearly, (b) follows from (c) since for $S=S_{\alpha}$ we have $S_k={0,1\ldots,\alpha_k-1}$ and $\mu_k=(0)$.

To prove (c), observe that the function $\Delta_S$ is antisymmetric and all possible monomials in $u$ are in the $S_n$-orbit of $\lambda$, so it is sufficient to compute the coefficient at $u^{\lambda}$. This coefficient is proportional to
$$
\Alt_{\Stab(\lambda)}(y_1^{a_1}\cdots y_n^{a_n})=\prod_{k}\Alt_{S_{k}}\left[\prod_{\lambda_i=k} y_i^{a_i}\right]=\prod_{k}\left[s_{\mu_k}(\yy_k)\cdot \prod_{r<s,\lambda_r=\lambda_s=k}(y_r-y_s)\right].
$$
\end{proof}

\begin{example}
For $\alpha=(1,2,1)$ we have $S=\{(0,0),(0,1),(1,1),(0,2)\}, \lambda=(0,1,1,2)$ and 
$$
\Delta_S=\Alt\left[(y_2-y_3)u_1^0u_2^1u_3^1u_4^2\right]
$$
For $S=\{(5,0),(3,1),(7,1),(2,2)\}$ we have $\lambda=(0,1,1,2),\mu_0=(5),\mu_1=(3,6),\mu_2=(2)$ and
$$
\Delta_S=\Alt\left[s_5(y_1)s_{6,3}(y_2,y_3)s_2(y_4)(y_2-y_3)u_1^0u_2^1u_3^1u_4^2\right].
$$
Note that $s_5(y_1)=y_1^5, s_2(y_4)=y_4^2$ and 
$$
s_{6,3}(y_2,y_3)=\frac{\vline\begin{matrix}y_2^7 & y_3^7\\ y_2^3 & y_3^3\end{matrix}\vline}{y_2-y_3}=-\frac{\Alt(y_2^3y_3^7)}{y_2-y_3}.
$$
\end{example}

\section{Trigonometric Cherednik algebra}
\label{sec:cherednik}
\subsection{Definitions}

We define the extended torus $\widetilde{T}=T\times  \mathbb{C}^{\times} $  and the corresponding Lie algebra $\widetilde{\ft}=\ft\oplus \C_{\hbar}$.  
The extended affine Weyl group $\tW:=W \ltimes X_*(T)$ is generated by the affine Weyl group $W^{\mathrm{aff}}=W\ltimes Q^\vee$ and an additional abelian group $\Omega=X_*(T)/Q^\vee$ where $Q^\vee$ is the coroot lattice. In general, $\Omega$ is the fundamental group of $G$, for example for $GL_n$, we have $\Omega\cong \Z$. We use action of \(\tW\) on \(\widetilde{\ft}^\vee\)
depending on $\hbar$.
The action of $w\in \tW$ on $\xi\in\ft$ will be denoted by ${}^{w}\xi$. We will denote the longest element in $W$ by $w_0$. 

\begin{definition}
\label{def:dDAHA}
The {\em trigonometric DAHA} of $G$ is the $\C[\hbar,c]$-algebra, which as a vector space looks like 
$$\HH_G=\HH_{c,\hbar}=\C[\tW]\otimes \C[\ft] \otimes \C[c]\otimes \C[\hbar]$$
and the algebra structure is determined as follows:
\begin{enumerate}
\item Each of the tensor factors is a subalgebra, and $c,\hbar$ are central. We denote by $\sigma_i$ the simple reflections in the copy of $\tW\subset \HH_G$.
\item $\sigma_i\xi-{}^{s_i}\xi \sigma_i=c \langle \xi, \alpha_i^\vee\rangle$ for all simple reflections $\sigma_i \in \tW$ and $\xi \in \widetilde{\ft}^\vee\subset \C[\widetilde{\ft}]$\footnote{In particular, for \(\xi\in\ft^\vee\) we have \({}^{s_0}\xi=s_0(\xi)+\hbar\langle\xi,\alpha_0^\vee\rangle\).  }.
\item For any $\omega \in \Omega\subset \tW$ and $\xi\in \widetilde{\ft}^\vee$,
$\omega \xi={}^\omega \xi \omega$
\end{enumerate}
\end{definition}

 For adjoint groups, the group \(\Omega\) is the group of symmetries of the fundamental domain of the action on
  \(\widetilde{\ft^\vee}\) of the group generated by the reflections \(s_i\), \(i=0,\dots,\dim(\ft)=r\).
  Thus the elements of \(\Omega\) are affine transformations of \(\widetilde{\ft^\vee}\) that permute
the set \((\alpha_0,1),(\alpha_1,0),\dots,(\alpha_r,0)\).

 The algebra $\HH_G$ is graded as follows: all generators of $\widetilde{W}$ have degree zero, while $\xi\in \ft^\vee\subset \C[\ft]$ have degree 1 (so that $\C[\ft]$ has a standard grading). The generators $c$ and $\hbar$ both have degree 1 as well. One can check that the above relations are homogeneous with respect to this grading.

 \begin{example}
For $G=\GL_n$ the group $\tW$ is generated by simple reflections $\sigma_1,\ldots,\sigma_{n-1}$ (which generate $W$) and an additional element $\pi$ (which generates $\Omega$). The lattice part of $\tW$ is generated by
$$
X_i=\sigma_{i-1}\cdots \sigma_1\pi\sigma_{n-1}\cdots \sigma_i,\quad i=1,\dots,n.
$$
The algebra $\HH_{c,\hbar}$ is generated by $\sigma_i,\pi$ and commuting variables $y_1,\ldots,y_n$, with the following relations:
$$
\sigma_iy_i=y_{i+1}\sigma_i-c, \sigma_iy_{i+1}=y_i\sigma_i+c,\ \sigma_iy_j=y_j\sigma_i\ (j\neq i,i+1),\quad i=1,\dots,n-1,
$$
$$
\pi y_i =y_{i+1}\pi (1\le i\le n-1),\ \pi y_n=(y_1+\hbar)\pi. 
$$

\end{example}

\begin{remark}
We will also use specializations of this algebra, to be called trigonometric DAHAs as well when there is no risk of confusion. Let us explain how this relates to parameter conventions in the literature. In \cite{OY}, the trigonometric DAHA is defined as above but with parameters $\delta=\hbar, c=u$. These correspond to the generators of $H^*_{\G_m^{\mathrm{rot}}}(*) $ and $H^*_{\G^{\mathrm{dil}}_m}(*)$, respectively. In \cite{OY}, the relations $c+\nu \delta = c+\nu \hbar=0$  and $\hbar=1$ are imposed for $\nu \in \C$. This specialization of parameters is often called ''the" trigonometric DAHA with parameter $\nu$. We will denote it by $\HH_{\nu}$. If we want to emphasize the role of $G$ instead of the parameters, we will write $\HH_G$ for any of the specializations of the Cherednik algebra.
\end{remark}
\begin{remark}
It is common in Cherednik algebra literature to specialize $\hbar$ to $1$ as above, so that the algebra $\HH_{c,\hbar}^{\hbar=1}$ admits a natural filtration in powers of $\hbar$, making the full $\HH_{c,\hbar}$ the Rees construction for this filtration.  In this language, $\HH_{c,\hbar}^{c=\hbar=0}$ is the associated graded of $\HH_{c,\hbar}^{\hbar=1}$, since $\HH_{c,\hbar}$ is flat over $\C[c,\hbar]$. With this in mind, we will use the specialization $c=\hbar=0$ and the associated graded interchangeably.
\end{remark}
\begin{remark}
Later on, we shall be interested in the family of Cherednik algebras $\HH_{c+i\hbar,\hbar}$ for $i\leq 0$ as well, and the specializations $c+(\nu+i)\hbar=0,\hbar=1,$ i.e. $c=-\frac{m+in}{n}$.
\end{remark}

We introduce the symmetrizer $\eee=\frac{1}{|W|}\sum_{w\in W} w$ and the antisymmetrizer $\eee_{-}=\frac{1}{|W|}\sum_{w\in W} (-1)^{\ell(w)} w$ in the group algebra of $W$. We define the {\em spherical} and {\em antispherical} subalgebras in $\HH_G$  as $\eee \HH_G\eee$ and 
$\eee_{-} \HH_G\eee_{-}$. 

Note that in the specialization $c=\hbar=0$ the structure of the algebra simplifies dramatically: $\HH_G^{c=\hbar=0}=\C[W]\ltimes  \C[T^\vee\times \ft]$, so 
$$
\eee \HH_G^{c=\hbar=0}\eee\cong \C[T^\vee \times \ft]^{W}.
$$
We will refer to this as the {\em commutative limit}, although this only gives the limit of the spherical subalgebra the structure of a commutative algebra.

The algebra $\HH_G$ has a representation 
\begin{equation}
\label{eq:Dunklembedding}
\HH_{c,\hbar}\hookrightarrow \text{Diff}_{\hbar}(\ft^{\reg})\rtimes \C[W]
\end{equation} defined e.g. in 
\cite[Section 2.13]{CherednikBook}.  Here $\text{Diff}_{\hbar}(\ft^{\reg})$ is the algebra of $\hbar-$difference operators on the Lie algebra $\ft$, possibly with poles along the root hyperplanes. In this representation, the generators of $\HH_G$ corresponding to simple reflections act by 
$$
\sigma_i=s_i+\frac{c}{y_{\alpha_i}}(s_i-1)
$$
and the lattice part of extended affine Weyl group acts in the standard way (this determines the action of $\Omega$).
Below we identify \(\HH_{c,\hbar}\) with its image inside
  \(\text{Diff}_{\hbar}(\ft^{\reg})\rtimes \C[W]\).

\begin{example}
For $G=\GL_n$, we get $\sigma_i=s_i+\frac{c}{y_i-y_{i+1}}(s_i-1)$ and $\pi\cdot f(y_1,\ldots,y_n)=f(y_2,\ldots,y_n,y_1+\hbar)$.
\end{example}

\subsection{Shift isomorphism}

We will need several involutions on the algebra $\HH_c$.

\begin{lemma}
The map $\Psi:w\mapsto (-1)^{\ell(w)}w_0w^{-1}w_0, \xi\mapsto {}^{w_0}\xi, \hbar\to -\hbar$ for $w\in \tW$ defines an involutive anti-automorphism of $\HH_G$.
\end{lemma}

\begin{proof}
Suppose that $w_0$ sends the simple root $\alpha_i$ to $-\alpha_j$ for some $j$, then $w_0s_i=s_jw_0$. The map $\Psi$ sends $\sigma_i$ to $-\sigma_j$, so we get:
$\Psi(\sigma_i\xi)=-{}^{w_0}\xi\sigma_j$, $\Psi({}^{s_i}\xi\sigma_i)=-\sigma_j{}^{w_0s_i}\xi$ and 
$$
\Psi(\sigma_i\xi-{}^{s_i}\xi\sigma_i)=\sigma_j{}^{s_jw_0}\xi-{}^{w_0}\xi\sigma_j
$$
while
$\langle \xi,\alpha_i^{\vee}\rangle=\langle {}^{s_jw_0}\xi,\alpha_j^{\vee}\rangle$, so the equation (2) is preserved. For the equation (3), observe that $w_0\omega^{-1}w_0=\omega$ since we assume that $\hbar\mapsto -\hbar$.
\end{proof}

\begin{example}
For $G=\GL_n$ we have 
$$
\Psi(\sigma_i)=-\sigma_{n-i},\ \Psi(y_i)=y_{n+1-i},\ \Psi(\pi)=\pi,\ \Psi(\hbar)=-\hbar
$$
\end{example}

Observe that $\Psi(\eee)=\eee_{-}$ and $\Psi(\eee_{-})=\eee$. In particular, $\Psi$ exchanges spherical subalgebra $\eee \HH_c \eee$ with the antispherical subalgebra $\eee_{-} \HH_{c} \eee_{-}$.

Using the difference representation \eqref{eq:Dunklembedding}, we can embed the spherical subalgebra 
$\eee \HH_c \eee$ into the algebra of $W$-twisted difference operators \(\text{Diff}_{\hbar}(\ft^{\reg})\). 
We will denote by $u^{\mu}$ is the translation by $\hbar\mu$. Any difference operator in $\eee \HH_c \eee$ can be written (up to symmetrization by $\eee$) as a linear combination of $u^{\mu}$ with coefficients in $\C[\ft^{\reg}]$. We filter the difference operators by the span of $u^{\mu}$ such that $\mu$ is in a $W$-orbit if a dominant coweight $\mu'$ with $\mu'\leq \lambda$. 

Below we will need some explicit generators for $\eee \HH_c \eee$ written as difference operators, up to lower order terms in this filtration.

\begin{theorem}
\label{thm:E}

a) Let $\lambda$ be a minuscule dominant coweight, respectively \(X^\lambda\in \widetilde{W}\), then 
$$
E_{\lambda,c}:=\eee X^{\lambda} \eee = \eee \prod_{\alpha(\lambda)=1} \frac{y_{\alpha}-c}{y_{\alpha}} u^{\lambda} \eee
$$ 
and 
$$
F_{\lambda,c}:=\eee X^{-\lambda} \eee = \eee \prod_{\alpha(\lambda)=1} \frac{y_{\alpha}+c}{y_{\alpha}} u^{-\lambda} \eee
$$
Here $u^{\lambda}$ is the translation by $\hbar\lambda$. 

b) For an arbitrary dominant coweight $\lambda$ we get
\begin{equation}
\label{eq: E lambda c}
E_{\lambda,c}=\eee \prod_{\alpha(\lambda)>0}\prod_{\ell=0}^{\alpha(\lambda)-1} \frac{(y_{\alpha}+\ell\hbar-c)}{(y_{\alpha}+\ell \hbar)}u^{\lambda}\eee +\mathrm{lower\ order\ terms}.
\end{equation}
If $\lambda$ is minimal in Bruhat order, the formula \eqref{eq: E lambda c} for $E_{\lambda,c}$  is exact.

c) More generally, for dominant coweight $\lambda$ which is minimal in the Bruhat order, and a polynomial $f(y)$ we define 
\begin{equation}
\label{eq: E lambda c f}
E_{\lambda,c}[f]=\eee f(y)\prod_{\alpha(\lambda)>0}\prod_{\ell=0}^{\alpha(\lambda)-1} \frac{(y_{\alpha}+\ell\hbar-c)}{(y_{\alpha}+\ell \hbar)}u^{\lambda}\eee,\ F_{\lambda,c}[f]=\eee  f(y)\prod_{\alpha(\lambda)>0}\prod_{\ell=0}^{\alpha(\lambda)-1} \frac{(y_{\alpha}+\ell\hbar+c)}{(y_{\alpha}+\ell \hbar)}u^{-\lambda} \eee.
\end{equation}
The spherical subalgebra $\eee \HH_c \eee$ is generated by such $E_{\lambda,c}[f]$ and $F_{\lambda,c}[f]$.
\end{theorem}

\begin{remark}
If $\lambda$ is minuscule, then it is indeed minimal in the Bruhat order. The converse is not true: indeed, there are no minuscule coweights at all for root systems $E_8, F_4, G_2$.
\end{remark}

We postpone the proof of Theorem \ref{thm:E} to Section \ref{sec: proof E formula}. Here we use it to relate the spherical and antispherical subalgebras.

\begin{lemma}
\label{lem: Phi and shift}
Suppose that $\lambda$ is a dominant coweight which is minimal in the Bruhat order. Then 
$$\Psi(E_{\lambda,c}[f])\Delta=\Delta E_{\lambda,c-\hbar}[f'],\quad
\Psi(F_{\lambda,c}[f])\Delta=\Delta F_{\lambda,c-\hbar}[f'],
$$
  where $
\Delta=\prod_{\alpha\in \Phi_+}y_{\alpha}
$
and $f'(y)=u^{\lambda}f(y)u^{-\lambda}.$
\end{lemma}

\begin{proof}  We prove the first equation, the second is similar.
Recall that $\Phi$ is an anti-automorphism, so it reverses the order of the factors, and $\Psi(\eee)=\eee_{-}$.
Therefore we  can write:
$$
\Psi(E_{\lambda,c}[f])=\eee_{-} u^{{}^{w_0}\lambda} \prod_{\alpha(\lambda)>0}\prod_{\ell=0}^{\alpha(\lambda)-1} \frac{(y_{\overline{\alpha}}-\ell\hbar-c)}{(y_{\overline{\alpha}}-\ell \hbar)}f({}^{w_0}y)\eee_{-},
$$ 
where we denote $\overline{\alpha}={}^{w_0}\alpha$.
By replacing $\lambda$ by $\lambda^{w_0}$ (since we symmetrize anyway) and $\alpha$ by ${}^{w_0}\alpha$ (since we take product over all roots $\alpha$), we can rewrite this product as  
$$
\Psi(E_{\lambda,c}[f])=\eee_{-} u^{\lambda}\prod_{\alpha(\lambda)>0}\prod_{\ell=0}^{\alpha(\lambda)-1} \frac{(y_{\alpha}-\ell\hbar-c)}{(y_{\alpha}-\ell \hbar)}f(y)\eee_{-}.
$$
Now $u^{\lambda}y_{\alpha}=(y_{\alpha}+\alpha(\lambda)\hbar)u^{\lambda}$, therefore
$$
\Psi(E_{\lambda,c}[f])=\eee_{-}\prod_{\alpha(\lambda)>0}\prod_{\ell=0}^{\alpha(\lambda)-1} \frac{(y_{\alpha}-\ell\hbar-c+\alpha(\lambda)\hbar)}{(y_{\alpha}-\ell \hbar+\alpha(\lambda)\hbar)} f'(y)u^{\lambda}\eee_{-}=
$$
$$
\eee_{-}\prod_{\alpha(\lambda)>0}\prod_{\ell=1}^{\alpha(\lambda)} \frac{(y_{\alpha}+\ell\hbar-c)}{(y_{\alpha}+\ell \hbar)} f'(y)u^{\lambda}\eee_{-}.
$$
In the last step we changed index of summation from $\ell$ to  $\alpha(\lambda)-\ell$. 
On the other hand, since $\lambda$ is dominant we have $\alpha(\lambda)\ge 0$ for any positive root $\alpha$, hence 
\begin{equation}
\label{eq: Delta past u lambda}
u^{\lambda}\Delta=u^{\lambda}\prod_{\alpha\in \Phi^+}y_{\alpha}=\prod_{\alpha\in \Phi^+}(y_{\alpha}+\alpha(\lambda)\hbar)u^{\lambda}=\prod_{\langle \alpha,\lambda\rangle>0}\frac{(y_{\alpha}+\alpha(\lambda)\hbar)}{y_{\alpha}}\Delta u^{\lambda}.
\end{equation}
Now we can compute
$$
\Psi(E_{\lambda,c}[f])\Delta=\eee_{-} \prod_{\alpha(\lambda)>0}\prod_{\ell=1}^{\alpha(\lambda)} \frac{(y_{\alpha}+\ell\hbar-c)}{(y_{\alpha}+\ell \hbar)}f'(y) u^{\lambda}\Delta\eee=
$$
$$
\eee_{-} f'(y)\left(\prod_{\alpha(\lambda)>0}\prod_{\ell=1}^{\alpha(\lambda)} \frac{(y_{\alpha}+\ell\hbar-c)}{(y_{\alpha}+\ell \hbar)}\right)\prod_{\alpha(\lambda)>0} \frac{(y_{\alpha}+\alpha(\lambda)\hbar)}{y_{\alpha}}\Delta u^{\lambda}\eee=
$$
$$
\eee_{-} \Delta f'(y)\prod_{\alpha(\lambda)>0}\prod_{\ell=0}^{\alpha(\lambda)-1} \frac{(y_{\alpha}+\ell\hbar-(c-\hbar))}{(y_{\alpha}+\ell \hbar)}u^{\lambda}\eee=\Delta E_{\lambda,c-\hbar}[f'].
$$
\end{proof}

\begin{theorem}
\label{thm:shiftiso}
There is a filtered algebra isomorphism $\eee \HH_{c-\hbar} \eee\cong\eee_-\HH_{c}\eee_-$. 
\end{theorem}

\begin{proof}
 The spherical subalgebra  $\eee \HH_{c-\hbar}\eee$ is generated by the elements $E_{\lambda,c-\hbar}[f],F_{\lambda,c-\hbar}[f]$, while the antispherical subalgebra $\eee_{-} \HH_{c}\eee_{-}$ is generated by  the elements $\Psi(E_{\lambda,c-\hbar}[f]),\Psi(F_{\lambda,c-\hbar}[f])$ since $\Psi$ exchanges the spherical and antispherical subalgebras.

 Here $\lambda$ is a dominant coweight which is minimal in the Bruhat order, and $f(y)\in \C[\ft]$
  (see \cite[Proposition 6.8]{BFN}).  
By Lemma \ref{lem: Phi and shift} we get
\begin{equation}
\label{eq: Phi shift generalized}
\Psi(E_{\lambda,c}[f])=\Delta E_{\lambda,c-\hbar}[f']\Delta^{-1},\ \Psi(F_{\lambda,c}[f])\Delta=\Delta E_{\lambda,c-\hbar}[f']\Delta^{-1}
\end{equation}
where $f'(y)=u^{\lambda}f(y)u^{-\lambda}.$

Let $M$ be an operator in  $\eee \HH_{c-\hbar}\eee$, then we can consider the operator $\Delta M \Delta^{-1}$ acting on antisymmetric polynomials. By \eqref{eq: Phi shift generalized} the operator $\Delta M \Delta^{-1}$ belongs to $\eee_{-} \HH_{c}\eee_{-}$ and   the operators  $\Delta M \Delta^{-1}$ generate $\eee_{-} \HH_{c}\eee_{-}$.
\end{proof}

\begin{remark}
In \cite{Heckman, Opdam} a similar isomorphism between the spherical and antispherical subalgebras was obtained using Dunkl representation by differential-difference operators. It is natural to ask if the two isomorphisms are the same. They are not, for the isomorphism in {\em loc. cit.} is given by conjugation by the "Vandermonde in $X$", in other words $\prod_{\alpha\in\Phi^+} (1-\alpha^\vee)\in \C[T^\vee]$, which acts by identity on the operators $E_{\lambda,c}[1]$ since $X\mapsto X$ in the differential Dunkl representation. The two isomorphisms are related by the Harish-Chandra transform of \cite{CherednikBook}. This is similar to the fact that in the rational case, there are two Dunkl embeddings, to $\Diff(\fh^{\reg})\rtimes W$ and $\Diff((\fh^*)^{\reg})\rtimes W$ in which one gets similar shift isomorphisms by either conjugation by $\prod_{\alpha\in \Phi^+}y_\alpha$ or respectively by $\prod x_{\alpha^\vee}$ \cite{BEG}, and the two isomorphisms are related by Cherednik's Fourier transform.

\end{remark}

\subsection{$\Z$-algebras}
\label{sec: z algebra}

We now recall the definition of $\Z$-algebras, as explained e.g. in \cite[Section 5]{GS1}. Note that our conventions are exactly opposite to those of {\em loc. cit.} because it makes the Springer action in Section \ref{sec:bfnspringer} a bit more natural.
\begin{definition}
\label{def:Zalg}
An associative (non-unital) algebra $B=\bigoplus_{i\leq j}B_{ij}$ is a {\em $\Z$-algebra} if 
$B_{ij}B_{jk}\subseteq B_{ik}$ for all $i\leq j\leq k$, $B_{ij}B_{lk}=0$ if $j\neq l$, and each $B_{ii}$ is unital such that $1_ib_{ij}=b_{ij}=b_{ij}1_{j}$ for all $b_{ij}\in B_{ij}$.
\end{definition}

The above definition ensures that $B_{ii}$ is a unital associative algebra for all $i$, and $B_{ij}$ is a $(B_{ii},B_{jj})$-bimodule. The $\Z$-algebra multiplication factors though the convolution of bimodules:
$$
\begin{tikzcd}
B_{ij}\bigotimes_{\C} B_{jk} \arrow{rr} \arrow{dr}& & B_{ik}\\
 & B_{ij}\bigotimes_{B_{jj}} B_{jk} \arrow{ur}&  
\end{tikzcd}
$$
The simplest example of $\Z$-algebras comes from $\Z$-graded algebras.

\begin{example}
\label{ex: z algebra from graded}
Suppose that $S=\oplus_{d}S_d$ is an associative $\Z$-graded algebra with multiplication $S_{d}S_{d'}\to S_{d+d'}$.
Define $B_{ij}=S_{j-i}$ for all $i$ and $j$, then $B(S)=\bigoplus_{i\leq j}B_{ij}$ is a {\em $\Z$-algebra}.
Note that in this example the algebras $B_{ii}$ are all isomorphic to $S_0$.
\end{example}

Our main source of $\Z$-algebras will be a filtered deformation of Example \ref{ex: z algebra from graded}. We say that a $\Z$-algebra $B$ is {\bf of graded type} if it has an algebra filtration (which we omit from the notations) such that $\gr B=B(S)$ for some commutative graded algebra $S$. Unpacking this definition, we get the following properties of $\gr B$:
\begin{itemize}
\item $S_0:=\gr B_{ii}$ is a commutative algebra which does not depend on $i$ up to isomorphism
\item $S_{j-i}:=\gr B_{ij}$ depends only on the difference $j-i$ up to isomorphism
\item For all $i,j,k$ we have a commutative square
$$
\begin{tikzcd}
\gr B_{ij}\bigotimes \gr B_{jk} \arrow{d} \arrow{r}& \gr B_{ik} \arrow{d}\\
S_{j-i}\bigotimes S_{k-j} \arrow{r} & S_{k-i}  
\end{tikzcd}
$$
\item The left and right actions of $S_0\simeq\gr B_{ii}\simeq \gr B_{jj}$ on the bimodule $\gr B_{ij}$ agree.
\end{itemize}

The last bullet point is related to the Harish-Chandra property for the bimodules $B_{ij}$, see \cite{Simental,LosevCompl}. Namely, {\em a Harish-Chandra bimodule} for a filtered algebra $A$ is a bimodule $B$ with an exhaustive filtration s.t. $[A_{\leq i},B_{\leq j}]\subseteq B_{i+j-d}$ s.t. $\gr B$ is finitely generated. The commutator condition implies that the left and right actions of $\gr A$ on $\gr B$ agree.

Also note that we can associate a pair of schemes to a $\Z$-algebra of graded type: the affine scheme $\Spec S_0$ and the scheme $\Proj S$. We have a natural morphism $\Proj S\to \Spec S_0$.

Next, we define modules over a $\Z$-algebra $B$. A graded vector space $M=\oplus M_i$ is a $B$-module if for all $i$ and $j$ we have multiplication maps $B_{ij}\otimes M_j\to M_i$ such that we have a commutative diagram
$$
\begin{tikzcd}
B_{ij}\otimes B_{jk}\otimes M_k \arrow{d}\arrow{r}& B_{ij}\otimes M_j \arrow{d}\\
B_{ik}\otimes M_k \arrow{r} & M_i.
\end{tikzcd}
$$
In particular, $M_i$ is a module over the algebra $B_{ii}$ for all $i$. If $B$ is of graded type and $M$ admits a filtration compatible with a filtration on $B$ then $\gr M$ is graded $S$-module for the graded algebra $S$. In particular, $\gr M$ defines a quasicoherent sheaf on $\Proj S$. 

\subsection{$\Z$-algebras from Cherednik algebras}

We now turn to defining a $\Z$-algebra $\cB={}_\bullet \cB_{\bullet}^{\hbar}$ as follows. The component ${}_i \cB_i^{\hbar}$ is the spherical Cherednik algebra 
$\eee \HH_{c+i\hbar}\eee$ with parameter $c+i\hbar$. The component  ${}_i \cB_{i+1}^{\hbar}$ is the {\em shift bimodule}
$${}_i \cB_{i+1}^{\hbar}=\eee \HH_{c+(i+1)\hbar,\hbar}\eee_-$$
over the algebras ${}_{i+1} \cB_{i+1}^{\hbar}=\eee \HH_{c+(i+1)\hbar,\hbar}\eee$ and 
$$
{}_{i} \cB_{i}^{\hbar}=\eee \HH_{c+i\hbar,\hbar}\eee\simeq \eee_- \HH_{c+(i+1)\hbar,\hbar}\eee_-.
$$
The last isomorphism is given by Theorem \ref{thm:shiftiso}.
Finally, for more general $i<j$ we define the shift bimodules
$${}_i \cB_{j}^{\hbar}={}_i\cB_{i+1}^{\hbar}\cdots {}_{j-1}\cB_{j}^{\hbar}$$ where $\cdot$ denotes the appropriate tensor product.

\begin{lemma}
At $\hbar=c=0$ one has ${}_{i}\cB_{j}^{\hbar=0}=A^{j-i}$, where $A$ is the subspace of diagonally antisymmetric polynomials in $\C[T^*T^\vee]$, and this is compatible with the multiplication. When $i=j$ this is the subspace of diagonally symmetric polynomials. 
\end{lemma}
\begin{proof}
Let us prove that ${}_{i}\cB_{i+1}^{\hbar=0}=A$. Indeed, ${}_{i}\cB_{i+1}^{\hbar=0} =\eee_{-} \HH \eee\cong \eee_{-} \C[T^*T^\vee]$  is the space of antisymmetric polynomials in $\C[T^*T^\vee]$. Similarly, ${}_{i}\cB_{i}^{\hbar=0}\cong \eee \C[T^*T^\vee]=\C[T^*T^\vee]^W$. Now 
$$
{}_{i}\cB_{j}^{\hbar=0}=A\otimes_{\C[T^*T^\vee]^{W}}   \otimes\cdots  \otimes_{\C[T^*T^\vee]^{W}}  A=A^{j-i}.
$$
\end{proof}

\begin{example}
\label{ex: module sl2}
Consider the trigonometric Cherednik algebra for $G=\GL_2$. 
For $\nu=1/2$ it has a $1$-dimensional representation $L_{1/2}(\triv)$ with invariant part $\eee L_{1/2}(\triv)\cong \eee_- L_{3/2}(\triv)$. Using this isomorphism, the bimodule 
$\eee \HH_{3/2} \eee_-$ sends $\eee \HH_{3/2} \eee_-\otimes_{\eee_- \HH_{3/2} \eee_-} \eee_- L_{3/2}(\triv)\cong \eee L_{3/2}(\triv)$.
More generally, 
$$\eee \HH_{(2k+1)/2} \eee_-\otimes_{\eee_-\HH_{(2k+1)/2}\eee_-} \eee_- L_{(2k+1)/2}(\triv)\cong \eee L_{(2k+1)/2}(\triv)$$ and the direct sum 
$$\bigoplus_{k\geq 0} \eee L_{(2k+1)/2}(\triv)$$ is a module for the $\Z$-algebra $\cB$.

\end{example}

\subsection{$\Z$-algebra for $\GL_n$}
Consider now the $\Z$-algebra as introduced above for $G=\GL_n$. We have 
\begin{theorem}
\label{th: Gordon Stafford}
For all $i\le j$ the $\C[\hbar]\otimes \C[y_1,\ldots,y_n]^{S_n}$-module ${}_{i}\cB_{j}^{\hbar}$ is free.
\end{theorem}
\begin{proof}
The idea of the proof is to replace the trigonometric Cherednik algebra
\(\mathbb{H}_{c,\hbar}=H^{\trig}_{c,\hbar}\) with the rational Cherednik algebra \(H^{\rat}_{\hbar}\) \cite{EG}. The  algebra
\(H^{\rat}_{\hbar}\) is the quotient of \(\mathbb{C}[X_1,\dots,X_n]\otimes\mathbb{C}[z_1,\dots,z_n]\rtimes S_n\) modulo the
relations:
\[[X_i,z_i]=\hbar+\sum_{j\ne i}\sigma_{ij}\ i=1,\dots,n,\]
\[[X_i,z_j]=-\sigma_{ij}, i\ne j,\]
where \(\sigma_{ij}\in S_n\) is the transposition (the generators $z_i$ are usually called $y_i$ in the rational Cherednik algebra literature).

For the rational Cherednik algebra the corresponding $\Z$--algebra \({}_\bullet\mathcal{B}_{\bullet}^{\rat,\hbar=\hbar_0}\) was constructed by Gordon and Stafford \cite{GS1,GS2} who defined a filtration on ${}_{i}\cB_{j}^{\rat,\hbar=\hbar_0}$ for any specialization of $\hbar$ and proved 
that $\gr {}_{i}\cB_{j}^{\rat, \hbar=\hbar_0}\simeq A^{j-i}_{\rat}$ using Haiman's results. Note that this was achieved without relying on Haiman's results in \cite{GGS}. This implies that ${}_{i}\cB_{j}^{\rat,\hbar}$ is free over $\hbar$. The freeness over $\C[y_1,\ldots, y_n]^{S_n}$ is e.g. \cite[Lemma 6.11(2)]{GS1}. We remark that this freeness uses results of \cite{BEG} about Morita equivalence of Cherednik algebras.

Now the trigonometric case is obtained by Ore localization in the central element $X_1\cdots X_n$, which commutes with the action of $\hbar$ and $y_i$. This follows from 
\begin{lemma}
There is a natural map 
$${}_{i}\cB_{j}^{\rat,\hbar=\hbar_0}\to {}_i \cB_j$$ which becomes an isomorphism upon localization in $\prod X_i$:
$({}_{i}\cB_{j}^{\rat,\hbar=\hbar_0})_{\prod X_i}\cong {}_i \cB_j$.
\end{lemma}
\begin{proof}
In \cite{Suz} it is shown that
\[\imath(w)=w,\quad \imath(X_i)=X_i,\quad \imath(z_i)=X_i^{-1}(y_i-\sum_{1\le j\le i}\sigma_{ji}),\]
extends to the algebra homomorphism \(\imath: \HH^{\rat}_{\hbar}\to \HH^{\trig}_{1,\hbar}\) that becomes an isomorphism
after localization by \(X_1\cdots X_n\). By hitting with $\eee$ on both sides, this implies the statement on the level of the spherical subalgebras. (To match parameters, we observe that \(H^{\trig}_{c,\hbar}\simeq H^{\trig}_{\lambda c,\lambda \hbar}\) for
any \(\lambda\in \mathbb{C}^*\).)

For the one-step bimodules, ${}_i\cB_{i+1}^{\rat}=\eee \HH^{\rat}_{c+i\hbar,\hbar}\eee_-$ by definition, so the result is true for $j=i+1$ as well. Finally, $${}_i\cB_j^{\rat}={}_i\cB^{\rat}_{i+1}\cdots {}_{j-1}\cB_j^{\rat}$$ and by standard properties of localization and tensor product we get the result.

\end{proof}
Now since tensoring with $\C[X_1,\ldots,X_n]_{\prod X_i}$ is faithfully flat, we deduce that since ${}_i\cB_j^{\rat}$ is free over $$\C[y_1,\ldots,y_n]^{S_n},$$ so is ${}_i\cB_j$.
This finishes the proof of Theorem \ref{th: Gordon Stafford}.
\end{proof}

More geometrically, the  bimodule ${}_{i}\cB_{j}^{\rat,\hbar}$ quantizes the line bundle $\calO(j-i)$ on the Hilbert scheme of points on $\C^2$ while ${}_{i}\cB_{j}^{\hbar}$ quantizes its restriction to the Hilbert scheme of $\C^{\times}\times \C$.

\begin{corollary}
The $\Z$-algebra $\cB$ is of graded type and for $G=\GL_n$, $\gr \cB$ corresponds to the graded algebra $S=\oplus_{d=0}^{\infty} A^{d}$. The corresponding algebraic varieties are $\Proj S=\Hilb^n(\C^{\times}\times \C)$ and $\Spec S_0=(\C^{\times}\times \C)^n/S_n$.
\end{corollary}

\section{Coulomb branches and  $\Z$-algebras}
\label{sec:coulomb}

In this section, we explain half of the main construction of the paper, namely the construction of a $\Z$-algebra associated to the Coulomb branch of the $3d$ $\cN=4$ theory with adjoint matter, or in other words the spherical trigonometric DAHA. Most of the results work in greater generality, and are stated as such wherever possible. In Section \ref{sec: adjoint coulomb} we specialize these general constructions to the case of adjoint representation. 

The other half of the main construction, consisting of a generalized affine Springer theory for this $\Z$-algebra, is treated in Section \ref{sec:bfnspringer}. 

\subsection{Coulomb branches}

Let $1\to G\to \tG\to G_F\to 1$ be an extension of algebraic groups, where $G$ is reductive and $G_F$ is 
diagonalizable. Let $N$ be an algebraic representation of $\tG$, $\bP\subset G(\cO)\subset G(\cK)$ be a 
standard parahoric subgroup and $N_\bP$ a lattice in $N(\cK)$ stable under $\bP$. We will only be interested 
in the case where $N=\Ad$, $N_\bP=\Lie(\bP)$.

Let $\cR_\bP:=\cR_{G,N,\bP,N_\bP}$ be the {\em parahoric BFN space of triples} as in \cite{GK}. More 
precisely, we have
\begin{definition}
$\cR_{G,N,\bP,N_\bP}$ is the fpqc sheaf on $\Sch/\kk$ associating to $S$ the groupoid of tuples $(\cP,\varphi,s,
\cP_\bP)$ where $\cP$ is a $G$-bundle on $S\times \Spec \cK$, $\cP_\bP$ is a $\bP$-reduction of $\cP$ over $S
\times \Spec \cO$, and $\varphi$ is a trivialization over $S\times \Spec \cO$ compatible with the $\bP$-
structure. Moreover, $s$ is a section of the associated $N$-bundle of $\cP_\bP$ such that $\varphi \circ s 
(t) \in N_\bP$.
\end{definition}
\begin{remark}
Dropping the condition that $\varphi \circ s (t) \in N_\bP$, we get the space $\cT_{G,N,\bP,N_\bP}=G(\cK)
\times_\bP N_{\bP}$, which is an (ind-)vector bundle over the partial affine flag variety $\Fl_\bP$. In 
particular, if $N=\Ad,$ this can be thought of as the cotangent bundle of $\Fl_\bP$.
\end{remark}
We recall the following definitions and theorems as motivation for the following sections.
 We define the group $$\tbP:=\ev_0^{-1}(G_F \ev_0(\bP))$$ where $\ev_0: \tG(\cO)\to \tG$ is the map sending $t\mapsto 0$. 
 In general, we refer by underline to flavor-deformed objects.
From \cite{BFN, GK, HKW}, we have
\begin{theorem}
\label{thm:coulombisalgebra}
$H_*^{\tbP\rtimes \C^\times}(\cR_\bP)=:\cA_\bP^{\hbar}$ is an associative algebra with unit. It is a flat deformation of $\cA_\bP:=H_*^{\tbP}(\cR_{\bP})$. When $\bP=G(\cO)$, the algebra $\cA^{\hbar}_{G(\cO)}$ is a filtered quantization of $\cA_{G(\cO)}$, which is commutative. The spectrum of $\cA_{G(\cO)}$ is called the {\em Coulomb branch} of $(G,N)$.
\end{theorem}

The papers \cite{HKW,GK} developed a generalization of affine Springer theory for Coulomb branches which is summarized in the next result below. We discuss it in more detail in Section \ref{sec:bfnspringer}.

\begin{theorem}[\cite{HKW,GK}]
Let $v\in N(\cK)$ and
let $L_v\subset \tbP\rtimes \C^\times$ be the stabilizer of $v$ in $\tbP \rtimes \C^\times_{\rot}.$ 
The algebra $\cA_\bP^{\hbar}$ acts on
$H^{L_v}_*(M_v^\bP)$ via natural cohomological correspondences, provided the group $L_v$ is compact in the $t$-adic topology.
\end{theorem}

\subsection{A category of line defects}
\label{sec:linedefects}
Heuristically, the equivariant BM homologies of the spaces of triples above are endomorphisms of objects in a "category of line operators" \cite{DGGH, Weekes, Webster} which is something like $G(\cK)$-equivariant $D$-modules on $N_\cK$.

We won't stipulate on the definition of the actual category (see however \cite{BKV} in the adjoint case), but this category should contain objects coming from $\eta=(U,\bP)$, where $U\subset N(\cK)$ is a $\bP$-stable lattice and $\bP$ is a parahoric subgroup of $G(\cK)$. We will simply {\em define} $\Hom(\eta,\eta')=H^{\bP'\rtimes \C^\times}_*({}_\eta \cR_{\eta'})$, where $${}_\eta \cR_{\eta'}=\left\{[g,s]\in G(\cK)\times^{\bP'}U'|gs\in U\right\}.$$ We will use the notation $U=N_\bP$ to emphasize $N_\bP$ is a $\bP$-stable lattice. By abuse of notation, we will also write $\Hom(\eta,\eta')$ for the flavor- or loop-rotation deformed versions of these spaces.
It is clear that when $\eta=\eta'=N_\bP$, we have 
$$\End(\eta)=\Hom(\eta,\eta)=\cA_\bP^{\hbar}$$ from Theorem \ref{thm:coulombisalgebra}.

\begin{theorem}
\label{thm:convolution}
There is an associative multiplication $\Hom(\eta,\eta')\otimes_\C\Hom(\eta',\eta'')\to \Hom(\eta,\eta'')$ via the following modification of the BFN convolution product.
   \begin{center}
   \begin{tikzcd}
{}_\eta\cR_{\eta'} \times {}_{\eta'}\cR_{\eta''} \arrow[d,"i",hook] & p^{-1}\left({}_\eta\cR_{\eta'}\times {}_{\eta'}\cR_{\eta''}\right) \arrow[l,"p"'] \arrow[r,"q"]\arrow[hook, d,"j"] & q\left(p^{-1}({}_\eta\cR_{\eta'}\times {}_{\eta'}\cR_{\eta''})\right) \arrow[d,"m"] \\
\cT_{\eta'} \times {}_{\eta'}\cR_{\eta''} & G(\cK) \times {}_{\eta'}\cR_{\eta''} \arrow[l,"p"] & {}_{\eta}\cR_{\eta''}
\end{tikzcd}
   \end{center}
Here the maps $p,q,m$ send $$p:(g_1,[g_2,s])\mapsto ([g_1,g_2s],[g_2,s]), \;\; q:(g_1,[g_2,s])\mapsto [g_1,[g_2,s]],$$ $$m:[g_1,[g_2,s]]\mapsto[g_1g_2,s]$$ and $i,j$ are inclusions of closed subvarieties.
\end{theorem}
\begin{proof}
This can be proved using a straightforward modification of the proof of associativity in \cite[Section 3]{BFN}. Similar results for $\eta=(N_\bP,\bP)$ where $\bP$ is an Iwahori subgroup are mentioned in \cite{Webster}.
\end{proof}

\begin{corollary}
\label{cor: bimodules}
For any $\eta$ the space $\Hom(\eta,\eta)$ is an associative algebra, and $\Hom(\eta,\eta')$ is a bimodule over $\Hom(\eta,\eta)$ and 
$\Hom(\eta',\eta')$. Given $\eta,\eta'$ and $\eta''$ we have a natural morphism of bimodules over  $\Hom(\eta,\eta)$ and 
$\Hom(\eta'',\eta'')$:
$$
\Hom(\eta,\eta')\bigotimes_{\Hom(\eta',\eta')}\Hom(\eta',\eta'')\to \Hom(\eta,\eta'').
$$
\end{corollary}
\begin{proof}
We need to prove the morphism from Theorem \ref{thm:convolution} is bilinear over $\Hom(\eta',\eta')$.
The only axiom of a tensor product we need to show is 
$m\cdot r\otimes n=m\otimes r\cdot n$ for $m\in \Hom(\eta,\eta'), n \in \Hom(\eta',\eta'')$ and $r\in \Hom(\eta',\eta')$, which is clear from the associativity of the construction.
\end{proof}

We will also use the notation ${}_\eta \cA_{\eta'}:=\Hom(\eta,\eta')$.
The following is a generalization of \cite[Lemma 5.3]{BFN}.
\begin{theorem}
The bimodule
${}_\eta \cA_{\eta'}=\Hom(\eta,\eta')$ is flat as a left $\C[\ft^*][\hbar,c]=H_*^{\tilde{T}}(pt)$-module.\end{theorem}
\begin{proof}
The proof is similar to  \cite[Lemma 5.3]{BFN}. The space ${}_\eta \cR_{\eta'}/\bP$ has a natural projection to the partial affine flag variety $G({\cK})/\bP$, and can be decomposed into (infinite-dimensional) affine fibrations over the affine Schubert cells in the latter. By definition, the equivariant homology of  ${}_\eta \cR_{\eta'}/\bP$ is a filtered colimit of the equivariant homologies of certain unions of these strata according to the Bruhat filtration. Any union of such even-dimensional cells is equivariantly formal, and its homology is flat over $H_*^{\tilde{T}}(pt)$.
%The associated graded for the Bruhat filtration is free by equivariant formality. 
On the other hand, filtered colimits of free modules are flat.
\end{proof}

\subsection{$\Z$-algebras and the flavor deformation}
Taking a sequence $\eta_0,\eta_{-1},\ldots$ it is clear from Theorem \ref{thm:convolution} that we get a $\Z$-algebra by taking
$$\cA=\bigoplus_{i\leq j} {}_i \cA_j$$ where ${}_i\cA_j$ denotes ${}_{\eta_i}\cA_{\eta_j}$.  That is,
\begin{theorem}
The algebra $\bigoplus_{i\leq j} {}_{i}\cA^{\hbar}_j$ is a $\Z$-algebra.
\end{theorem}

In the setup of BFN Coulomb branches, nontrivial $\Z$-algebras are most easily obtained via a {\em flavor deformation} of $G$, i.e. by letting $G_F$ be nontrivial. We now explain this procedure for $G_F=\G_m$ and associate to $(\tG,N)$ a $\Z$-algebra. 

Specifically, we can consider a sequence $\eta_i=(t^{-i}U,G(\cO))$ for some fixed lattice $U$ (for example, $U=N(\cO)$). If we set $\hbar=0$, it is easy to see that  $S_{j-i}:={}_{i}\cA^{\hbar=0}_j$ depends only on $j-i$ and the algebra $\cA^{\hbar=0}$ is of graded type as in Section \ref{sec: z algebra}. In particular, at $\hbar=0$ all commutative algebras ${}_{i}\cA^{\hbar=0}_i$ are isomorphic to the commutative Coulomb branch algebra $S_0=\cA_{G(\cO)}$. In particular, this construction yields a map $\Proj S\to \Spec S_0$, which is a variant of the construction of a partial resolution of the Coulomb branch in \cite{BFN3}.

\subsection{$\Z$-algebras in the abelian case}
Since it might be of independent interest and is used for computations  below, we now work out the $\Z$-algebras for the cases when $G=T$ is a diagonalizable algebraic group and $\{\eta_i\}_{i=0}^\infty$ is given by $\eta_i=(T(\cO), 
t^{i\phi} N(\cO))$ for some (flavor) cocharacter $\phi: \G_m\to T$. 
Note that when $N=0$, the $\Z$-algebra collapses to $\cA_{T,0}^{\hbar}[c]$ where $c$ is the flavor parameter (the generalization to more flavors is straightforward). 

Under $\cT_{j}\hookrightarrow \Gr_T\times N(\cK)$ the image is naturally identified with $$\bigsqcup_{\lambda \in \Gr_T} \{t^\lambda\}\times t^jt^{\lambda}N(\cO)$$
and similarly $${}_i\cR_j\cong \bigsqcup_{\lambda\in \Gr_T} \{t^\lambda\} \times (t^jt^{\lambda} N(\cO) \cap t^i N(\cO)).$$

Now let ${}_i r_j^\lambda$ be the preimage of $\lambda \in \Gr_G$ under the projection ${}_i \cR_j\to \Gr_T$. Suppose also $N$ is the direct sum of the characters $\xi_1,\ldots, \xi_n$ as a $T$-representation.

\begin{theorem}
\label{thm:abelianconvolution}
Under the convolution product in Theorem \ref{thm:convolution}, we have for all $i,j,k\in \Z$ that
$${}_i r_j^\lambda {}_j r_k^\mu= \prod_{\ell=1}^n A_\ell (i,j,k,\lambda,\mu ) {}_i r_k^{\lambda+\mu}$$
where 
$$A_{\ell}(i,j,k,\lambda,\mu)=\prod_{a=\max(\lambda+i,k-\mu)+1}^{\max(\lambda+i,k-\mu,j)}(\xi_\ell+c+(a+\xi_\ell(\lambda))\hbar)\prod_{b=\min(\lambda+i,k-\mu,j)+1}^{\min(\lambda+i,k-\mu)}(\xi_\ell+c+(b+\xi_\ell(\lambda))\hbar)$$
\end{theorem}
\begin{proof}
We may restrict to the case where the rank of $T$ is $1$. In this case, the computation is essentially \cite[Proposition 3.10]{Webster}, generalizing \cite[Theorem 4.1]{BFN}. In the notation of {\em loc. cit.} we have
$${}_i r_j^\lambda=y_\lambda r(\lambda+i,j),\; {}_j r_k^\mu=y_\mu r(\mu+j,k)$$
so we get 
\begin{align*}
{}_i r_j^\lambda{}_jr_k^\mu&=y_\lambda r(\lambda+i,j)r(j,k-\mu)y_\mu\\
 & =y_\lambda e\left(\frac{t^{\lambda+i}N(\cO)\cap t^{k-\mu}N(\cO)}{t^{\lambda+i}N(\cO)\cap t^{k-\mu} N(\cO)\cap t^jN(\cO)}\right)e\left(\frac{t^{\lambda+i}N(\cO)+t^{k-\mu}N(\cO)+t^jN(\cO)}{t^{\lambda+i}N(\cO)+t^{k-\mu} N(\cO)}\right)y_\mu
\end{align*}
And we compute the Euler classes 
$$e\left(\frac{t^{\lambda+i}N(\cO)\cap t^{k-\mu}N(\cO)}{t^{\lambda+i}N(\cO)\cap t^{k-\mu} N(\cO)\cap t^jN(\cO)}\right)=\prod_{a=\max(\lambda+i,k-\mu)+1}^{\max(\lambda+i,k-\mu,j)}(\xi_\ell+c+a\hbar)$$
$$e\left(\frac{t^{\lambda+i}N(\cO)+t^{k-\mu}N(\cO)+t^jN(\cO)}{t^{\lambda+i}N(\cO)+t^{k-\mu} N(\cO)}\right)=\prod_{b=\min(\lambda+i,k-\mu,j)+1}^{\min(\lambda+i,k-\mu)}(\xi_\ell+c+b\hbar)
$$
From the relation $y_\lambda\chi=(\chi+\hbar \langle \chi,\lambda\rangle)y_\lambda$ for $\chi \in \ft^*$ we get that 
$${}_i r_j^\lambda{}_jr_k^\mu=A_\ell(i,j,k,\lambda,\mu) {}_i r_k^{\lambda+\mu}$$

\end{proof}
\begin{remark}
When $\hbar=c=0$, the above becomes $${}_i r_j^\lambda {}_j r_k^\mu=\prod_{\ell=1}^n\frac{\xi_\ell^{\max(\lambda+i,k-\mu,j)}}{\xi_\ell^{\max(\lambda+i,k-\mu)}}\cdot \frac{\xi_\ell^{\min(\lambda+i,k-\mu)}}{\xi_\ell^{\min(\lambda+i,k-\mu,j)}}{}_i r_k^{\lambda+\mu}$$

\end{remark}

\begin{lemma}
All algebras in question are naturally graded with
$$
\deg \xi_{\ell}=2,\ \deg({}_i r_j^\lambda)=|\lambda+i-j|.
$$
\end{lemma}

\begin{proof}
Observe that 
\begin{equation}
\label{eq: abs}
|a-b|+|b-c|-|a-c|=2\left(\max(a,b,c)-\max(a,c)+\min(a,c)-\min(a,b,c)\right).
\end{equation}
Indeed, both sides  of \eqref{eq: abs} are symmetric in $a$ and $c$ and vanish if $b$ is between $a$ and $c$. If $b<a<c$ then we get
$$
(a-b)+(c-b)-(c-a)=2(a-b)=2(c-c+a-b),
$$
while for $a<c<b$ we get 
$$
(b-a)+(b-c)-(c-a)=2(b-c)=2(b-c+a-a).
$$
By substituting $a=\lambda+i,b=j,c=k-\mu$ we can verify that the defining equations are homogeneous since
$$
\deg({}_i r_j^\lambda)=|\lambda+i-j|,\ \deg({}_j r_k^\mu)=|k-\mu-j|=|j+\mu-k|,\ \deg({}_i r_k^{\lambda+\mu})=
|(\lambda+i)-(k-\mu)|=|\lambda+\mu+i-k|.
$$
\end{proof}

As in \cite[Appendix]{BFN2} and \cite{BFN}, the inclusion $\Gr_T\hookrightarrow \cT_j$ as a subbundle gives rise to an injective map $z^*$ in equivariant Borel-Moore homology:
$${}_iz_j^*: {}_i \cA^{\hbar}_j \hookrightarrow H^{\underline{T}(\cO)\rtimes \G_m}_*(\Gr_T).$$
If $u_\lambda\in H^{\underline{T}(\cO)\rtimes \G_m}_*(\Gr_T)$ is the class of the cocharacter $\lambda$, or more algebraically the $\hbar$-difference operator on $\ft$ acting on $f\in k[\ft]$ by $f(x)\mapsto f(x+\hbar \lambda)$, we have
\begin{theorem}
\label{thm:abelianembedding}
Under ${}_iz_j^*$, we have $${}_iz_j^*({}_ir_j^\lambda)=
e(t^\lambda t^i N(\cO) / t^\lambda t^i N(\cO)\cap t^j N(\cO)) u_\lambda$$
\end{theorem}
Note that the injectivity of ${}_iz_j^*$ is clear from the above.

\section{Coulomb branches and $\Z$-algebras in the adjoint case}
\label{sec: adjoint coulomb}

\subsection{From Coulomb branch to Cherednik algebra}

We now discuss the adjoint case. For arbitrary $G$ and $N=\Ad$, the construction of \cite{BFN} yields a noncommutative resolution of $T^*T^\vee/W$ in the sense of \cite{VdB}.
Instead of the spherical case, we focus on the Iwahori case as well as the resulting $\Z$-algebras.

First of all, we claim Theorem \ref{thm:convolution} for $\eta=(\bI, \Lie(\bI))$ gives a realization of the trigonometric DAHA (as conjectured in many places, including \cite{BFN}) and that the resulting action on the affine Springer fibers coincides with Yun's action. The goal of this section is to prove these claims, and to show that for $\eta=(G(\cO), \Lie(G(\cO)))$ we similarly get the spherical trigonometric DAHA, as expected in \cite{BFN, BFM} and other places.

We now state the main theorem of this section.
\begin{theorem}
\label{thm:iwahoricoulombisdaha}
The Iwahori-Coulomb branch algebra $$\tcA_{G,\bI}^{\hbar}=H^{\tbI\rtimes \C^\times}_*(\cR_{N_\bI,\bI})$$ is naturally isomorphic to $\HH_{c,\hbar}$.
\end{theorem}
\begin{proof}

We denote by $\cR_{N_\bI,\bI}^{\leq w}$ the preimage of the Schubert variety $\Fl_G^{\leq w}$ under the projection
$\cR_{N_\bI,\bI}\to \Fl_G$. The algebra $\tcA_{G,\bI}^{\hbar}$ is generated over $H^*_{\tbI\rtimes \C^\times}(\mathrm{pt})$ by the non-canonically defined "fundamental classes" $[\cR_{N_\bI,\bI}^{\leq w}]$ for $w\in \tW$, which however have a unique leading term in the Bruhat order. This follows exactly in the same way as in the spherical case of \cite[Section 6(i)]{BFN}.

Note that when $w=s_i$ is a simple reflection, the restriction of the projection $\cR_{N_\bI,\bI}^{\leq w}\to \Fl^{\leq w}\cong \P^1$ outside the origin is a (profinite-dimensional) vector bundle and its closure in $\cR$ is still a vector bundle over $\P^1$. In this case we can and will take 
$[\cR_{N_\bI,\bI}^{\leq s_i}]$ to be the corresponding fundamental class. In the type A setting, this was noted for example in \cite[Section 4.2.]{BEF}.
We will prove that $\tcA_{G,\bI}^{\hbar}$ has a faithful representation in $\hbar$-difference operators on $\ft$ which satisfies the relations of the dDAHA.

Consider the restriction of the projection $\cR_{N_\bI,\bI}\to \Fl$ to the fixed points $\Fl^T\cong \tW$. Denote the 
pullback by $\cR_{N_\bI,\bI,T}$. This gives rise to a morphism $\iota_*: H^{\tbI\rtimes \C^\times}_*(\cR_{N_\bI,\bI,T})\to \tcA_{G,\bI}^{\hbar}$. Just as noted in \cite[Section 4.1.]{BEF} for the $G=GL_n$-case, the proof of 
\cite[Lemma 5.11]{BFN} goes through word for word for these ind-varieties and we have an algebra embedding 
$$\mathbf{z}^*: H^{\tbI\rtimes \C^\times}_*(\cR_{N_\bI,\bI,T})\hookrightarrow H^{\tbI\rtimes \C^\times}_*(\Fl^T)$$ Upon localization in the generalized roots, we obtain an injection $z^*(\iota_*)^{-1}:\tcA^{\hbar}_{G,\bI}\hookrightarrow H^{\tbI\rtimes \C^\times}_*(\Fl^T)\cong \Diff(\fh^{reg})\rtimes \C[W]$, similar to the spherical case handled in \cite[Section 5(v)]{BFN}. It is clear that under the algebra filtration coming from the usual Bruhat order on $\Fl^T$, this injection respects the filtrations.

In particular, similar to \cite[Proposition 6.2]{BFN} we get
$$[\cR_{N_\bI,\bI}^{\leq w}][\cR_{N_\bI,\bI}^{\leq w'}]=[\cR_{N_\bI,\bI}^{\leq ww'}]+\text{ lower order terms in Bruhat order}$$

Note that each $w$ has a reduced expression, say $w=s_{i_1}\cdots s_{i_j}$. In particular,
$$[\cR_{N_\bI,\bI}^{\leq s_{i_1}}]\cdots[\cR_{N_\bI,\bI}^{\leq s_{i_j}}]=[\cR_{N_\bI,\bI}^{\leq s_{i_1}\cdots s_{i_j}}]+ \text{ lower order terms}.$$
This implies that the classes $[\cR_{N_\bI,\bI}^{\leq s_i}]$ above the one-dimensional $\bI$-orbits  in $\Fl_G$ generate $\tcA_{G,\bI}^{\hbar}$ together with the equivariant parameters. So it remains to check these satisfy the right relations.

Via a similar localization computation as in the following section,
 $[\cR_{N_\bI,\bI}^{\leq s_i}]$ 
is of the form $a+bs_i$, where $1, s_i$ are the corresponding fixed-point classes and act as the identity and the simple reflection on $\C[\hat{\fh}]$, respectively. The coefficient of $s_i$ is the Euler class of $\cT_{N_\bI,\bI}^{\leq s_i}/\cR_{N_\bI,\bI}^{\leq s_i}\to \P^1$ divided by the tangent weight, more precisely $b=\frac{\alpha_i+c}{\alpha_i}$. Similarly, we have $a=\frac{-\alpha_i+c}{-\alpha_i}$, giving that $[\cR^{\leq s_i}]$ acts via 
$$[\cR^{\leq s_i}] f=
(1+s_i)f+c\frac{(s_i-1)f}{\alpha_i}$$

As the relations between the $[\cR^{\leq s_i}]$ are exactly those from Definition \ref{def:dDAHA}, this defines a homomorphism $\HH_G\to \tcA_{G,\bI}^{\hbar}$ sending $1+\sigma_i\mapsto  [\cR^{\leq s_i}]$, $c\mapsto c, \hbar\mapsto \hbar$ and $\C[\ft]\ni f\mapsto f\in \C[\ft]\cong H^*_T(pt)$. The relations follow from those in the usual polynomial representation of the degenerate DAHA in the $\hbar$-difference representation \eqref{eq:Dunklembedding}. This shows $\tcA_{G,\bI}^{\hbar}$ is a quotient of $\HH_G$; faithfulness follows from \cite[Proposition 1.5.6.]{CherednikBook}.

\end{proof}
\begin{remark}
In $K$-theory, a similar proposition is proven for the full DAHA with some specific rational parameter values in \cite[Section 2.5]{VV2}.
\end{remark}
\begin{corollary}
\label{cor: sperical coulomb is daha}
The Coulomb branch algebra $\cA^{\hbar}_{G,G(\cO)}$ of Braverman-Finkelberg-Nakajima is isomorphic to the spherical subalgebra of $\HH_{c,\hbar}$.
\end{corollary}
\begin{proof}
Let $\eee=|W|^{-1}\sum_{w\in W} w\in \HH_{c,\hbar}\cong \cA^{\hbar}_{G,\bI}$. 
Recall that $N=\Ad$ in this section. Then, in the notation of Theorem \ref{thm:iwahoricoulombisdaha},
we have the left $\underline{G}(\cO)\rtimes \C^\times $-equivariant projection $\cR_{N_\bI,\bI}\to \cR_{N_\cO,G(\cO)}$ obtained as the combination of the projection $\Fl\to \Gr$ and the inclusion $N_\bI\hookrightarrow N_\cO$. It fits into a cartesian square 
\begin{equation}
\label{eq:SpringerCoulombDiagram}
\begin{tikzcd}
\cR_{N_\bI,\bI} \arrow[r,"\varphi'"]\arrow[d, "\pi"'] & \left[ \tilde{\fg}/G \right]=[\fb/B] \arrow[d, "\pi'"]\\
\cR_{N_\cO,G(\cO)} \arrow[r, "\varphi"]&  \left[ \fg/G \right]
\end{tikzcd}
\end{equation}
in particular, taking $\underline{\bI}\rtimes \C^\times$-equivariant homology, by classical Springer theory we have that 
$$H^{\underline{\bI}\rtimes \C^\times}_*(\cR_{N_\cO, G(\cO)})\cong \cA^{\hbar}_{G,\bI}\eee$$ To get $\underline{G}(\cO)\rtimes \C^\times$-invariants, using Atiyah-Bott we compute $$H^{\underline{G}(\cO)\rtimes \C^\times}_*(\cR_{N_\cO, G(\cO)})\cong \eee \cA^{\hbar}_{G,\bI}\eee$$

\end{proof}
\begin{remark}
This proves the speculation in \cite[Remark 6.20]{BFN}. For $G=\GL_n$ this was proved by Kodera and Nakajima in \cite{KN}.
\end{remark}

Finally, we give a geometric realization of the ''shift" bimodules of the trigonometric Cherednik algebra using line operators.

Let  $\eta=(t^i \Lie(\bI),\bI)$ and $\eta'=(t^i\Lie(\bI\cap t G(\cO)),\bI)$, where $\Lie(\bI\cap t G(\cO))$ is the pronilpotent radical of $\Lie(\bI)$. In the notations of  section \ref{sec:linedefects} denote 
$${}_{i+1}\tcR_i:={}_\eta \cR_{\eta'}$$ and $${}_{i+1} \tcA_i=H_*^{\tbI \rtimes \C^\times}({}_{i+1}\tcR_i)$$
\begin{theorem}
\label{thm: one step}
There are natural isomorphisms of graded bimodules
$${}_{i+1}\cA_i\cong \eee {}_{i+1}\tcA_i \eee\cong \eee_- {}_i\tcA_i \eee$$
\end{theorem}
\begin{proof}
Let ${}_{i+1}\tcR_i\hookrightarrow {}_{i}\tcR_i$ be the natural inclusion. 
$(g,s)\in {}_i \tcR_i$ belongs to ${}_{i+1}\tcR_i$ exactly when $gs\in \Lie(\bI)\cap t\fg(\cO)$, or in other words when it is in the kernel of the map 
$${}_{i}\tcR_i\to [\fb/B]$$ sending $(g,s)$ to $gs$ mod $t$.
We also have the cartesian square
\begin{equation}
\begin{tikzcd}
{}_{i}\tcR_i\arrow[r,"\varphi'"]\arrow[d, "\pi"'] & \left[ \tilde{\fg}/G \right]=[\fb/B] \arrow[d, "\pi'"]\\
{}_i \cR_i \arrow[r,"\varphi"] &  \left[ \fg/G \right]
\end{tikzcd}
\end{equation}
similar to \eqref{eq: Springer diagram}. Now, note that by finite-dimensional Springer theory $$H_*^{\tbI\rtimes \C^\times}({}_i\cR_i)={}_i\tcA_i \eee$$ and $$H_*^{\tG_\cO\rtimes \C^\times}({}_i \cR_i)=\eee {}_i\tcA_i\eee$$ Finally ${}_{i+1}\cA_i\cong \eee {}_{i+1}\tcA_i\eee$ and $$\eee_- {}_i\tcA_i\eee[2\dim G/B]\cong \eee {}_{i+1} \tcA_i \eee$$ again similarly to the proof of Lemma \ref{lem: spherical antispherical}.

\end{proof}

\subsection{Localization of the spherical algebra in the adjoint case}
\label{subsec:spherical adjoint localization}
We now analyze the $\Z$-algebra introduced in the previous section via localization to fixed points. In particular, we may deduce results about the associated graded of the Bruhat filtration for the convolution algebras, using an ''abelianization" procedure appearing e.g. in \cite{BDG}.
We should note that similar fixed-point analysis does not apply to the Springer action itself unless we are in a situation similar to \cite{OY,VVfd,GK}, but we are still able to deduce many results about the convolution action on general grounds in Section \ref{sec:bfnspringer}. 

We let $G$ and $N$ be arbitrary for now. Suppose $\bP=G(\cO)$. 
The spaces ${}_i \cR_j$ have natural closed embeddings to ${}_i \cR_j\hookrightarrow G(\cK)\times^{G(\cO)} t^jN(\cO)$.
Moreover, there is the embedding of the zero-section $$z: \Gr_G\hookrightarrow G(\cK)\times^{G(\cO)} t^jN(\cO)$$ and, if we denote by ${}_i \cR_{T}{}_j$ the space of triples constructed using $(T,N)$, an inclusion 
$$\iota: {}_i\cR_T{}_j\hookrightarrow {}_i \cR_j.$$ 
The latter map gives rise to an equivariant pushforward 
$$\iota_*: H_*^{T(\cO)}({}_i\cR_{T}{}_j)\to H_*^{G(\cO)}({}_i\cR_j)$$ (see \cite[Lemma 5.17 and Remark 5.23]{BFN}). The map $z$ for $G=T$ gives the maps ${}_i z_j^*$ from Theorem \ref{thm:abelianembedding}. With a choice of a maximal torus $T$, call the union of the roots of $G$ with respect to $T$ with the $T$-weights of the representation $N$ {\em generalized roots}. 

Then, similarly to \cite{BFN, BFN2} we have
\begin{proposition}
\label{prop:differenceembedding}
We have an embedding $${}_iz_j^*(\iota_*)^{-1}: {}_i \cA_j^{\hbar}\hookrightarrow \cA_{T,0}^{\hbar}[\hbar^{-1}, (\text{generalized roots}+m\hbar+n c)^{-1}| m,n\in \Z].$$ Note that this is {\em not} a ring homomorphism unless $i=j$, but a bimodule homomorphism, as in Theorem \ref{thm:abelianembedding}.
\end{proposition}

Let $\pi:{}_i\cR_j\to \Gr_G$ be the projection. We use the Cartan decomposition of the affine Grassmannian into $G(\cO)$ orbits:
$$
\Gr_G=\bigsqcup_{\lambda\in X_*^+}\Gr_G^{\lambda},\quad \Gr_G^{\lambda}=G(\cO)t^{\lambda}G(\cO)/G(\cO)
$$
The closures of these orbits will be denoted by $\Gr^{\leq \lambda}_G=\overline{\Gr_G^{\lambda}}$.
Then the subvariety ${}_i\cR_j^{\leq \lambda}:=\pi^{-1}(\Gr^{\leq \lambda}_G)$ gives rise to a class in equivariant Borel-Moore homology as in \cite[Section 2]{BFN}.

In particular, we have the following localization formula.
\begin{lemma}
\label{lem:locformula}
For a minuscule cocharacter 
$\lambda$, we have 
\begin{equation}
\label{eq: locformula}
{}_iz_j^*(\iota_*)^{-1}f\cap [{}_i\cR_j^{\leq \lambda}]=\sum_{\lambda'=w\lambda\in W\lambda}\frac{wf\times e(t^{\lambda'} t^j N(\cO)/t^{\lambda'} t^jN(\cO)\cap t^i N(\cO))}{e(T_{\lambda'}\Gr_G^\lambda)}u_{\lambda'}
\end{equation}
\end{lemma}
\begin{proof}
We are using Borel-Moore homology, so results  of Brion \cite{Brion} apply and the formula follows from Theorem \ref{thm:abelianembedding}. For the case $i=j=0$, see \cite[Proposition 6.6]{BFN}.

\end{proof}

If $\lambda$ is not minuscule, the corresponding Schubert variety is not smooth and there is no nice formula for $[{}_i\cR_j^{\leq \lambda}]$. To overcome this, we consider the Bruhat filtration on the Coulomb branch algebra $\cA$ defined by the classes of $[{}_i\cR_j^{\leq \lambda}][f]$. The corresponding Bruhat filtration on $\cA_T$  is given by the classes of $\mu$ such that $\mu$ is in a Weyl group orbit of a dominant coweight $\mu'$ with $\mu'\le \lambda$. In the sequel of the paper we will write $\gr \cA$ (resp $\gr \cA_T$) for the associated graded algebras with respect to the Bruhat filtrations. 

By \cite[Section 6.(i)]{BFN} the localization map $\iota_*: \cA_T\to \cA$ agrees with the respective Bruhat filtrations and one can define a map $\gr \iota_*:\gr \cA_T\to \gr \cA$. To save space, we will denote 
$$
(\gr \iota_*)^{-1}[{}_j\cR_i^{\lambda}][f]:=(\gr \iota_*)^{-1}\gr [{}_j\cR_i^{\lambda}][f].
$$
The right hand side of the equation \eqref{eq: locformula} yields the formula for the associated graded  with respect to the Bruhat filtration, see e.g. \cite[Eq. (6.3)]{BFN}. We first consider the case $G=\GL_n$.

\begin{theorem}
\label{thm:gr of bruhat}
Let $G=\GL_n$, then the following hold:

(a) For $c=\hbar=0$, arbitrary cocharacter $\lambda$ and a function $f(y)$ which is symmetric under the stabilizer of $\lambda$ we have the following:
$${}_j z_i^* (\gr \iota_*)^{-1}[{}_j\cR_i^{\lambda}][f]=\sum_{\lambda'\in W\lambda}f'\prod_{s\neq r}\frac{ (y_r-y_s)^{\max(\lambda'_r-\lambda'_s+i,j)}}{(y_r-y_s)^{\lambda'_r-\lambda'_s+i} (y_r-y_s)^{\max(\lambda'_r-\lambda'_s,0)}}u^{\lambda'}
$$
Here $f'$ is the image of $f$ under any permutation in $W$ which sends $\lambda$ to $\lambda'$. If $\lambda$ is minuscule, the formula is exact without taking associated graded. 

(b) For general $c,\hbar$ we have
$$
{}_j z_i^* (\gr \iota_*)^{-1}[{}_j\cR_i^{\lambda}][f]=\sum_{\lambda'\in W\lambda}f'\frac{\prod_{\lambda'_r-\lambda'_s+i<j}\prod_{\ell=0}^{j-(\lambda'_r-\lambda'_s+i)-1}(y_r-y_s+(\lambda'_r-\lambda'_s+i+\ell)\hbar+c)}{\prod_{s\neq r}\prod_{\ell=0}^{\max(\lambda'_r-\lambda'_s,0)}(y_s-y_r+\ell\hbar)}u^{\lambda'}
$$
where the notations are as above.
\end{theorem}

\begin{proof}
We compute the right hand side in the equation \eqref{eq: locformula}.
If $\lambda=(\lambda_1,\ldots,\lambda_n)\in X_*(T)\subset X_*(\GL_n)$ we get
$$t^{\lambda}.t^iN(\cO)=\begin{pmatrix}
t^iN(\cO) & t^{\lambda_1-\lambda_2+i} N(\cO)& t^{\lambda_1-\lambda_3+i}N(\cO) & \cdots & t^{\lambda_1-\lambda_n+i}N(\cO)\\
t^{\lambda_2-\lambda_1+i}N(\cO) & t^{i}N(\cO) & t^{\lambda_2-\lambda_3+i}N(\cO) & \cdots & t^{\lambda_2-\lambda_n+i}N(\cO) \\
t^{\lambda_3-\lambda_1+i}N(\cO) & t^{\lambda_3-\lambda_2 +i}N(\cO) & t^{i}N(\cO) & \cdots & t^{\lambda_3-\lambda_n+i}N(\cO)\\
\vdots & \vdots & \vdots & \ddots & \vdots \\
t^{\lambda_n-\lambda_1+i}N(\cO) & t^{\lambda_n-\lambda_2+i}N(\cO) & t^{\lambda_n-\lambda_3+i}N(\cO) & \cdots & t^i N(\cO)
\end{pmatrix} $$
Whence we compute the Euler class at $c=\hbar=0$:  
$$e(t^{\lambda}t^iN(\cO)/t^{\lambda}t^iN(\cO)\cap t^jN(\cO))=
\prod_{s\neq r} \frac{(y_r-y_s)^{\max(\lambda_r-\lambda_s+i,j)}}{(y_r-y_s)^{\lambda_r-\lambda_s+i}}
$$ 
For general $c,\hbar$ the factors with $\lambda'_r-\lambda'_s+i\ge j$ still contribute 1, and the formula for the Euler class reads as
$$
\prod_{\lambda'_r-\lambda'_s+i<j}\prod_{\ell=0}^{j-(\lambda'_r-\lambda'_s+i)-1}(y_r-y_s+(\lambda'_r-\lambda'_s+i+\ell)\hbar+c).
$$
It is well known that the tangent space $T_{\lambda}\Gr^{\lambda}$ is naturally identified with
$\frac{N_\cO}{N_\cO\cap t^\lambda.N_\cO}$, from which we get 
$$e(T_{\lambda}\Gr_G^\lambda)=\prod_{s\neq r} \prod_{\ell=0}^{\max(\lambda_r-\lambda_s,0)}(y_s-y_r+\ell\hbar).$$
\end{proof}

\begin{remark}
The above formula makes sense even if $f$ is not symmetric with respect to the stabilizer of $\lambda$. In this case, we first symmetrize with respect to the stabilizer of $\lambda$ and then symmetrize with respect to the whole group $W=S_n$. 
\end{remark}

For a general group $G$ and $N=\Ad$, we write can write the formula as follows.
\begin{proposition}
\label{prop:noncomm localization for any group}
For arbitrary $G$ and arbitrary coweight $\lambda$ we have:
$${}_j z_i^* (\gr \iota_*)^{-1}[{}_j \cR_i^{\leq \lambda}][f]=\sum_{\lambda'\in W\lambda}
f'\frac{\prod_{\alpha(\lambda')+i<j}\prod_{\ell=0}^{i-\alpha(\lambda')-j-1}(y_{\alpha}+(\alpha(\lambda')+j+\ell)\hbar+c)}{\prod_{\alpha\in \Phi} \prod_{\ell=0}^{\max(0, \alpha(\lambda'))-1}(y_{\alpha}+\ell\hbar)}u_{\lambda'}$$
\end{proposition}
\begin{proof}
The proof of Theorem \ref{thm:gr of bruhat} is naturally adopted to arbitrary root systems. 
For the interest of space, we leave the details for the reader.
\end{proof}

\begin{lemma}
For $G=\GL_n$ and $N=\Ad$ and the minuscule coweight $\omega_m=(1,\ldots,1,0,\ldots,0)$ and $i\geq j$, we have 
\begin{eqnarray}
z^*\iota_*^{-1}[{}_i\cR_i^{\leq \omega_m}]=&\sum_{I\subset[n], |I|=m} \prod_{r\in I, s\notin I} \frac{y_s-y_r+(i-1)\hbar+c}{y_r-y_s}u^{I}\\
z^*\iota_*^{-1}[{}_{i+1}\cR_{i}^{\leq \omega_m}]=&\sum_{I\subset[n], |I|=m} 
\frac{(\prod_{r\in I,s\notin I} (y_s-y_r+(i-1)\hbar+c)(y_s-y_r+i\hbar+c))(\prod_{r\in I, s\in I \text{ or } r\notin I, s\notin I} y_r-y_s+i\hbar+c)}{\prod_{r\in I, s\notin I} y_r-y_s}u^{I}\\
z^*\iota_*^{-1}[{}_{j}\cR_{i}^{\leq \omega_m}]=&\prod_{k=i+1}^{j-1}\prod_{r,s}(y_r-y_s+k\hbar+c)\cdot z^*\iota_*^{-1}[{}_{i+1}\cR_{i}^{\leq \omega_m}]
\end{eqnarray}
for $j\ge i+2$.
\end{lemma}
\begin{proof}
This is a direct application of Theorem \ref{thm:gr of bruhat}, recall that since $\omega_m$ is minuscule we do not need to take associated graded.  A symmetrization of $\omega_m$ leads to a weight 
$$
\lambda'=(\lambda'_1,\ldots,\lambda'_n),\ \lambda'_r=\begin{cases}
1 & \text{if}\ r\in I\\
0 & \text{otherwise}
\end{cases}
$$
for some $m$-element subset $I$. In this case $\max(\lambda'_r-\lambda'_s,0)$ equals 1 if $r\in I,s\notin I$ and 0 otherwise which gives the denominator. For the numerator, we observe
$$
\lambda'_r-\lambda'_s+i=\begin{cases}
i-1 & \text{if}\ r\notin I,s\in I\\
i & \text{if}\ r\in I,s\in I\ \text{or}\ r\notin I,s\notin I\\
i+1 & \text{if}\ r\in I,s\notin I\\
\end{cases}
$$
and the result follows.
\end{proof}

More generally, we have the formula for arbitrary $G$ and minuscule $\lambda$.
\begin{lemma}
\label{lemma:minusculelocalizationformula}
For $N=\Ad$ and $\lambda$ minuscule,
\begin{eqnarray}
z^*\iota_*^{-1}[{}_i\cR_i^{\leq \lambda}][f]=&\sum_{\lambda'=w\lambda\in W\lambda} wf \times \prod_{\alpha(\lambda')=1}\frac{-y_{\alpha}+(i-1)\hbar+c}{y_{\alpha}}u^{\lambda'}
\\
z^*\iota_*^{-1}[{}_{i+1}\cR_{i}^{\leq \lambda}][f] =&\sum_{\lambda'=w\lambda\in W\lambda} wf \times 
\frac{\left(\prod_{\substack{ \alpha(\lambda')=1}}(y_{\alpha}+(i-1)\hbar+c)(
y_{\alpha}+i\hbar+c)\right)\left(\prod_{\substack{  \alpha(\lambda')=0}}y_{\alpha}+i
\hbar+c\right)}{y_{\alpha}}u^{\lambda'}\\
z^*\iota_*^{-1}[{}_{j}\cR_{i}^{\leq \lambda}]=&\prod_{k=i+1}^{j-1}\prod_{\alpha\in \Phi}(y_{\alpha}+k\hbar+c)\cdot z^*\iota_*^{-1}[{}_{i+1}\cR_{i}^{\leq \lambda}]
\end{eqnarray}
\end{lemma}

\begin{lemma}
\label{lem: epsilon simplified}
Let $\varepsilon(x)=\max(x+i,j)-(x+i)-\max(x,0)$, then
$$
\varepsilon(x)+\varepsilon(-x)=\begin{cases}
j-i & \text{if}\ |x|\ge |j-i|,\\
(j-i)+|j-i|-|x| & \text{if}\ |x|\le |j-i|.
\end{cases}
$$

\end{lemma}

\begin{proof}
Let us first prove that for arbitrary $x,d$ one has
\begin{equation}
\label{eq: x and d}
\max(x,d)+\max(-x,d)=d+\max(|x|,|d|)=\begin{cases}
d+|x| & \text{if}\ |x|\ge |d|,\\
d+|d| & \text{if}\ |x|\le |d|.\\
\end{cases}
\end{equation}
Clearly, $\max(x,d)+\max(-x,d)=\max(|x|,d)+\max(-|x|,d)$. For $d\ge 0$ we get $\max(-|x|,d)=d$ and \eqref{eq: x and d} is clear.
For $d<0$ we can rewrite  
$$
\max(|x|,d)+\max(-|x|,d)=|x|-\min(|x|,|d|)=d+\max(|x|,|d|).
$$
Now we can prove lemma, by letting $d=j-i$. Note that $\max(x+i,j)=i+\max(x,j-i)$, therefore
$$
\varepsilon(x)+\varepsilon(-x)=\max(x+i,j)-(x+i)-\max(x,0)+\max(-x+i,j)-(-x+i)-\max(-x,0)=
$$
$$
\max(x,j-i)-\max(x,0)+\max(-x,j-i)-\max(-x,0).
$$
Now we can use  \eqref{eq: x and d} with $d=j-i$.
\end{proof}

\begin{corollary}
\label{cor: explicit formula for R lambda}
Let $G=\GL_n$. At $c=\hbar=0$ we get
$$
{}_j z_i^* (\gr\iota_*)^{-1}[{}_j\cR_i^{\lambda}][f]=\pm \Sym\left(f\cdot \Delta^{j-i} \prod_{r<s,|\lambda_r-\lambda_s|<|j-i|}(y_r-y_s)^{|j-i|-|\lambda_r-\lambda_s|}u^{\lambda}\right)
$$
\end{corollary}

\begin{proof}
Consider a pair $r<s$. In the right hand side of Theorem \ref{thm:gr of bruhat} we get
$$
(y_r-y_s)^{\varepsilon(\lambda'_r-\lambda'_s)}(y_s-y_r)^{\varepsilon(\lambda'_s-\lambda'_r)}=\pm (y_r-y_s)^{\varepsilon(\lambda'_r-\lambda'_s)+\varepsilon(\lambda'_s-\lambda'_r)}.
$$
By Lemma \ref{lem: epsilon simplified}, the result follows.
\end{proof}

\begin{example}
\label{ex: one step}
Again for $G=\GL_n$, assume that $j=i+1$, then  at $c=\hbar=0$ we get
$$
{}_{i+1} z_i^* (\gr\iota_*)^{-1}[{}_{i+1}\cR_i^{\lambda}][f]=\pm \Delta \Alt\left( f\cdot \prod_{r<s,\lambda_r=\lambda_s}(y_r-y_s) u^{\lambda}\right)
$$
\end{example}

For arbitrary groups and $\hbar=c=0$ we get a similar formula.
\begin{corollary}
\label{thm:localization for any group any weight}
For arbitrary $G$ and $\lambda$ we have 
$${}_j z_i(\gr \iota_*)^{-1} [{}_j \cR_i^\lambda][f]=
\Sym_W\left(f\Delta^{j-i}\prod_{\alpha\in\Phi^+,|\alpha(\lambda)|<|j-i|}(y_{\alpha})^{|j-i|-|\alpha(\lambda)|}u^\lambda\right)$$
\end{corollary}
\begin{proof}
The proof follows from setting $\hbar=c=0$ in Theorem \ref{prop:noncomm localization for any group} in exactly the same way as Theorem \ref{thm:gr of bruhat} and Corollary \ref{cor: explicit formula for R lambda}.
\end{proof}

\subsection{Proof of Theorem \ref{thm:E}}
\label{sec: proof E formula}
In this section we restate and prove Theorem  \ref{thm:E} using Coulomb geometric realization of the trigonometric Cherednik algebra $\HH_{c,\hbar}$ and its spherical subalgebra as  Coulomb branch algebras. 

\begin{theorem}

a) Let $\lambda$ be a minuscule dominant coweight, respectively \(X^\lambda\in \widetilde{W}\), then 
$$
E_{\lambda,c}:=\eee X^{\lambda} \eee = \eee \prod_{\alpha(\lambda)=1} \frac{y_{\alpha}-c}{y_{\alpha}} u^{\lambda} \eee
$$ 
and 
$$
F_{\lambda,c}:=\eee X^{-\lambda} \eee = \eee \prod_{\alpha(\lambda)=1} \frac{y_{\alpha}+c}{y_{\alpha}} u^{-\lambda} \eee
$$
Here $u^{\lambda}$ is the translation by $\hbar\lambda$. 

b) For an arbitrary dominant coweight $\lambda$ we get
\begin{equation}
\label{eq: E lambda c 2}
E_{\lambda,c}=\eee \prod_{\alpha(\lambda)>0}\prod_{\ell=0}^{\alpha(\lambda)-1} \frac{(y_{\alpha}+\ell\hbar-c)}{(y_{\alpha}+\ell \hbar)}u^{\lambda}\eee +\mathrm{lower\ order\ terms}.
\end{equation}
If $\lambda$ is minimal in Bruhat order, the formula \eqref{eq: E lambda c 2} for $E_{\lambda,c}$  is exact.

c) More generally, for dominant coweight $\lambda$ which is minimal in the Bruhat order, and a polynomial $f(y)$ we define 
\begin{equation}
\label{eq: E lambda c f 2}
E_{\lambda,c}[f]=\eee f(y)\prod_{\alpha(\lambda)>0}\prod_{\ell=0}^{\alpha(\lambda)-1} \frac{(y_{\alpha}+\ell\hbar-c)}{(y_{\alpha}+\ell \hbar)}u^{\lambda}\eee,\ F_{\lambda,c}[f]=\eee  f(y)\prod_{\alpha(\lambda)>0}\prod_{\ell=0}^{\alpha(\lambda)-1} \frac{(y_{\alpha}+\ell\hbar+c)}{(y_{\alpha}+\ell \hbar)}u^{-\lambda} \eee.
\end{equation}
The spherical subalgebra $\eee \HH_c \eee$ is generated by such $E_{\lambda,c}[f]$ and $F_{\lambda,c}[f]$.
\end{theorem}

\begin{proof}[Proof of Theorem \ref{thm:E}]

By  Proposition \ref{prop:noncomm localization for any group} and Corollary \ref{cor: sperical coulomb is daha} we can write $\tcA_{G,\bI}^{\hbar}\simeq \HH_{c,\hbar}$ and
 $\cA_{G,G(\cO)}^{\hbar}={}_{0}\cA^{\hbar}_{0}\simeq \eee\HH_{c,\hbar}\eee$.

To prove Theorem \ref{thm:E}(b),  we identify the operator $E_{\lambda,c}\in \eee \HH_{c,\hbar}\eee$  (up to lower order terms in the Bruhat filtration) with the class 
${}z^* (\gr \iota_*)^{-1}[ {}_0\cR^{\leq \lambda}_0]\in {}_{0}\cA^{\hbar}_{0}$.
The localization formula in Proposition \ref{prop:noncomm localization for any group} for $f=1$ and $i=j=0$ then implies the desired formula for $E_{\lambda,c}$, again up to lower order terms. If $\lambda$ is minimal in the Bruhat order, there are no lower order terms and the formula is exact.

The proof of Theorem \ref{thm:E}(c) is similar and again follows from Proposition \ref{prop:noncomm localization for any group} for $i=j=0$. The last claim about generation follows from
\cite[Proposition 6.8]{BFN} which states that the classes ${}z^* (\gr \iota_*)^{-1}[ {}_0\cR^{\leq \lambda}_0][f]$ for dominant  $\lambda$ which are minimal in the Bruhat order and arbitrary $f$, and opposite classes corresponding to 
$-\lambda$, generate the Coulomb branch algebra ${}_{0}\cA^{\hbar}_{0}$. 

In part Theorem \ref{thm:E}(a) the coweight $\lambda$ is minuscule, and the formula from (b) simplifies. First, we have either $\alpha(\lambda)=1$ or $\alpha(\lambda)=0$ for all $\alpha$, so $\ell=0$ for all nontrivial factors. Second, 
$\lambda$ is minimal in Bruhat order, so there cannot be any lower order terms and the formula is exact. 

\end{proof}

\begin{remark}
Alternatively, one can prove Theorem \ref{thm:E}(b) using a tedious but explicit computation in the affine Hecke algebra, see \cite[Example 5.4,Theorem 5.9]{Kirillov} and \cite[Proposition 5.13]{KN}. The translation $X^{\lambda}$ can be written as a product of elements of $\Omega$ and simple reflections $\sigma_i$ in some order, and one can control the leading term of each factor. This leads to a formula for the leading term for $X^{\lambda}$, and its symmetrization.
\end{remark}

\begin{remark}
In this paper, we use both  results from Section \ref{sec: adjoint coulomb} in Section  \ref{sec:cherednik}, and results from Section  \ref{sec:cherednik} in Section \ref{sec: adjoint coulomb}.  We would like to assure a cautious reader that such a  nonlinear narrative does not lead to circular reasoning in any proofs.

 In the proof of Proposition \ref{prop:noncomm localization for any group} and Corollary \ref{cor: sperical coulomb is daha}  we only use the definitions of $\HH_c$ and $\eee\HH_c\eee$ and the construction of their polynomial representations by difference operators.  Proposition \ref{prop:noncomm localization for any group} then directly follows (similarly to Theorem \ref{thm:gr of bruhat}) from localization in the Coulomb branch algebra and does not use any results from Section \ref{sec:cherednik}.
In the proof of Theorem \ref{thm:E} we then use  Proposition \ref{prop:noncomm localization for any group}.

The remainder  of Section  \ref{sec:cherednik} heavily uses Theorem \ref{thm:E},  but no further results from Section  \ref{sec: adjoint coulomb}. Finally, the remainder of Section  \ref{sec: adjoint coulomb} heavily uses the results from Section  \ref{sec:cherednik}. 
\end{remark}

\subsection{Factorization of bimodules}

\begin{lemma}
\label{lem: split weight}
Suppose that $\lambda$ is an arbitrary integral coweight for $\GL_n$ and $d>0$. Then there exist $d$ coweights $\mu^{(0)},\ldots,\mu^{(d-1)}$ such that $\mu^{(0)}+\ldots+\mu^{(d-1)}=\lambda$ and for all $i$ and $j$   the following holds:

1) If $|\lambda_i-\lambda_j|<d$ then
$$
d-|\lambda_i-\lambda_j|=\sum_{k,\mu^{(k)}_i=\mu^{(k)}_j}1.
$$

2) If $|\lambda_i-\lambda_j|>d$ then $\mu^{(k)}_i\neq \mu^{(k)}_j$ for all $k$.
\end{lemma}

\begin{proof}
We define $\mu^{(k)}$ by ``dividing $\lambda$ by $d$ with remainder''. More precisely, let $\lambda_i=dq_i+r_i$ where $0\le r_i<d$.
We define
$$
\mu^{(k)}_i=\begin{cases}
q_i+1 & \text{for}\ k<r_i\\
q_i & \text{for}\ k\ge r_i\\ 
\end{cases}
$$
Clearly,  $\mu^{(0)}_i+\ldots+\mu^{(d-1)}_i=\lambda_i$.
Without loss of generality, we can assume that $\lambda_j\ge \lambda_i$. We have the following cases:

1) $\lambda_j=dq_i+r_j, r_j\ge r_i$. In this case $\mu^{(k)}_i=\mu^{(k)}_j$ for $k<r_i$ and $k\ge r_j$,
so 
$$
\sum_{k,\mu^{(k)}_i=\mu^{(k)}_j}1=d-(r_j-r_i)=d-(\lambda_j-\lambda_i).
$$

2) $\lambda_j=d(q_i+1)+r_j, r_j<r_i$. In this case $\mu^{(k)}_i=\mu^{(k)}_j$ for $r_j\le k<r_i$ and 
$$
\sum_{k,\mu^{(k)}_i=\mu^{(k)}_j}1=r_i-r_j=d-(\lambda_j-\lambda_i).
$$

3) If  $\lambda_j>d(q_i+1)+r_j$ then $\mu^{(k)}_j\ge q_i+2$ for $k<r_i$ and $\mu^{(k)}_j\ge q_i+1$ for $k\ge r_i$, so
 $\mu^{(k)}_i\neq \mu^{(k)}_j$ for all $k$.

\end{proof}

\begin{example}
Suppose that $d=2$, then we split $\lambda=\mu^{(0)}+\mu^{(1)}$ as follows. If $\lambda_i=2k$ is even, we set $\mu_i^{(0)}=\mu_i^{(1)}=k$; if $\lambda_i=2k+1$ is odd, we set $\mu_i^{(0)}=k+1$ and $\mu_i^{(1)}=k$. Clearly, if $\lambda_i=\lambda_j$ then 
both $\mu_i^{(0)}=\mu_j^{(0)}$ and $\mu_i^{(1)}=\mu_j^{(1)}$. If $|\lambda_i-\lambda_j|=1$, it is not hard to see that exactly one of equations $\mu_i^{(0)}=\mu_j^{(0)}$ and $\mu_i^{(1)}=\mu_j^{(1)}$ holds.
\end{example}

\begin{corollary}
\label{cor: surjective}
Suppose that $G=\GL_n$, $j-i=d$  and $c=\hbar=0$. For an arbitrary coweight $\lambda$ and $\mu^{(k)}$ as in Lemma \ref{lem: split weight} we have
\begin{equation}
\label{eq:sur}
{}_{j} z_i^* (\gr \iota_*)^{-1}[{}_{j}\cR_i^{\lambda}]=\pm \prod_{k=0}^{d-1} {}_{i+k+1} z_{i+k}^*(\gr \iota_*)^{-1}\left[{}_{i+k+1}\cR_{i+k}^{\mu^{(k)}}\right]+\mathrm{lower\ order\ terms}.
\end{equation}
\end{corollary}

\begin{proof}
By Lemma \ref{lem: split weight} we get 
$$
\prod_{r<s,|\lambda_r-\lambda_s|<d}(y_r-y_s)^{d-|\lambda_r-\lambda_s|}=\prod_{k}\prod_{r<s,\mu^{(k)}_r=\mu^{(k)}_s}(y_r-y_s).
$$
By Corollary \ref{cor: explicit formula for R lambda}, the left hand side of Eq. \eqref{eq:sur} is a symmetric polynomial with leading term (in the dominance order on the $u_\lambda$) 
$$
\pm \Delta^d \prod_{r<s,|\lambda_r-\lambda_s|<d}(y_r-y_s)^{d-|\lambda_r-\lambda_s|}u^{\sort(\lambda)}
$$
while the right hand side is a product of $d$  symmetric polynomials with leading terms
$$
\pm \Delta \prod_{r<s,\mu^{(k)}_r=\mu^{(k)}_s}(y_r-y_s)u^{\sort(\mu^{(k)})}
$$
It is easy to see that in the above construction $\sort(\lambda)=\sort(\mu^{(0)})+\ldots+\sort(\mu^{(d-1)})$, so the result follows. 
\end{proof}
\begin{remark}
It seems reasonable to conjecture analogs of Lemma \ref{lem: split weight} and Corollary \ref{cor: surjective} for other groups, at least for simply laced groups.  This would have the consequence that the isomorphism constructed in Theorem \ref{thm:Z-algebra-iso} would hold for other groups, showing for instance that the global sections of the line bundle we construct equal $A_G^d$. Since the line bundle is not expected to be ample outside the simply laced case (see \cite{LosevDeformations}) we do not expect the result to hold in general.

We also note that the combinatorics appearing in the Lemma are closely related to the {\em root-system chip-firing} of \cite{chipfiring}. It would be interesting to make the connection more precise. The second author thanks Pavel Galashin for correspondence regarding this point.
\end{remark}

\subsection{Proof of Theorem \ref{thm: properties of coulomb algebra}}
\label{sec: proof properties coulomb algebra}

In this section we restate and prove Theorem \ref{thm: properties of coulomb algebra}.

\begin{theorem}
The algebras  ${}_{0}\cA^{\hbar=0}_{d}$ have the following properties:

a) For $d=0$, we have ${}_{0}\cA^{\hbar=0}_{0}=\C[T^\vee \times \ft]^W$.

b) For $d=1$, we have ${}_{0}\cA^{\hbar=0}_{1}=A$.

c) For all $d$ the module  ${}_{0}\cA^{\hbar=0}_{d}$ is a free $\C[\ft]^W$-submodule of $\eee_d\C[T^\vee \times \ft]$.

d) For $G=\GL_n$, we have ${}_{0}\cA^{\hbar=0}_{d}=({}_{0}\cA^{\hbar=0}_{1})^d=A^d$ for all $d$.
\end{theorem}

 Recall that by Proposition \ref{prop:differenceembedding} and \cite[Lemma 5.17.]{BFN}, we have an embedding 
\begin{equation}
\label{eq: localization at hbar 0}
{}_{0}\cA^{\hbar=0}_{d}(G)\hookrightarrow  {}_{0}\cA^{\hbar=0}_{d}(T)=\C[T^\vee \times \ft].
\end{equation}

\begin{lemma}
\label{lem: localization factors through centralizer}
Let $t\in \ft$, and let $Z_G(t)$ denote the centralizer of $t$ in $G$. Then the inclusion \eqref{eq: localization at hbar 0} factors through ${}_{0}\cA^{\hbar=0}_{d}(Z_G(t))$:
$$
\begin{tikzcd}
{}_{0}\cA^{\hbar=0}_{d}(G) \arrow[bend left]{rr} \arrow[dotted]{r} & {}_{0}\cA^{\hbar=0}_{d}(Z_G(t)) \arrow{r}&  {}_{0}\cA^{\hbar=0}_{d}(T) 
\end{tikzcd}
$$
\end{lemma}
\begin{proof}
We use the following general result \cite[Lemma 5.1]{BFN2}. Suppose that $N$ is a representation of $G$, consider the one-parameter subgroup $E_t=\exp(\mathbb{R}t)$ in $G$. Then 
$$
\cR_{G,N}^{E_t}\simeq \cR_{Z_G(t),N^t}.
$$
When $N=\fg$ we get $N^t=\Lie Z_G(t)$, and therefore
$$
{}_i\cR_G{}_j^{E_t}\simeq {}_i\cR_{Z_G(t)}{}_j
$$
(we drop $N$ from the notation).
The chain of inclusions of fixed points 
$$
{}_i\cR_G{}_j^{T}\simeq {}_i\cR_T{}_j\hookrightarrow {}_i\cR_G{}_j^{E_t}\simeq {}_i\cR_{Z_G(t)}{}_j\hookrightarrow {}_i\cR_G{}_j
$$
induces a  commutative diagram in equivariant BM homology, where the horizontal maps are injective and the vertical maps are isomorphisms:
$$
\begin{tikzcd}
H_*^{G}({}_i\cR_G{}_j) \arrow{r}\arrow{d} & H_*^{Z_G(t)}({}_i\cR_{Z_G(t)}{}_j) \arrow{r}\arrow{d} & H_*^T({}_i\cR_T{}_j) \arrow{d}\\
{}_{0}\cA^{\hbar=0}_{d}(G) \arrow{r}& {}_{0}\cA^{\hbar=0}_{d}(Z_G(T)) \arrow{r} & {}_{0}\cA^{\hbar=0}_{d}(T).
\end{tikzcd}
$$
\end{proof}

\begin{proof}[Proof of Theorem \ref{thm: properties of coulomb algebra}]
a) We regard $\C[\ft]^W$ as the equivariant cohomology of a point, and ${}_{0}\cA^{\hbar=0}_{d}$ is realized as the equivariant Borel-Moore homology of a certain space which admits an affine paving by Bruhat cells. Therefore it is equivariantly formal and its equivariant cohomology is a free module over $H^{*}_G(\pt)$. The embedding to $\C[T^\vee \times \ft]$ is realized by the inclusion to equivariant Borel-Moore homology of the fixed point set.  
That we land in the $\eee_d$-isotypic component follows from the fact that the localization is defined using $T$-equivariant cohomology and to pass to $G$-equivariant cohomology we take $W$-invariants. See for instance \cite[Remark 5.23]{BFN}.

b) This is a specialization of Corollary \ref{cor: sperical coulomb is daha} at $\hbar=0$.

c) By part (a) we have inclusion ${}_{0}\cA^{\hbar=0}_{1}\subset \eee_{-}\C[T^\vee \times \ft]=A$.  By Theorem  \ref{thm: one step} (specialized at $\hbar=0$) this is an isomorphism.

d) This is a specialization of Theorem \ref{thm:Z-algebra-iso} at $\hbar=0$. 
\end{proof}

\subsection{The geometric $\Z$-algebra for the adjoint representation}

We are ready to prove the main result of this section.

\begin{theorem}\label{thm:Z-algebra-iso}
When $G=\GL_n$, the $\Z$-algebras $\cA$ and $\cB$ are isomorphic. For general $G$, there is an injection $\cA\hookrightarrow \cB$ inducing ${}_{j-1} \cA_j\cong {}_{j-1}\cB_j$ and ${}_j\cA_j\cong {}_j\cB_j$.
\end{theorem}

\begin{proof}
We need to prove the following facts:
\begin{itemize}
\item[(a)] ${}_i\cA_i\cong {}_i\cB_i$ as algebras
\item[(b)] ${}_i\cA_{i+1}\cong {}_i\cB_{i+1}$ as bimodules over ${}_i\cA_i$ (resp.  ${}_i\cB_i$) and  ${}_{i+1}\cA_{i+1}$ (resp.  ${}_{i+1}\cB_{i+1}$)
\item[(c)] ${}_i\cA_{i+1}\cdots {}_{j-1}\cA_{j}\hookrightarrow {}_i\cA_{j}$ and this is an isomorphism for $G=\GL_n$. 
Note that ${}_i \cB_j\cong {}_i\cB_{i+1}\cdots {}_{j-1}\cB_j$ by definition.
\end{itemize}
Part (a) follows from Theorem \ref{thm:iwahoricoulombisdaha}. Part (b) follows from Theorem \ref{thm: one step}. 

In type $A$, it is instructive to review what part (b) says in order to prove part (c). We can compute the bases in the associated graded spaces on both sides: $\gr {}_{i}\cB_{i+1}=A$ is the space of antisymmetric polynomials in $\C[x_1^\pm,\ldots, x_n^\pm,y_1,\ldots, y_n]$ and by Lemma \ref{lem: Haiman dets}(a)
it has a vector space basis $\Delta_{S}$ parametrized by all $n$-element subsets of $\Z_{\ge 0}\times \Z$. On the other hand, in $\gr {}_{i}\cA_{i+1}$
we have a basis $\gr [{}_{j}\cR_i^{\lambda}][f]$ parametrized by a weight $\lambda$ and a function $f$. By Example \ref{ex: one step}
and Lemma \ref{lem: Haiman dets}(c) these can be explicitly identified by setting $f$ to be the product of Schur polynomials. Finally, having a filtered homomorphism inducing an isomorphism on associated graded spaces gives an isomorphism.

Let us prove part (c) for $G$ arbitrary. By Corollary \ref{cor: bimodules} the convolution product gives a natural map  
\begin{equation}
\label{eq: convolution iso}
{}_i\cA_{i+1}\bigotimes_{{}_{i+1}\cA_{i+1}}\cdots \bigotimes_{{}_{j-1}\cA_{j-1}}{}_{j-1}\cA_{j}\to {}_i\cA_{j},
\end{equation}
To check that \eqref{eq: convolution iso} is injective, it is sufficient to check that it becomes an isomorphism after localization in the multiplicative set generated by  $\{y_\alpha+n\hbar+mc|\alpha\in\Phi, m,n\in \Z\}$ which we get from \cite[Remark 3.24]{BFN}.
Finally, we need to prove that it is an isomorphism for $G=\GL_n$. To prove that it is surjective, we first consider the commutative limit $\hbar=c=0$ and take the associated graded with respect to the Bruhat filtration. Then surjectivity follows from Corollary \ref{cor: surjective}.

Next, we use parts (a) and (b) of the theorem to rewrite the left hand side of \eqref{eq: convolution iso} as
$$
{}_i\cB_{i+1}\bigotimes_{{}_{i+1}\cB_{i+1}}\cdots \bigotimes_{{}_{j-1}\cB_{j-1}}{}_{j-1}\cB_{j}={}_i\cB_{j}.
$$
By Theorem \ref{th: Gordon Stafford} this is free over $\C[y_1,\ldots,y_n]^{S_n}$. Since the space ${}_i\cR_{j+1}$ is equivariantly formal, the bimodule ${}_i\cA_{j+1}$ is free over $\C[\hbar]\otimes \C[y_1,\ldots,y_n]^{S_n}$ as well.
Therefore \eqref{eq: convolution iso} is surjective for general $c,\hbar$ . 

\end{proof}

\begin{remark}
For $G=\GL_n$, Simental \cite{Simental} classified Harish-Chandra bimodules for the {\em rational} Cherednik algebra and proved that the shift bimodule is the unique Harish-Chandra bimodule which sends polynomial representation to the polynomial representation. In particular, this implies an analogue of Theorem \ref{thm:Z-algebra-iso} for the rational Cherednik algebra.

It would be interesting to know if the methods of \cite{Simental} can be generalized to the trigonometric case to give an alternate proof of  Theorem \ref{thm:Z-algebra-iso}.
\end{remark}

Combining the above result with the $\Proj$ construction  we get 
\begin{corollary}
When $G=\GL_n$, the graded algebra $\bigoplus_{i} {}_{j-i} \cA_j^{\hbar=c=0}$ is naturally isomorphic to the homogeneous coordinate ring of
$\Hilb^n(\C\times \C^{\times})$ for any $j$.
\end{corollary}
\begin{proof}
Specialize the above theorem for $c=\hbar=0$ and use Theorem \ref{th: Gordon Stafford}.
\end{proof}
\begin{remark}
One should also compare this to the results in  \cite{BFN3} which essentially show ${}_i \cA_j\cong \cO(j-i)$ in the case $G=\GL_n, N=\Ad\oplus V^{\ell}$ for $\ell\geq 1$, using factorization and results about the Hilbert schemes on $A_{\ell-1}$-resolutions.
\end{remark}

\subsection{A flag $\Z$-algebra}

In this section, we sketch to what extent the construction of ${}_i \cA_j$ extends to the flag level, i.e. when we replace $\fg(\cO)$ by the standard Iwahori subalgebra and $G(\cO)$ by the Iwahori subgroup $\bI$. This gives a Springer-theoretic construction of the ''one-step" shift bimodule ${}_{i-1}\cA_i$. On the level of affine Springer fibers, the analogous geometry is discussed in Section \ref{sec:geometricaction}.

Let $ev_0^{-1}(\fb)=\fri$ be the standard Iwahori subalgebra. Then $ev_0^{-1}(0)\subset \fri$. 
Consider the sequence of subalgebras 
$$
\fg_0:=\fg(\cO)\supset \fg_{1/2}:=\fri \supset \fg_1:=t\fg(\cO) \supset t\fri \supset t^2\fg(\cO)\supset t^2\fri \supset \cdots
$$

Then, as $k$-vector spaces (but importantly, {\em not} as Lie algebras) we have the subquotients 
$$
\fg_{i}/\fg_{i+1/2}\cong \begin{cases}
\fb, & i \in \Z+1/2\\
\fn_-, & i \in \Z
\end{cases}
$$
\begin{example}
For $G=SL_2$ we have 
$$
\begin{pmatrix}
\cO & \cO \\
\fm & \cO
\end{pmatrix}
\supset\begin{pmatrix}
\fm & \fm \\
\fm & \fm
\end{pmatrix}
\supset\begin{pmatrix}
\fm & \fm \\
\fm^2 & \fm
\end{pmatrix}
\supset \cdots 
$$
so that 
$\fg_{1/2}/\fg_1\cong \fb,
\; \fg_1/\fg_{3/2} \cong \fn_-$.
\end{example}
\subsubsection{Bimodules}
Consider now the spaces 
$${}_j\tcR_i:=\{[g,s]\in G(\cK)\times^{\bI} \fg_i | gs\in \fg_j\}, \; i \in \frac{1}{2}\Z$$ 
And $${}_j \cR_i :=\{[g,s]\in G(\cK)\times^{G(\cO)}\fg_i|gs\in\fg_j\}, \; i \in \Z.$$
\begin{proposition}
Let $[\fb/B]\xrightarrow{r} [\fg/G]$ be the Grothendieck-Springer resolution.
Then for $i\in \Z$ we have the cartesian diagrams
$$\begin{tikzcd}
 {{}_{i+1/2}\tcR_{i+1/2}} \arrow[d]\arrow[r,"\psi"]& \left[\fb/B\right] \arrow[d]\\
 {{}_{i}\cR_{i}} \arrow[r,"\phi"']& \left[\fg/G\right] 
\end{tikzcd}$$
\end{proposition}
In particular,
$${}_{i}\cA_i\cong \bfe\; {}_{i+1/2}\cA_{i+1/2}\;\bfe$$ by Springer theory.
On the other hand, it is easy to see that $$
\phi^{-1}(0)=\{[g,s]\in G(\cK)\times^{G(\cO)}\fg_i|gs\in t\fg_i\}={}_{i+1}\cR_i
$$ and
$$(r\circ\psi)^{-1}(0)=\{[g,s]\in G(\cK)\times^{\bI}\fg_{i+1/2}|gs\in \fg_{i+1}\}={}_{i+1}\tcR_{i+1/2}.$$
In particular, 
$${}_{i+1}\cA_i=\bfe_- \; {}_{i+1/2}\tcA_{i+1/2} \Delta  \; \bfe=\bfe\; {}_{i+1}\tcA_{i+1/2}\;\bfe.$$

From what we have before, 
$${}_i \tcA_i$$ is the trigonometric Cherednik algebra when $i\in 1/2+\Z$. The algebra for $i\in \Z$ is \textbf{not} a Cherednik algebra, but indeed a matrix algebra over the spherical Cherednik algebra, as in \cite{Weekes, Webster}.

\section{Generalized affine Springer theory}
\label{sec:bfnspringer}
\subsection{Generalized affine Springer fibers}

In this section we generalize the Springer action from \cite{GK,HKW} to the line operators discussed above. 
Let $\bP$ be a parahoric subgroup, $N_\bP$ be a lattice in $N(\cK)$ stable under $\bP$. Given this data, denote $\eta=(\bP,N_\bP)$.
Further, suppose that $$1\to G\to \tG\to G_F\to 1$$ is an extension of algebraic groups and that $\tbP$ is a parahoric subgroup of $\tG(\cK)$ which fits into an extension
$$1\to \bP\to \tbP\to G_F(\cO)\to 1$$ so that $\tbP\cap G(\cK)=\bP$. Let $\tG^\cO_\cK$ be the preimage in $\tG(\cK)$ of $G_F(\cO)$.

\begin{definition}
\label{def:gasf}
Let $v\in N(\cK)$. The {\em generalized affine Springer fiber} of $v$ is the ind-subscheme of $\Fl_\bP
$ defined by $${}_{\eta}M_v:=\{g\in \Fl_\bP|g^{-1}.v\in N_\bP\}.$$
\end{definition}

\begin{remark}
Recall that if $N=\Ad$, $\bP$ is a fixed parahoric subgroup, $N_\bP=\Lie(\bP)$, ${}_{\eta}M_\gamma=\Sp_\gamma^{\bP}$, the classical affine Springer fiber for $\bP$.
\end{remark}

\begin{definition}
The {\em orbital variety} of $\gamma\in N_\cK$ and $\eta=(\bP,N_{\bP})$ is 
$${}_{\eta}\O_{\gamma} := \tG(\cK) . \gamma\cap N_{\bP}.$$
\end{definition}
\begin{remark}
Note that the orbital variety only depends on the lattice $N_\bP$. However, we always use it in conjunction with $\bP$, explaining the slightly redundant notation with $\eta$.
\end{remark}
In particular, we have 
\begin{lemma}[\cite{GK,HKW}]
\label{lem:stackquotient}
We have an isomorphism of stacks
$[{}_\eta\O_{\gamma}/\bP] \cong [G_\gamma\backslash {}_\eta M_\gamma]$, where
$G_\gamma$ is the stabilizer of $\gamma$ in $G(\cK)$.
\end{lemma}
\begin{proof}
Let ${}_\eta X_\gamma=\{g\in \tG_{\cK}^{\cO}\rtimes \C^\times| g^{-1}.\gamma\in N_{\bP}\}$. Then there are maps 
\begin{center}
\begin{tikzcd}
 & {}_\eta X_\gamma \arrow[dl] \arrow[dr] & \\
 {}_\eta M_\gamma & & {}_\eta \O_\gamma
\end{tikzcd}
\end{center}
making ${}_\eta X_\gamma$ a $\bP$-torsor over ${}_\eta M_\gamma$ and a $G_\gamma$-torsor over ${}_\eta \O_\gamma$.
\end{proof}
\begin{lemma}
\label{lem:stackquotient2}
Suppose ${}_\eta M_\gamma$ is finite-dimensional over $\C$. Note that the $\cK$-group $G_\gamma$ admits a N\'eron model $\mathcal{G}_\gamma/\cO$. We have that 

(i) ${}_\eta M_\gamma$ admits a $G_\gamma$-equivariant
dualizing complex $\omega_{{}_\eta M_\gamma}$.

(ii)  For $K_\gamma\subseteq \mathcal{G}_\gamma(\cO)$, the equivariant Borel-Moore homology $H_*^{K_\gamma}({}_\eta M_\gamma)=:H^{\bP}_*(K_\gamma\backslash {}_\eta X_{\gamma}/\bP)$ is well-defined. Here ${}_\eta X_\gamma=\{g\in \tG_{\cK}^{\cO}\rtimes \C^\times| g^{-1}.\gamma\in N_{\bP}\}$ as before.
\end{lemma}
\begin{proof}
(i) is clear by finite-dimensionality and Lemma \ref{lem:stackquotient}. For (ii), we can approximate ${}_\eta M_\gamma$ by finite-type $K_\gamma$-stable varieties, and then take the colimit.
\end{proof}

\begin{example}
Suppose $N=\Ad$, and $\gamma$ is split regular semisimple. Then $G_\gamma$ is a split maximal torus in $G(\cK)$, in fact the loop group of a split maximal torus $T\subset G$. The equivariant BM homology $H_*^{T}(\Sp_\gamma)$ is studied in the next section.
\end{example}
\begin{remark}
We may extend the setup to the flavor-deformed equivariant version by considering ${}_\eta\widetilde{\O}_\gamma:=\tG_{\cK}^\cO\rtimes \C^\times . \gamma$ and and its quotient by $\tbP\rtimes \C^\times$ instead. We leave constructions of these extended notions to the reader or refer to \cite{GK}.
\end{remark}
Suppose that $\tG=G\times \C^{\times}$, so $G_F=\C^{\times}$ is the flavor group above. The group $\tG$ acts on $N$ via $v\mapsto hg^{-1}\gamma g$.
We denote the resulting GASF $$\underline{M}_\gamma=\tG_{\cK} .\gamma \cap N_\bP/\tbP$$
Since $[h]=[t^d]\in X^*(\C^{\times})=\Z$ we see that
$\underline{M}_\gamma$ splits into components
$$M_{t^{-d}\gamma}\cong\{g\in \Gr_G|t^dg^{-1}\gamma g=g^{-1}t^d\gamma g \in \fg(\cO)\}=\Sp_{t^{-d}\gamma}.$$ We recognize this to be the affine Springer fiber of $t^{-d}\gamma$ in $\Fl_\bP$, or in other words that  $$\underline{M}_\gamma=\bigsqcup_{d\in \Z} M_{t^d\gamma}.$$

\subsection{Springer action from the Coulomb perspective}

We now define the Coulomb branch version of  the Springer action, and in particular the geometric action of our $\Z$-algebra. 
Let $K_\gamma$ be as before and $X_\gamma$ as in Lemma \ref{lem:stackquotient2}.
\begin{theorem}
\label{thm:action}
The following convolution diagram defines naturally associative maps $$H_*^{\tbP\rtimes \C^\times_{\rot}}({}_{\eta}\cR_{\eta'})\otimes H_*^{K_\gamma}({}_{\eta'} M_\gamma)\to H_*^{K_\gamma}({}_{\eta}M_\gamma).$$
   \begin{center}
   \begin{tikzcd}
{}_{\eta}\cR_{\eta'} \times {}_{\eta'} \O_\gamma \arrow[d, "i"] & p^{-1}({}_{\eta}\cR_{\eta'} \times {}_{\eta'} \O_\gamma)\arrow[d]\arrow[l,"p"']\arrow[r,"q"]& q(p^{-1}({}_{\eta}\cR_{\eta'} \times {}_{\eta'}\O_\gamma))\arrow[r,"m"] & {}_{\eta}\O_\gamma \\
G(\cK)\times_{\bP'}N_{\bP'} \times {}_{\eta'} \O_\gamma & G(\cK)\times {}_{\eta'} \O_\gamma\arrow[l]
\end{tikzcd}
\end{center}

Here $p: (g,s)\mapsto ([g,s],s)$, $q$ is the quotient by the diagonal action of $\bP'$ and $m$ is the map sending $[g,s]\mapsto g.s$. 
\end{theorem}
\begin{proof}
Compare this to the proof of
\cite[Theorem 4.5.]{GK}.
We explain the maps induced in BM homology by $p,q,m$. 
Consider the space $${}_\eta\cP_{\eta'}:=\{(g,s)\in \tG^{\cO}_{\cK}\rtimes \C^\times_{\rot}\times N_{\bP'}| g^{-1}.s\in N_\bP\}$$
and note there are maps $\pi_1:{}_\eta\cP_{\eta'}\to N_{\bP'}$ and $(g,s)\mapsto s$ and $\pi_2: {}_\eta\cP_{\eta'}\to N_\bP$ given by $(g,s)\mapsto g^{-1}.s$. Then consider the $G_\gamma/K_\gamma$-torsor $\pi: K_\gamma\backslash {}_{\eta'} X_\gamma\to N_\bP$ and define $${}_{\eta'} \cF_{\gamma,K_\gamma}:=\pi_*\omega_{K_\gamma\backslash {}_{\eta'}X_\gamma}[-2\dim \tbP+2\dim K_\gamma],$$ which is an object in the $\tbP \rtimes \C^\times_{\rot}$-equivariant derived category of $N_\bP$ supported on ${}_{\eta'}\O_\gamma$.

First of all, we have the ''pull-back with support" map $p^*$ (see \cite[Section 3(ii)]{BFN}) 
	\begin{align}
	p^*:\, &H_{\tbP \rtimes \C^\times_{\rot} \times \tbP' \rtimes \C^\times_{\rot}}^{-*}({}_\eta\cR_{\eta'}\times N_\bP, (\omega_{{}_\eta\cR_{\eta'}})\boxtimes ({}_{\eta'} \cF_{\gamma,K_\gamma}))\nonumber\\ 
	& =H_*^{\tbP\rtimes \C^\times_{\rot}}({}_\eta \cR_{\eta'})\otimes H_*^{\tbP'\rtimes \C^\times_{\rot}}(K_\gamma\backslash {}_{\eta'}X_{\gamma})\to H^*_{\tbP\rtimes \C^\times_{\rot}\times \tbP'\rtimes \C^\times_{rot}}({}_\eta\cP_{\eta'}, \pi_1^!({}_{\eta'} \cF_{\gamma,K_\gamma})).
	\end{align}

Further, we have a map
	$\pi_1^!{}_{\eta'} \cF_{\gamma,K_\gamma}\to \pi^!_2{}_{\eta} \cF_{\gamma,K_\gamma}$ and since $\pi_2=m\circ q$, we get 
	$$q_*: H^*_{\tbP\rtimes \C^\times_{\rot}\times \tbP'\rtimes \C^\times_{\rot}}({}_\eta\cP_{\eta'}, \pi_1^!{}_{\eta'} \cF_{\gamma,K_\gamma})\to H^*_{\tbP\rtimes \C^\times_{\rot}}(q({}_\eta\cP_{\eta'}), m^!{}_{\eta} \cF_{\gamma,K_\gamma})$$
	
	Finally, $m$ is (ind-)proper because its fibers are closed subvarieties of a partial affine flag variety, so that using the adjunction $m_!m^!\to \id$ we get a map
	$$(m\circ q)_*: H_*^{\tbP \rtimes \C^\times_{\rot}}(q({}_\eta\cP_{\eta'}),m^!{}_\eta \cF_{\gamma,,K_\gamma})\to H_*^{\tbP \rtimes \C^\times_{\rot}}(K_\gamma\backslash {}_\eta X_\gamma)=H_*^{K_\gamma}({}_\eta M_\gamma)$$
	
See \cite{GK, HKW} for more details, for example the proof of associativity of the maps.

\end{proof}

While the convolution diagram in Theorem \ref{thm:action} is rather abstract and the maps in Borel-Moore homology involved are defined sheaf-theoretically, in easy cases it is possible to analyze the action as follows. Similar to \cite[Section 4.2]{GK}, we define the {\em Hecke stack} for $\gamma,\eta$ which has $\C$-points 
$${}_\eta \cR_{\eta'}^\gamma(\C)=\{(s_2,g,s_1)\in {}_{\eta}\O_\gamma\times G(\cK) \times {}_{\eta'}\O_\gamma|g.s_1=s_2\}/\bP.$$ Here the quotient is by the action 
$h.(s_2,g,s_1)=(s_2,gh^{-1},hs_1)$. There is a natural Schubert stratification of ${}_{\eta}\cR^\gamma_{\eta'}$ inherited from ${}_{\eta}\cR_{\eta}$, where 
$${}_{\eta}\cR^\gamma_{\eta'}\hookrightarrow {}_{\eta}\cR_{\eta'}$$ via $[s_2,g,s_1]\mapsto [g,s_1]$. Similarly we have maps 

\begin{equation}
\label{eq:pushpull}
\begin{tikzcd}
 & {}_{\eta}\cR^\gamma_{\eta'} \arrow[dl] \arrow[dr]& \\
{}_{\eta} M_\gamma & & {}_{\eta'} \O_\gamma
\end{tikzcd}
\end{equation}
We will use this diagram later  on in our computation of certain shift maps.

In the adjoint case, the name ''Springer action" is warranted, as it coincides with the action defined by Yun, Oblomkov-Yun \cite{OY} (and Varagnolo-Vasserot \cite{VV2}):
\begin{theorem}
\label{thm:actionscoincide}
Let $N=\Ad$ and $\eta=(\bI,N_\bI)=\eta'$. Then the action of the algebra $\tcA_{G,\bI}$ on $H_*^{K_\gamma}(M_\gamma)$ defined by Theorem \ref{thm:action} coincides with the one defined in \cite{OY} on the equivariant homology of affine Springer fibers, under the isomorphism of Theorem \ref{thm:iwahoricoulombisdaha}
\end{theorem}
\begin{proof}
Theorem \ref{thm:iwahoricoulombisdaha} shows that the Springer action of simple reflections in the affine Weyl group is the same. The equivariant parameters act by Chern classes of line bundles on the affine flag variety, and that the relations are the same follows from Theorem \ref{thm:iwahoricoulombisdaha}.
\end{proof}

The novel feature in allowing arbitary $\eta, \eta'$ shows the following. 
\begin{corollary}
\label{cor:springer modules}
 The convolution product in Theorem \ref{thm:action} gives maps
$${}_{j} \cA_{i}^{\hbar} \otimes H_*^{K_\gamma}(M_{t^{i}\gamma})\to H_*^{K_\gamma}(M_{t^{j}\gamma})$$ 
that naturally assemble into an action of the $\Z$-algebra $B^{\hbar}=\bigoplus_{i\leq j} {}_j\cA_i^{\hbar}$
Moreover, the action in Theorem \ref{thm:convolution} not including loop rotation, i.e. setting $\hbar=0$,
defines maps
$$H_*^{G(\cO)}(\cR_{\tG,N}^d)\times H_*^{K_\gamma}(M_\gamma^{d'})\to
H_*^{K_\gamma}(M_\gamma^{d+d'}).$$
\end{corollary}

In particular, the above corollary gives a geometric construction of ''column vector" modules for our geometric $\Z$-algebra $B=\bigoplus_{i\leq j\leq 0} {}_i\cA_j$. 

\subsection{The adjoint case}

In the case $N=\Ad$, the construction of these affine Springer theoretic modules is also closely related to the construction of a commutative (partial) resolution as in the previous sections, in the following way. For $G=\GL_n$, by the results of \cite{BFN3}, the commutative limit ${}_{i+d} \cA_{i}^{\hbar=0}$ is identified with the sections of $\cO(d)$ on the Hilbert scheme of points $\Hilb^n(\C^\times\times \C)$. In particular, 
$$\Proj \bigoplus_{d\geq 0} {}_{i+d} \cA_{i}^{\hbar=0}\cong \Hilb^n(\C^\times \times \C).$$
In general, we have the following proposition.
\begin{proposition}
Let $\hbar=0$. Then every $\gamma\in N_\cK$ and a choice of $K_{\gamma}$ as in Theorem \ref{thm:action} gives a quasicoherent sheaf $\cF_\gamma^{K_{\gamma}}$ on the partial resolution of the Coulomb branch given by 
$$\Proj \bigoplus_{d\geq 0} {}_{i+d} \cA_{i}^{\hbar=0}.$$
\end{proposition}
\begin{corollary}
When $G=\GL_n$, the above construction gives a quasicoherent sheaf on $\Hilb^n(\C^\times\times \C)$ associated to $\gamma\in N_\cK$.
\end{corollary}

It is in general hard to compute which sheaf this is. In all examples we have checked, this should be a coherent sheaf for regular semisimple elements. This is a conjecture that we discuss in Section \ref{sec:finitegeneration}.

Recall that Lemma \ref{lem:support} tells us that the support of $\cF_\gamma^{K_{\gamma}}$ is determined by $K_\gamma$, i.e. the equivariance we consider, as well as the splitting type of $\gamma$.

\begin{example}
When $\gamma=zt^d$ as above, we get twists of the "Procesi bundle" as shown in \cite{Kivinen} which are supported everywhere. See Proposition \ref{prop:procesi} for the precise statement.

When $\gamma$ is elliptic, these sheaves are supported on the punctual Hilbert scheme over $(1,0)\in \C^\times\times \C$.
\end{example}
\subsection{Action on representations}
\label{sec:geometricaction}
Let $\tSp_\gamma'=\{g\bI|g\gamma g^{-1}\in t\fg(\cO)\}$. Consider the Springer module $M_\gamma=H_*^{K_\gamma}(\Sp_\gamma)$. 
Then we have a natural map $M_\gamma \to M_{t\gamma}$ given by inclusion. There are also maps
$$\widetilde{M}_\gamma\to \widetilde{M}_{t\gamma}' \to \widetilde{M}_{t\gamma}$$ given by inclusion, or in other words convolving with the identity or point class in ${}_{1/2} \tcA_1 $ and then by the point class in ${}_{1} \tcA_{3/2}$.

There are also natural Gysin maps $\widetilde{M}_{t\gamma}\to \widetilde{M}_{t\gamma}' \to \widetilde{M}_\gamma$. The first one is codimension zero, and the second one codimension $\dim G/B$. Composed, on the level of equivariant parameters, these look like the \textbf{square} of the Vandermonde determinant $\Delta$. The issue arises from the normalization in the embedding to difference-reflection operators, in which the point class of ${}_{1}\tcA_{3/2}$ is naturally identified with $\Delta$ (and a Cartan part), but acts as the identity on components (effectively, it cuts down the tangent spaces of the components by $\Delta$).

Note also we have \textbf{two} projections $$\tSp_{t\gamma}'\to \Sp_\gamma, \; \tSp_\gamma\to \Sp_\gamma$$ the first one of which is a fibration, and the second one has fibers which are usual Springer fibers (stratified fibration).
Effectively, the two line operators (for $(1/2, 1)$ and $(1, 3/2)$) in the flags get squeezed down to a single one on the spherical level (the one for $(0,1)$).

\begin{lemma}
\label{lem: dim shift}
We have $\dim \Sp_{t\gamma}=\dim \Sp_{\gamma}+\dim G/B.$
\end{lemma}

\begin{proof}
By a result of Bezrukavnikov \cite{Bezrukavnikov} the dimension of  $\Sp_{\gamma}$ is given by  
\begin{equation}
\dim \Sp_{\gamma}=\frac{1}{2}\left(\nu_{\mathrm{ad}}(\gamma)-\mathrm{rk}(\fg)+\dim(\fh^w)\right),
\end{equation}
where  $w\in W$ is such that $Z(\gamma)$ is of type $w$, $\fh^w$ denotes the $w$-invariants in $\fh$ and 
$\nu_{\mathrm{ad}}(\gamma)$ is the valuation of 
$$
\det\left(\mathrm{ad} \gamma: \fg(K)/Z(\gamma)\rightarrow \fg(K)/Z(\gamma)\right)
$$
It is easy to see that changing $\gamma$ to $t\gamma$ does not change $w$. The matrix $\mathrm{ad} \gamma$ is multiplied by $t$ which changes $\nu_{\mathrm{ad}}(\gamma)$ by $|\Phi|=2\dim G/B$, and the result follows.
\end{proof}

\begin{lemma}
\label{lem: preimage sp}
Let $\pi: \tSp_{t\gamma}\to \Sp_{t\gamma}$ be the natural projection. If $\gamma$ is elliptic then 
$\pi^{-1}(\Sp_{\gamma})$ is an irreducible component of $\tSp_{t\gamma}$. More generally, if $\gamma$ is regular semisimple, and $C$ is an irreducible component of $\Sp_\gamma$, then $\pi^{-1}(C)$ is an irreducible component of $\tSp_{t\gamma}$.
\end{lemma}

\begin{proof}
By the proof of Lemma \ref{lem: spherical antispherical} the projection $\pi^{-1}(\Sp_{\gamma})\to \Sp_{\gamma}$ has fibers $G/B$ at every point. Since $\gamma$ is elliptic, $\Sp_{\gamma}$ is irreducible and hence $\pi^{-1}(\Sp_{\gamma})$ is irreducible as well. If $C$ is as in the statement of the Lemma, the same proof goes through. 

Furthermore, all components  of $\widetilde{\Sp_{t\gamma}}$ have dimension $\dim \Sp_{t\gamma}$.
By Lemma \ref{lem: dim shift} we have 
$$
\dim \pi^{-1}(\Sp_{\gamma})=\dim \Sp_{\gamma}+\dim G/B=\dim \Sp_{t\gamma},
$$
and the result follows. 
\end{proof}

Lemma \ref{lem: preimage sp} allows us to construct an important correspondence between $\Sp_{\gamma}$ and $\Sp_{t\gamma}$.  By the work of Tsai \cite{Tsai}, there are $W$ many irreducible components up to the centralizer action in $\tSp_{t\gamma}$. Furthermore, we expect the following:

\begin{conjecture}(\cite[Conjecture 8.6]{Tsai})
\label{conj: regular}
The Springer action on $H_*(\tSp_{\gamma})$ yields a regular representation in top-dimensional homology spanned by the classes of the irreducible components.
\end{conjecture}

 In particular, for $\gamma$ elliptic there is a distinguished component $\pi^{-1}(\Sp_{\gamma})$ and another component biregular to $\Sp_{t\gamma}$, and (assuming Conjecture \ref{conj: regular}) one can define a correspondence in Borel-Moore homology sending the former to the latter (for example, the symmetrizer $\eee$ would suffice). More generally, fix a component $C$ as above and note that the lattice part of the centralizer of $\gamma$ acts transitively on the set of these components \cite{KL88}. Now the class of $\pi^{-1}(C)$ can either be sent to the class of any lattice translate of $\pi^{-1}(C)$, or by the symmetrizer in the finite Weyl group to a one-dimensional $W$-invariant subspace of the BM homology of $\Sp_{t\gamma}$.

This leads to the following:

\begin{proposition}\label{prop:ell_shift}
Assume $\gamma$ is elliptic and \(G\) is simply-connected, and assume Conjecture  \ref{conj: regular} holds for $\gamma$. Consider the correspondence $\eee[\pi^{-1}(\Sp_{\gamma})]*-$ between $\Sp_{\gamma}$ and $\Sp_{t\gamma}$. The action of this correspondence in homology corresponds to the action of some class in ${}_{i+1}\cR_{i}$ as in Theorem \ref{thm:action}, which sends the fundamental class of $\Sp_{\gamma}$ to the fundamental class of $\Sp_{t\gamma}$.
\end{proposition}
\begin{proof}
  Lets construct a cycle \(\Gamma'\) in \({}_{i+1}\cR_{i}^{\gamma}\) such that correspondence \eqref{eq:pushpull} with the class \(\Gamma'\) sends
  \([\Sp_{\gamma}]\) to \([\Sp_{t\gamma}]\). First, we define \(\Gamma\subset {}_{i+1}\cR_{i}^{\gamma} \) as the lift of the graph of the embedding of \(\Sp_{\gamma}\) into
  \(\Sp_{t\gamma}\). The lift \(\Gamma\) is defined as the locus of triples \((s_2,g,s_1)\in {}_{i+1}\cR_i^{\gamma}\) such that \(G(\cO). s_1=G(\cO). s_2\). 

  Let \(\eta=(G(\cO),t^i\fg(\cO))\), \(\eta'=(G(\cO),t^{i+1}\fg(\cO))\) and
  \(\tilde{\eta}=(\mathbf{I},t^i\fg(\cO))\). In particular, \({}_{i+1}\cR_i^\gamma={}_\eta\cR_{\eta'}^\gamma\)
  and on the homology of the  fibers of the projection \(\tilde{\pi}:
  {}_\eta\cR_{\tilde{\eta}}^\gamma\to {}_{i+1}\cR_i^\gamma\) there is an action of \(W\). The push-forward along the projection \(\tilde{\pi}\) is the projection onto the \(W\)-invariant part of the homology.

  Let \(\tilde{q}:{}_\eta\cR_{\tilde{\eta}}^\gamma\to {}_{\tilde{\eta}}\mathbb{O}_\gamma\) be map from the corresponding diagram \eqref{eq:pushpull}. The previous proposition implies that the map \(\tilde{q}\) restricted to \(\tilde{\Gamma}=\tilde{\pi}^{-1}(\Gamma)\)
  is dominant over one of irreducible component of \({}_{\tilde{\eta}}\mathbb{O}_\gamma\). By Conjecture  \ref{conj: regular} the set of irreducible components of
  \({}_{\tilde{\eta}}\mathbb{O}_\gamma\) is a regular representation of \(W\). %\fixme{We don't know this! Only that there are $|W|$ many, see \cite[Conjecture 8.6.]{Tsai}}.
 Thus there is \(w\in W\) such \(w.[\tilde{\Gamma}]\) projects dominantly onto \({}_{\eta'}\mathbb{O}_\gamma\).  The class
  \(\Gamma'=\tilde{\pi}_*(w.[\tilde{\Gamma}])\) satisfies required properties.

\end{proof}

\begin{corollary}
  Under the assumption of the previous proposition we have the relation between the fundamental classes:
 \[ [\Sp_{t^{i+1}\gamma}]\in {}_{-i-1}\cA_{-i} * [\Sp_{t^i\gamma}]\]
\end{corollary}

\section{Finite generation and examples}
\label{sec:finitegeneration}

\subsection{Finite generation conjecture}
\label{sec:finite-gener-conj}

As we saw in Section \ref{sec:bfnspringer}, in particular Theorem \ref{thm:action}, the space
$$
\bF_{\gamma}:=\bigoplus_{k=0}^{\infty} H_{*}(\Sp_{t^k\gamma})
$$
is a graded module over the graded algebra $\bigoplus_{d=0}^{\infty}{}_0 \cA_{d}$. Equivalently, $\bF_{\gamma}$ defines a quasi-coherent sheaf $\cF_{\gamma}$ on $\Proj \bigoplus_{d=0}^{\infty}{}_0 \cA_{d}$.

\begin{conjecture}
\label{conj: fin gen}
The module $\bF_{\gamma}$ is finitely generated and the sheaf $\cF_{\gamma}$ is coherent.
\end{conjecture}

Note that by Theorem \ref{thm: lattice action factors} the homology of $\Sp_{t^k\gamma}$ is finitely generated over ${}_0 \cA_{0}$. For $G=\GL_n$ the graded algebra  $\bigoplus_{d=0}^{\infty}{}_0 \cA_{d}=\bigoplus_{d=0}^{\infty}A^d$ is generated by ${}_0 \cA_{0}$ and ${}_0 \cA_{1}=A$,   so Conjecture \ref{conj: fin gen}
is equivalent to saying that for a given $\gamma$ there exists $k_0$ such that $\bF_{\gamma}$ is generated by 
$\bigoplus_{k=0}^{k_0} H_{*}(\Sp_{t^k\gamma})$ under the action of ${}_0 \cA_{0}$ and ${}_0 \cA_{1}$.

Below we prove the conjecture in some special cases.

\begin{theorem}
Conjecture \ref{conj: fin gen} holds for $G=\GL_n$ and $\gamma=\diag(s_0,\ldots,s_n)$ for $s_i\neq s_j$. 
\end{theorem}

\begin{proof}
This follows from Proposition \ref{prop:procesi} below.
\end{proof}

\begin{example}
Let $G=\GL_2$ and 
$$\gamma=\begin{pmatrix}
t & 0 \\
0 & -t
\end{pmatrix}.$$
Then $$\Sp_{t^d\gamma}\cong \bigsqcup_{\Z} C_{d+1}$$ where each $C_{d+1}$ is an infinite chain of $\P^{d+1}$, consecutive members of which intersect transversally along a $\P^d$. These $\P^d$ are Spaltenstein varieties of $d$-planes in $2d$-space stable under a nilpotent element with Jordan blocks of sizes $(d,d)$, motivically equivalent to projective spaces $\P^d$.
The inclusion maps are again embedding the chains into one another and they are regular embeddings (because they are effective Cartier divisors).

Note that the direct sum of homologies of these $\P^{d+1}$ surjects onto the homology of $\Sp_{t^d\gamma}$. Let us prove that the module $\bF_{\gamma}$ is generated by the homology of $\Sp_{\gamma}$ under the action of ${}_0 \cA_{0}$ and ${}_0 \cA_{1}$. 

Indeed, $\Sp_{\gamma}$ is just a discrete set of points in bijection with the affine Weyl group. Its homology is a free rank two module over the lattice of translations. 
\end{example}

	In  case of elliptic \(\gamma\) the finite generation conjecture follows from the stable of \({}_0 \cA_{0}\)-cyclicity of the homology of \(\Sp_{t^k\gamma}\):

\begin{proposition}\label{prop:fin-gen-ell}
  Let us assume that \(\gamma\) is elliptic, Conjecture  \ref{conj: regular} holds for $t^k\gamma$ and \([\Sp_{t^k\gamma}]\in H_{*}(\Sp_{t^k\gamma})\) is the fundamental class.
  If there exists \(N\)  such that  \(H_{*}(\Sp_{t^k\gamma})=\C[T^*T^{\vee}]^W\cdot [\Sp_{t^k\gamma}]\) for \(k\ge N\) then Conjecture~\ref{conj: fin gen} holds for
  \(\bF_\gamma\).
\end{proposition}
\begin{proof} The actions of \({}_0 \cA_{0}=\C[T^*T^{\vee}]^W\) and \({}_0 \cA_{1}\) on \(\bF_\gamma\) commute.
  Hence by Proposition~\ref{prop:ell_shift} the submodule \(\oplus_{k\ge N}H_{*}(\Sp_{t^k\gamma})\) is generated by \({}_0 \cA_{0}\) and \({}_0 \cA_{1}\) from \([\Sp_{t^N\gamma}]\).
  The module \(\oplus_{k<N} H_{*}(\Sp_{t^k\gamma})\) is finite dimensional.
\end{proof}

The subalgebra \(\C[\ft]^W\) is isomorphic to the cohomology ring \(H^*(\Gr_G)\), it acts on \(H_*(\Sp_\gamma)\) by cap product. It is natural to conjecture
that a stronger version of %the condition of 
the previous proposition is true for \(G=\PGL_n\).

\begin{conjecture}\label{conj:gen_pgl}
  Let \(\fg=\mathrm{Lie}(\PGL_n)\) and \(\gamma\in \fg(\cO)\) is an elliptic regular semisimple topologically nilpotent element. Then
  \[H_{*}(\Sp_\gamma)=H^*(\Gr_G)\cap [\Sp_\gamma].\]
\end{conjecture}

If \(G=\GL_n\) or \(G=\SL_n\) and \(\gamma\in \fg(\cO)\) is an elliptic element then \(\Sp_\gamma\) has many connected components and the group
\(\pi_0(G_\gamma)\) permutes the connected components. In the light of aforementioned Theorem~\ref{thm: lattice action factors} it is natural to propose

\begin{conjecture}\label{conj:gen_gl}  Let \(\fg=\mathrm{Lie}(\GL_n)\) or \(\fg=\mathrm{Lie}(\SL_n)\)  and \(\gamma\in \fg(\cO)\) is an elliptic regular semisimple topologically nilpotent element. Then
  \[H_{*}(\Sp_\gamma)=\C[T^*T^{\vee}]^W [\Sp_\gamma].\]
\end{conjecture}

\begin{remark}
  The conjecture is false outside of type \(A\) since there are examples of elliptic affine Springer fibers with homology of not of type \((p,p)\) \cite{KL88,OY}. Note however that $H_*(\Sp_\gamma)$ is always finitely generated under $\C[T^*T^{\vee}]^W$ by \cite{YunSph} (see also Lemma \ref{lem:support}).
\end{remark}

For the homogeneous elements Conjecture~\ref{conj:gen_gl} is known \cite{OY1} and one can deduce

\begin{theorem}\label{thm:mn}
Conjecture \ref{conj: fin gen} holds for $G=\GL_n$ and equivalued $\gamma_{m,n}$ with characteristic polynomial $x^m-y^n$, $\gcd(m,n)=1$.
\end{theorem}
\begin{proof}

  The affine Grassmanian \(\Gr_G\) has \(\pi_1(\GL_n)=\Z\) connected components \(\Gr_G=\Gr_G^0\times \Z\). Respectively, we have
  \(\Sp_\gamma=\Sp_\gamma^0\times \Z\).
  
  Observe that if $\gamma_{m,n}$ is equivalued  with characteristic polynomial $x^m-y^n$ then $t^k\gamma_{m,n}$ is equivalued with characteristic polynomial $x^{kn+m}-y^n$.
  The compactified Jacobian \(J_{m/n}\) of the one-point compactification of the planar curve \(\{x^m-y^n\}\) is irreducible and homeomorphic to \(\Sp_{\gamma_{m,n}}^0\) \cite{OY1}.
  Moreover, \(\Gr_G^0=\Gr_{\PGL_n}\) and the \(\Sp_{\gamma_{m,n}}^0\) is the corresponding Springer fiber.

  It is shown in \cite{OY1} that for  \(\Sp_{\gamma_{m,n}}^0\) Conjectures~  \ref{conj: regular} and \ref{conj:gen_pgl} are true.  The group \(\pi_0(G_{\gamma_{m,n}})=\Z\) acts transitively on the connected components of \(\Sp_\gamma\), hence Conjecture~\ref{conj:gen_gl} is true for \(\Sp_{\gamma_{m,n}}\). Thus the theorem follows from Proposition~\ref{prop:fin-gen-ell}.

\end{proof}

\begin{example}
For $G=\GL_2$ and $\gamma=\gamma_{1,2}$ we recover the $\Z$-algebra module from Example \ref{ex: module sl2}.
\end{example}

\subsection{Examples in type A}
\label{sec:examples}
In the case \(G=\GL_n\) the sheaf \(\cF_{\gamma}\) can be described in terms of geometry of \(\Hilb_n(\C^\times\times \C)\) for some homogeneous \(\gamma\).

\begin{proposition}
\label{prop:punctualexample}
Let $\gamma$ be homogeneous of slope $(kn+1)/n$. Then $\cF_\gamma$ is isomorphic to the restriction of $\cO(k)$ to the punctual Hilbert scheme at $(1,0)\in \C^\times\times \C$.
\end{proposition}
\begin{proof}
The localized equivariant homology $H_*^{\G_m}(\Sp_\gamma)$ affords the unique finite-dimensional representation of $\eee \HH_{-\frac{kn+1}{n}}\eee$ as was checked in \cite{OY, VVfd}. By Lemma \ref{lem:support}, $\cF_\gamma$ is supported on this punctual Hilbert scheme (i.e. the corresponding fiber of the Hilbert-Chow map). Completing our $\Z$-algebra at a neighborhood of the identity in $T^\vee$, we get a completion of the rational Cherednik algebra and the corresponding bimodules, for $\GL_n$ with parameters given by integral shifts of $(kn+1)/n$. The Gordon-Stafford construction then implies \cite{GS2} that the corresponding sheaf on the punctual Hilbert scheme coincides with $\cO(k)$.

\end{proof}
\begin{proposition}
\label{prop:procesi}
Let $\gamma$ be homogeneous of slope $k$, or more generally equivalued of valuation $k$. Then $\cF_\gamma^{T}$ is isomorphic to $\cP\otimes \cO(k)$ where $\cP$ is the Procesi sheaf on $\Hilb^n(\C^\times\times \C)$.
\end{proposition}

\begin{proof}
For equivalued $\gamma$ of valuation $k$ the main result of \cite{Kivinen} identifies the equivariant Borel-Moore homology $H^T_*(\Sp_{\gamma})$ with the space of global sections of $\cP\otimes \cO(k)$ on $\Hilb^n(\C^\times\times \C)$ as a module over the algebra of global functions 
$$
{}_0\cA_{0}=\C[T^*T^{\vee}]^W=\C[x_1^{\pm},\ldots,x_n^{\pm},y_1,\ldots,y_n]^{S_n},
$$
By the work of Haiman \cite{Haiman} we get: 
\begin{equation}
\label{eq: haiman}
H^0(\Hilb^n(\C^\times\times \C),\cP\otimes \cO(k))=\bigcap_{i\neq j}\langle 1-x_i/x_j,y_i-y_j\rangle ^k,\quad H^i(\Hilb^n(\C^\times\times \C),\cP\otimes \cO(k))=0,\quad i>0.
\end{equation}
By Theorem \ref{thm: properties of coulomb algebra} the graded algebra ${}_{0}\cA_{\bullet}$ is generated by the degree 1 component ${}_{0}\cA_{1}=A$, where 
$A$ is the space of antisymmetric polynomials in $\C[x_1^{\pm},\ldots,x_n^{\pm},y_1,\ldots,y_n]$. 

It is easy to see that by \eqref{eq: haiman} we have a correctly defined map
$$
A\otimes H^0(\Hilb^n(\C^\times\times \C),\cP\otimes \cO(k))\to H^0(\Hilb^n(\C^\times\times \C),\cP\otimes \cO(k+1)),
$$
and it follows from \cite{Kivinen} that this agrees with the convolution ${}_{0}\cA_{1}\otimes H^T_*(\Sp_{\gamma})\to H^T_*(\Sp_{t\gamma})$. This completes the proof.
\end{proof}

For general $\gamma$ elliptic of slope $\frac{m}{n}$, the situation is as follows. Whilst our construction gives a sheaf $\cF_\gamma$, which is coherent by Proposition \ref{prop:fin-gen-ell}, we do not know how to identify this sheaf on $\Hilb^n(\C^\times\times \C)$. Indeed, a variant of this problem already appears in \cite[Problem 5.5]{GS2}.

\subsection{Beyond type A}

For general $G$, both the computations of the cohomology of affine Springer fibers and the sheaves on $\widetilde{\Comm}_{G^{\vee},\fg^{\vee}}$ are very complicated. It would be for example interesting to compute the sheaf one gets from the Bernstein-Kazhdan example of \cite[Appendix]{KL88}. Nevertheless, we have the following analogue of Proposition \ref{prop:procesi} for general $G$.

\begin{theorem}
\label{thm: integral slope any type}
Let $G$ be arbitrary and  $\gamma$ equivalued of valuation $k$. Then we have the isomorphism of graded modules
$$
\bF_{\gamma}=\bigoplus_{j=0}^{\infty} H^T_*(\Sp_{t^j\gamma}) \simeq \bigoplus_{j=0}^{\infty} \bigcap_{\alpha\in\Phi^+}\langle 1-\alpha^\vee, y_\alpha\rangle^{k+j}
$$
over the graded algebra $\bigoplus {}_0 \cA(0)_d\simeq \bigoplus \eee_d \bigcap_{\alpha\in\Phi^+}\langle 1-\alpha^\vee, y_\alpha\rangle^{d}$. As a consequence, the  corresponding sheaves over $\Proj \bigoplus {}_0 \cA(0)_d =\widetilde{\Comm}_{G^{\vee}}$ are isomorphic as well.  
\end{theorem}

\begin{proof}
The proof is similar to Proposition \ref{prop:procesi}. By the main result of \cite{Kivinen} the isomorphism holds for each $j$ separately  on the level of modules over ${}_0 \cA(0)_0\cong \C[T^*T^{\vee}]^W$. The comparison of the action of ${}_0 \cA(0)_d$ follows from Theorem \ref{thm: coulomb is symbolic} and the constructions in \cite{Kivinen, GK}. More precisely, the result in \cite{Kivinen} identifies $\Delta^j H^T_*(\Sp_{t^j\gamma})$ with
$\bigcap_{\alpha\in\Phi^+}\langle 1-\alpha^\vee, y_\alpha\rangle^{k+j}$ 
inside $\C[T^*T^\vee]\cong H_*^T(\Gr_T)$ using GKM localization. The latter has a multiplication structure which coincides with convolution on the Coulomb branch for $T$ with zero matter. The fact that the convolution action for ${}_0 \cA(0)_0$ respects the localization is \cite[Proposition 4.15.]{GK}.
\end{proof}

Let \(G\) be quasisimple of adjoint type. Respectively, let \(cox\in W\) be the Coxeter element of the Weyl group of \(G\) and \(n\) be the order of \(cox\). For any \(m\) co-prime with \(n\) there is a regular semisimple element
\(\gamma_{m,n}\in \fg(\mathcal{O})\) which is homogeneous: \(\gamma_{m,n}(\lambda\cdot t)=\lambda^{m/n}\Ad_{g(\lambda)}\gamma_{m,n}(t) \), \(g(\lambda)\in G\). The element \(\gamma_{m.n}\) is unique up to
rescaling and conjugation, an explicit construction of \(\gamma_{m,n}\) can found for example in \cite{OY}. The element \(\gamma_{m,n}\) is equivalued of valuation
\(m/n\).

The stabilizer in $\tG_\cK\rtimes \G_m$ is given by \(L_{\gamma_{m,n}}=\mathbb{G}_m\), and it acts naturally on \(\Sp_{\gamma_{m,n}}\). It is shown in \cite{OY} that \(\dim \Sp_{\gamma_{m,n}}^{\mathbb{G}_m}=0\), the fixed points are isolated
and that the localized homology \(H_*^{\mathbb{G}_m}(\Sp_{\gamma_{m,n}})\otimes \C(\hbar)\) is generated by tautological classes \(H^*_{\mathbb{G}_m}(\Gr_G)\) from
the fundamental class \([\Sp_{\gamma_{m,n}}]\). We expect the generation statement in the non-equivariant setting:
\begin{conjecture}\label{conj:gen_general}
Let \(G\), \(\mathrm{Lie}(\fg)\),\(\gamma_{m,n}\in \fg(\cO)\) are as above, then \[H_{*}(\Sp_{\gamma_{m,n}})=H^*(\Gr_G)\cap [\Sp_{\gamma_{m,n}}].\]
\end{conjecture}
Note also that this ''Coxeter case" gives the so called spherical simple modules of the trigonometric DAHA, as first observed in \cite{VVfd}. More generally, the slopes with so called regular elliptic denominators yield (spherical and other) finite-dimensional modules of the trigonometric DAHA \cite{VVfd,OY}. Since $\gamma$ elliptic implies $t\gamma$ elliptic, one sees that the tensor products by the shift bimodules ${}_{i-1}\cB_i$ send finite-dimensional modules to finite-dimensional modules, which one could also deduce from the theory of shift functors for trigonometric DAHA like in \cite{BEG}. As far as the authors are aware, this theory is still undeveloped (but see \cite{WLiu} for some progress), but would potentially give insight on the $m=1, n=h$ case of Proposition \ref{prop:punctualexample} for other groups.

\end{document}